%% file: main.tex
\documentclass[a4paper,10pt]{amsart}

\usepackage[utf8]{inputenc}
\usepackage[T1]{fontenc}        

\usepackage{hyperref}
\usepackage{amsmath}
\usepackage{amsfonts}
\usepackage{amssymb}
\usepackage{amsthm}
\usepackage{graphicx}
\usepackage{geometry}
\usepackage{float}
\usepackage{dsfont}
\usepackage{mathrsfs}
\usepackage{subfigure}
\usepackage{ulem}

\usepackage{CJKutf8} 
\usepackage{tikz}
\usepackage{pgfplots}
\usetikzlibrary{pgfplots.groupplots}
\usetikzlibrary{matrix}     
\usetikzlibrary{arrows, automata}

\newtheorem{lemma}{Lemma}
\newtheorem{theorem}{Theorem}
\newtheorem{prop}{Proposition}
\newtheorem{definition}{Definition}
\newtheorem{rem}{Remark}

\newtheorem{corollary}{Corollary}

\renewcommand{\paragraph}[1]{\noindent \textit{#1}}

\numberwithin{equation}{section}

\newcommand{\R}{\mathbb R}
\renewcommand{\I}{\mathcal{I}}

\newcommand{\RR}{\mathbb{R}}

\newcommand{\MM}{\mathcal{M}}
\newcommand{\WW}{\mathcal{W}}

\DeclareMathOperator*{\argmin}{arg\,min}

\title[Modeling knowledge Evolution]{Convergence of knowledge in a cultural evolution model with population structure, random social learning and credibility biases}

\author{Sylvain Billiard}%
\address{\flushleft
  \underline{Sylvain Billiard}\\[.3em]
  Univ. Lille, CNRS, UMR 8198 - Evo-Eco-Paleo, F-59000 Lille, France  \\[.2em]
  Email: \texttt{sylvain.billiard@univ-lille.fr}}

\author{Maxime Derex}
\address{\flushleft
  \underline{Maxime Derex}\\[.3em]
  Institute for Advanced Study in Toulouse, CNRS, UMR 5314\\[.2em]
  F-31015 Toulouse, France\\[.2em]
  Email: \texttt{maxime.derex@iast.fr}}

\author{Ludovic Maisonneuve}
\address{\flushleft
  \underline{Ludovic Maisonneuve}\\[.3em]
  National Museum of Natural History, UMR 7205, Institut de Systématique, Evolution et Biodiversité\\[.2em]
  F-75005 Paris, France\\[.2em]
  Email: \texttt{ludovic.maisonneuve@mnhn.fr}}
\author{Thomas Rey}
\address{\flushleft
  \underline{Thomas Rey}\\[.3em]
  Univ. Lille, CNRS, UMR 8524, Inria – Laboratoire Paul Painlevé\\[.2em]
  F-59000 Lille, France\\[.2em]
  Email: \texttt{thomas.rey@univ-lille.fr}}
\thanks{TR was partially funded by Labex CEMPI (ANR-11-LABX-0007-01) and ANR Project MoHyCon (ANR-17-CE40-0027-01). MD has received funding from the European Union's Horizon 2020 research and innovation programme under Marie Sklodowska-Curie grant agreement number 748310. Support from the ANR-Labex Institute for Advanced Study in Toulouse is acknowledged.}

\begin{document}

\begin{abstract}

    Understanding how knowledge is created and propagates within groups is crucial to explain how human populations have evolved through time. 
    Anthropologists have relied on different theoretical models to address this question. 
    In this work, we introduce a mathematically oriented model that shares properties with individual based approaches, inhomogeneous Markov chains and learning algorithms, such as those introduced in [F. Cucker, S. Smale, \textit{Bull. Amer. Math. Soc}, 39 (1), 2002] and [F. Cucker, S. Smale and D.~X Zhou, \textit{Found. Comput. Math.}, 2004]. 
    After deriving the model, we study some of its mathematical properties, and establish theoretical and quantitative results in a simplified case. 
    Finally, we run numerical simulations to illustrate some properties of the model.
    \\[1em]
\end{abstract}

\maketitle


\section{Introduction}

    \subsection{On social learning}
    
Computers, spaceships and scientific theories have not been invented by single, isolated individuals. Instead, they result from a collective process in which innovations are gradually added to an existing pool of knowledge, most often over multiple generations \cite{Boyd2011, Richerson2005}. The ability to learn from others (\emph{social learning}) is pivotal to that process because it allows innovations to be shared and be built upon by other individuals. 

This process, termed \emph{cumulative culture}, has been extensively studied by evolutionary anthropologists, both theoretically and experimentally \cite{Henrich2014, Powell2009, Derex2013, Muthukrishna2014}. Most existing theoretical models, however, rely on strong assumptions and omit important aspects of social dynamics. For instance, previous models typically assume that individuals learn from the most skilled member of their social group. Yet, in real life, many reasons can prevent this strategy to come about: individuals might fail to evaluate each other's skills and hierarchical or spatial structures might preclude individuals from accessing to the most useful sources of social information, among others. 

The aim of this work is to develop a more general mathematical model of knowledge evolution by relaxing hypotheses and incorporating more realistic forms of social interaction dynamics, such as those taking place in hierarchically or spatially structured populations.

    \subsection{Outline of the paper}

        In this work, we develop a new mathematical model that aims to describe the dynamics of knowledge creation and propagation among interacting individuals. The model is properly introduced and simple applications are given in Section \ref{sec:modelPres}. 
        In Section \ref{sec:mathRes}, we study some of the mathematical properties of the model, and establish theoretical, quantitative results in a simplified case describing the evolution of knowledge among interacting individuals. 
        {Finally, we develop a numerical method to simulate our model in Section \ref{sec:numRes}. This method allows us to run numerical analyses of the model in cases where we do not have analytical results and to present numerical illustrations of the classical model of \cite{CuckerSmale2004} on the evolution on language, which is contained in our model.}

    \subsection*{Acknowledgments} TR would like to thanks Mylène Maida for useful discussions on the inhomogeneous Markov chain structure of the model. LM would like to thanks Dorian Ni for his feedback on the model.

\section{Presentation of the mathematical model}

\label{sec:modelPres}

In this section, we shall present the model describing the evolution of knowledge within a finite population. Many different definitions of knowledge have been proposed. Here, we consider that knowledge results from conceptualizations that appropriately reflect the structure of the world and model \emph{conceptualization} as functions linking a set of possible experiences to a set of possible concepts. We call these functions \emph{knowledge-like functions}.

Time is supposed discrete. At each time step the knowledge-like function of individuals changes according to a learning dynamic that depends on both social and individual learning. Our model is an extension of the model of Cucker, Smale and Zhou describing the evolution of language \cite{CuckerSmale2004}, and can be seen as an hybrid between a learning algorithm \cite{CuckerSmale2001} and an individual based model \cite{AmbrosioFornasierMorandottiSavare2018, Albi2019vehicular}. 

We suppose that each individual influences each other through a social learning matrix $\Lambda \in \mathcal{M}_N(\RR)$. This matrix depends on both the structure of the population (\textit{e.g.} a professor has a strong impact on its students, while students have less impact on their professor) and the credibility that each individual grants to each other. These influences are described by a \emph{structure matrix} $\Gamma \in \MM_N(\RR)$ and a \emph{credibility matrix} $C \in \MM_N(\RR)$, respectively. Knowledge-like functions also evolve by \emph{individual learning} which is described as a stochastic process that we will detail in the following. The learning algorithm then takes into account \emph{both} social and individual learning.

Let us first start with some useful notations that we shall use in the following:

    \begin{itemize}
        \item The space of square matrices of size $N >0$ with coefficients in $\mathbb K$ will be denoted by $\mathcal M_N(\mathbb K)$.
        
        \item The vector of $\RR^N$ composed of $1$s will be denoted by $\mathbf e$:
        \begin{equation}
            \label{def:eqVector}
            \mathbf e = (1, \ldots, 1)^T \in \RR^N.
        \end{equation}
    
        \item The distance from a function $f$ to a set $\mathcal{X}$ is defined by
        \begin{equation*}
            d(f,\mathcal X) = \inf_{g \in \mathcal{X}} d(f,g).
        \end{equation*}
    
    \end{itemize}
    
\subsection{Modeling Knowledge}

\begin{definition}
A \emph{knowledge setting} $\mathcal{K}$ is a triple $(\mathcal{E},\mathcal{C},\mathcal{F})$ where :
\begin{enumerate}
    \item $\mathcal{E}$ is a closed and bounded subset of $\mathbb{R}^n$.
    \item $\mathcal{C} \subset \mathbb{E}^l$ with $l\in \mathbb{N}^*$, $\mathbb{E}$ an euclidean space, and $0 \in \mathcal{C}$.
    \item $\mathcal{F}$ is a subset of the set of the functions from $\mathcal{E}$ to $\mathcal{C}$.
\end{enumerate}
\end{definition}

The set $\mathcal{E}$ represents all the possible \emph{experiences}, and $\mathcal{C}$ represents all the \emph{concepts} (an illustration is presented in Fig \ref{fig:example_color}). 

\begin{definition}
A knowledge-like function $f \in \mathcal F$ is a function from the  experience set $\mathcal{E}$ to the concept set $\mathcal{C}$.
\end{definition}

Each knowledge-like function represents represents the knowledge of one individual. Let $e$ be in $\mathcal{E}$, when there is a $ c \in \mathcal{C}$ such as $f(e) = c$ and $c \neq 0$, we say that the knowledge-like function \emph{conceptualizes} $e$. We assume that individuals conceptualize all experiences they go through. Elements that are not conceptualized (i.e. not experienced) by individuals are sent to the zero of the set $\mathcal{C}$ by their knowledge-like function.

\medskip
\paragraph{Example. The knowledge-like function associated to colors.}
Let $\mathcal{E}$ be $[0, 1000] \subset \mathbb{R}$ representing the set of wavelengths in nanometers. We remind that $[380, 750]$ is the set of the visible spectrum. An individual associates each element of $\mathcal{E}$ to a color as shown in Figure \ref{fig:example_color}. The set of concepts $\mathcal{C}$ contains the name of the color and $0$. In this case $\mathcal{E}$ is a continuous space and $\mathcal{C}$ is a discrete space.

A knowledge-like function associates a color to each wavelength, or 0 if the individual has not conceptualized this color. For example $f$ defined below is a knowledge-like function.

\begin{equation}
    \label{eq:explejapaneseKnowledge}
    f'(e) =
\left\{
	\begin{array}{ll}
	    \text{purple}  & \mbox{if } e \in [380, 430], \\
	    \text{blue}  & \mbox{if } e \in [430, 520], \\
	    \text{green}  & \mbox{if } e \in [520, 565], \\
	    \text{yellow}  & \mbox{if } e \in [565, 610], \\
		\text{red}  & \mbox{if } e \in [610, 750], \\
		0 & \text{otherwise.}
	\end{array}
\right.
\end{equation}

\begin{figure}[ht]
    \centering
    \includegraphics[width=5cm]{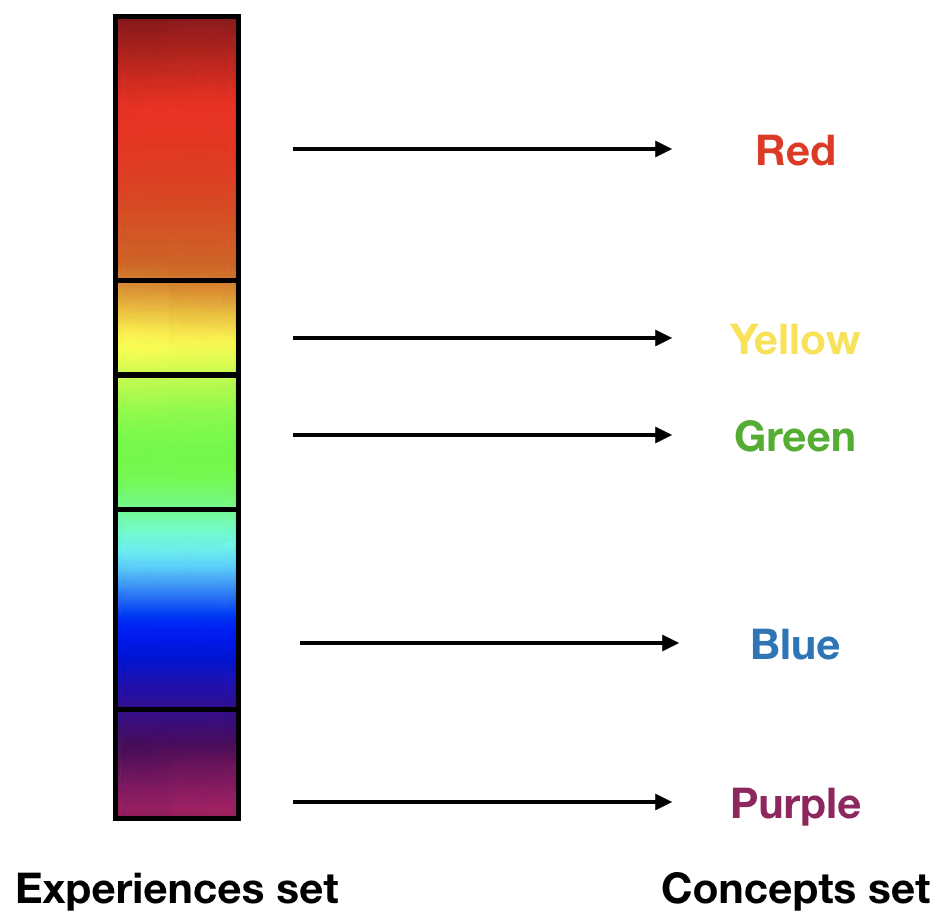}
    \caption{Illustration of a knowledge-like function of the visible spectrum. The arrows point to the color associated with each wavelength and illustrate the knowledge-like function $f$ \eqref{eq:expleKnowledge}.}
    \label{fig:example_color}
\end{figure}

\begin{rem}
    In Japanese the kanji \begin{CJK}{UTF8}{min}青\end{CJK} (\textit{ao}) names both colors green and blue. As in other languages, the Japanese language did not differentiate between these two colors at the beginning of its history. This language could be modeled by the knowledge-like function f':
    \begin{equation}
    \label{eq:expleKnowledge}
    f(e) =
\left\{
	\begin{array}{ll}
	    \text{purple}  & \mbox{if } e \in [380, 430], \\
	    \text{green}  & \mbox{if } e \in [430, 565], \\
	    \text{yellow}  & \mbox{if } e \in [565, 610], \\
		\text{red}  & \mbox{if } e \in [610, 750], \\
		0 & \text{otherwise}.
	\end{array}
\right.
\end{equation}
This knowledge-like function $f'$ is different from $f$ given in \eqref{eq:explejapaneseKnowledge}. Within the same population, individuals can have different knowledge-like functions.
\end{rem}
\subsection{Individual Based Model}

Let $N$ be the number of individuals in the population. Each individual $i$ is associated with a knowledge-like function $k^i \in \mathcal F$.

\begin{definition}
    \label{def:Gamma}
    A \emph{structure matrix} $\Gamma = (\gamma_{ij})_{1\leq i,j \leq N}$ is a square matrix of size $N$ describing the influence that individuals have on each other. More precisely for each $(i,j)\in\{1,...,N\}^2, \gamma_{ij} \in \RR$ describes the strength of the influence of $j$ on $i$. 
\end{definition}

\begin{rem}
  \label{rem:inertia}
For all $i \in \{1,...,N\}$ the greater the $\gamma_{ii}$, the less the individual $i$ will be influenced by others. So $\gamma_{ii}$ can be interpreted as the \emph{inertia} of the individual $i$.
\end{rem}

\paragraph{Examples of structure matrices.}

\begin{itemize}
    \item We consider a population of $N$ individuals structured in age, sorted such that individual 1 is the youngest and $N$ the oldest. It has been shown that older individuals tend to have a higher inertia  \cite{Gopnik7892}, which can be modelled by the condition $\gamma_{11} < ... < \gamma_{NN}$.
    \smallskip
    \item Let us consider the relationship between a parent and her offspring. The offspring learns a lot from her parent but the situation is not symmetric. Let $s \in (0,1)$ describes the influence of a parent on her offspring. We have 
    \begin{equation*}
        \Gamma = \begin{pmatrix}1&0.1\\s&1-s\end{pmatrix}.
    \end{equation*}
    \item We consider now the relationship between two students and their professor. Because of her status, the professor has a high influence on her students, while being very  little influenced by them. Assuming that the relationship between the students is symmetric, we have
    \begin{equation*}
        \Gamma = \begin{pmatrix}1&0.1&0.1\\1&0.2&0.2\\1&0.2&0.2\end{pmatrix}.
    \end{equation*}
\end{itemize}

\subsection{Likelihood landscape}

In our model, some conceptualizations (i.e. knowledge-like functions) appropriately reflect the structure of the world, while other do not. For instance, in an environment in which blue berries are safe to eat while green berries are unsafe, color categorizations that discriminate between blue and green are superior because they appropriately capture the structure of the environment. Individuals don't know \emph{a priori} how to categorize their environment. An individual who, by chance, only ever ate blue/safe berries might consider that discriminating between blue and green makes no sense. Yet, an individual who got sick after eating green/unsafe berries is likely to refine her color conceptualization to avoid being sick again. Sometimes, alternative and irreconcilable conceptualizations are equally likely. As an illustration let us consider the shape illustrated in figure \ref{fig:cup_or_faces}. One might consider that it represents (1) two faces (in black) or (2) one cup (in white). Additional observations will not allow individuals to decide whether one conceptualization is more likely than the other.

\begin{figure}
\centering
\includegraphics[width=4cm]{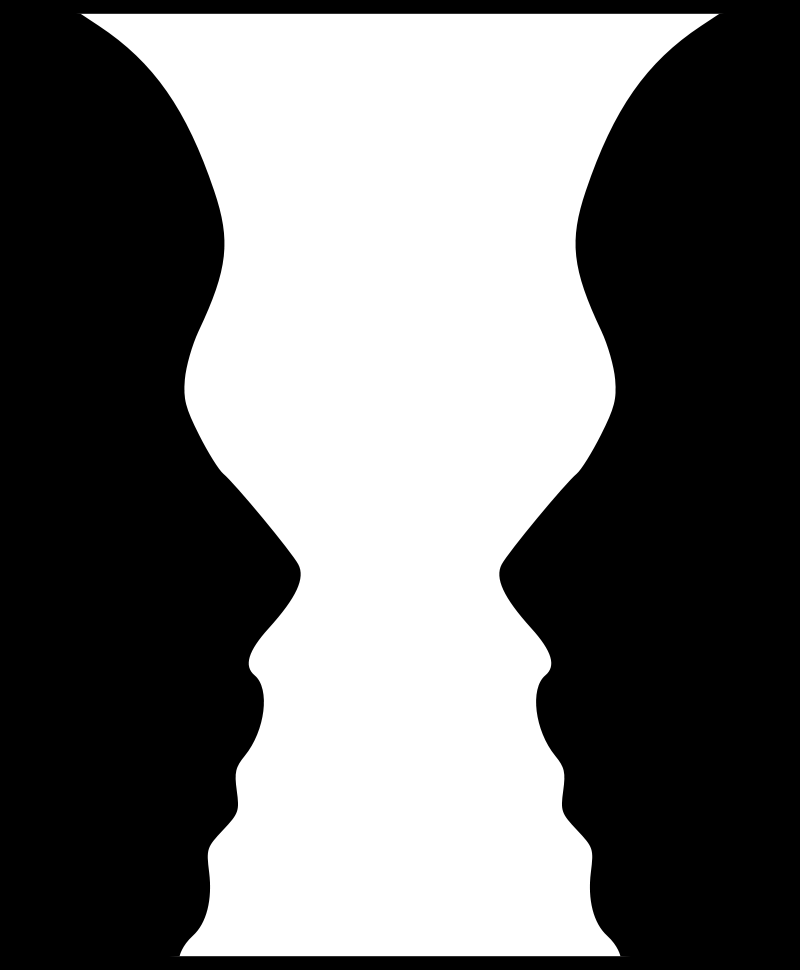}
\caption{On this figure one may see two faces (in black) or a cup (in white). Image courtesy of Bryan Derksen, under license CC BY-SA 3.0}
\label{fig:cup_or_faces}
\end{figure}

In our model, we assume that  individuals evaluate the likelihood of their conceptualization according to their own experience. To do so we define a likelihood landscape as following:

\begin{definition}
  \label{def:likelihoodLandscape}
$\forall c\in \mathcal{C}$, we define a function 
\begin{align*}
    L(\cdot,c): & \, \mathcal{E} \to [0,1],\\
    & \, e \mapsto L(e,c).
\end{align*}
with 
\begin{equation*}
    \forall e \in \mathcal{E}, \ L(e,0) = \frac{1}{2}.
\end{equation*} 

The map $L$ is called the \emph{likelihood landscape}. 
\end{definition}
For all $e$, $c$ in $\mathcal{E} \times \mathcal{C}$, $L(e,c)$ represents how well the concept $c$ explain the experience (or observation) $x$.

\medskip
\paragraph{Examples of likelihood landscapes.}

\begin{itemize}
    \item Let us consider the evolution of two different concepts in a population: flat earth ($F$), and round earth ($R$). Individuals can have experiences where the Earth seems flat ($f$) and others where the Earth seems round ($r$) (seeing a picture of the Earth, a boat vanishing behind the horizon, etc). In this case $\mathcal{E} = \{f,r\}$ and $\mathcal{C} = \{0,F,C\}$. We define the likelihood landscape as :
    \begin{equation*}
        L(e, F) = \left\{ \begin{array}{ll}
		1 & \mbox{if } e = f,\\
		0 & \mbox{if } e = r,\\
	\end{array}
    \right.
    \end{equation*}
    and
    \begin{equation*}
        L(e, R) = \left\{ \begin{array}{ll}
    		1 & \mbox{if } e = f,\\
    		1 & \mbox{if } e = r.\\
    	\end{array}
    \right.
    \end{equation*}
    
    When the earth seems flat ($f$) the earth could be flat or round (because round surfaces can appear flat when observed up close), so both concepts ($F$ and $R$) are likely. However, when the earth seems round only the concept that the earth is round is likely.
    \smallskip
    \item Let us consider again the example of color developed above (Fig. \ref{fig:example_color}), where $E$ is the set of all wavelengths of the visible spectrum and $\mathcal{C}$ is the set of colors. Moreover, let us consider that it does not make sense to discriminate between colors. In that case, we would define the likelihood landscape as $L(e,c) = 1$ for all $(e, c) \in \mathcal{E} \times \mathcal{C} \setminus\{0\}$.
\end{itemize}

\subsection{Credibility}

 In addition to the population structure $\Gamma$ from Definition \ref{def:Gamma}, the influence of individuals on each other also depends on their \emph{credibility}, through the credibility matrix $C$. The level of credibility attributed to an individual by another depends on both the knowledge-like functions, and the likelihood landscape. 
 
 More precisely,  $c_{ij}$ describes the credibility individual $i$ attributes to individual $j$. If the credibility given to individual $j$ by individual $i$ is high relatively to that one attributed to other individuals (including herself), that means that individual $i$ is more prone to adopt individual $j$'s conceptualization.
 We will describe in Section 2.6 how this adoption changes one's knowledge-like function.
We also consider that the credibility $c_{ii}$ that an individual $i$ gives to its own categorization can be affected by her own new experiences. In other words, individuals are able of self-criticism. The lower the self-credibility, the more likely an individual is to be influenced by other individuals (self-credibility directly affects an individual's inertia).

\begin{definition}
 A credibility matrix $C = (c_{ij})_{1 \leq i,j \leq N}$ is a square matrix of size $N$ defined by:
\begin{equation}
    \label{def:credibilityMin}
    c_{ij} = \max(\Tilde{c}_{ij}, c_{\min}),
\end{equation}
where $c_{\min} \geq 0$ is a fixed parameter , and

\begin{equation}
    \label{def:credibilityMat}
    \Tilde{c}_{ij} = \frac{1}{1 + \mathds{1}_{\{i\neq j\}}\int_{c \in k_j(X)}\, \mathrm{d}c}\exp{\left(\int_{e \in E}\ln{\left(L(e,k_j(e))\mathds{1}_{\{k_i(e)\neq0\}} \right)}\, \mathrm{d}e\right)}.
\end{equation}
\end{definition}

\begin{rem}
    The term $\left (1 + \mathds{1}_{\{i\neq j\}}\int_{c \in k_j(X)}\, \mathrm{d}c \right )^{-1}$ penalizes individuals who use on a wider range of concepts, which means that conceptualization that rely on smaller number of concepts are more likely to spread. The second term corresponds to the evaluation of the likelihood of the knowledge-like function of individual $j$ on the experiences experienced by individual $i$. Note that for all $e$, $L(e, 0) = \frac{1}{2}$ (corresponding to cases where individual $j$ has not experienced $e$), which decreases individual's credibility. This means that an individual $i$ considers an individual $j$ less credible if $j$ has not experienced an experience individual $i$ has gone through. 
    
    Finally, the constant $c_{\min}$ corresponds to the minimal credibility: if $c_{\min}>0$, individuals with low credibility can still influence other individuals.
\end {rem}

\begin{rem}
If the set $\mathcal E$ contains a finite number of elements the credibility formula reduces to
\begin{equation*}
    \Tilde{c}_{ij} = \frac{1}{1 + \mathds{1}_{\{i\neq j\}}\int_{c \in k_j(X)}\, \mathrm{d}c}\prod_{e \in E}\mathds{1}_{k_i(e)\neq0} L(e,k_j(e)),
\end{equation*}
namely, the second part of the formula is similar to a measure of likelihood in probability theory \cite{le_gall_2006}. In this formula, associating few experiences with unlikely concepts penalizes credibility a lot.
\end{rem}

\paragraph{Application to the round \emph{vs.} flat earth example.} Let $\mathcal{E} = \{f_1, f_2, f_3, r_1, r_2\}$ and $\mathcal{C} = \{0,F,R\}$, together with $c_{\min} := 0$. For any $i\in\{1,2,3\}, f_i$ are experiences where the Earth as likely to be flat as it is round (e.g. a human watching the horizon), and $r_1, r_2$ are experiences where the Earth is unlikely to be flat and likely to be round. We consider a population of 4 individuals with different knowledge-like functions $k_1, k_2, k_3$ and $k_4$ such as: 

\begin{equation*}
    k_1(e) = \left\{ \begin{array}{ll}
		F & \mbox{if } e = f_1,\\
		0 & \mbox{otherwise},\\
	\end{array}
    \right.\quad k_2(e) = \left\{ \begin{array}{ll}
		F & \mbox{if } e = f_1 \text{ or } e = r_1,\\
		0 & \mbox{otherwise },\\
	\end{array}
    \right.
\end{equation*}

\begin{equation*}
    k_3(e) = \left\{ \begin{array}{ll}
		R & \mbox{if } e = f_1 \text{ or } e = r_1,\\
		0 & \mbox{otherwise},\\
	\end{array}
    \right.\quad k_4(e) = \left\{ \begin{array}{ll}
		F & \mbox{if } e = f_1, \\
		R & \mbox{if } e = r_1, \\
		0 & \mbox{otherwise.}\\
	\end{array}
    \right.
\end{equation*}

We now compute the credibility matrix, 
\[C = \begin{pmatrix}
1&1&1&1\\
\frac{1}{2}&0&1&1\\
\frac{1}{2}&0&1&\frac{1}{2}\\
\frac{1}{2}&0&1&1
\end{pmatrix}.
\]

Let us normalize C such that its lines sum up to $1$, in order to easily read the influences on an individual $j$ in the row $i$:
\[
\title{C} = 
\begin{pmatrix}
\frac{1}{4}&\frac{1}{4}&\frac{1}{4}&\frac{1}{4}\\
\frac{1}{5}&0&\frac{2}{5}&\frac{2}{5}\\
\frac{1}{4}&0&\frac{1}{2}&\frac{1}{4}\\
\frac{1}{5}&0&\frac{2}{5}&\frac{2}{5}
\end{pmatrix}.
\]

Individual \#1 has only experienced $f_1$, she judges the other individuals (and herself) regarding that sole experience. All individuals associate $f_1$ with an appropriate conceptualization. So individual \#1 gives the same credibility to all individuals. Individuals \#2, 3 and 4 all have experienced $f_1$ and $r_1$. They evaluate individual \#1 as less credible than themselves because \#1 has not experienced $r_1$. Individual \#3 and 4 evaluate \#2 as not credible at all because she associates $r_1$ with a concept that is not probable anymore (i.e. the earth is flat while their experience shows it is round). \#2 judges herself not credible because she associates $r_1$ with an unlikely concept. Individual \#3 evaluates \#4 as less credible than herself because \#4 uses several concepts (in our model, less parsimonious conceptualizations are penalized).

\subsection{Social learning}

The \emph{social learning} matrix $\Lambda \in \MM_N(\R)$ represents the influence of individuals on each other. This matrix captures the effect of the structure of the population described in $\Gamma$ and the effect due to credibility $C$.

The influence of individual $j$ on  $i$ depends on the structural influence $\gamma_{ij}$ of $j$ on $i$, and on the credibility $c_{ij}$ that $i$ gives to $j$. We shall assume in this work that these phenomena are \emph{multiplicative}.

\begin{definition}
The social learning matrix $\Lambda = (\lambda_{ij})_{1 \leq i,j \leq N}$ is a square matrix of size $N$ defined by:
\begin{equation}
    \label{def:socialLearningMat}
    \lambda_{ij} = \left\{ \begin{array}{ll}
		\frac{\gamma_{ij}c_{ij}}{\sum_{l=1}^N\gamma_{il}C_{il}} & \mbox{if } \sum_{l=1}^N\gamma_{il}C_{il} \neq 0,\\
		1/N & \mbox{otherwise.}\\
	\end{array}
    \right..
\end{equation}
\end{definition}

\subsection{Dynamics}

Finally, we consider a dynamical, discrete time model: knowledge-like functions evolve over time, altogether with associated quantities such as individuals' credibility. 
Let us denote by $k^t := (k_1^t,...,k_N^t)\in \mathcal{F}^N$ the state of the population at time $t >0$. As time evolves, individuals modify their conceptualization by the learning algorithm presented in \cite{CuckerSmale2001}:
\begin{equation}
    \label{eq:LearningAlg}
    S_i^t \mapsto k_i^{t+1},
\end{equation}
which computes knowledge-like function from a sample
\begin{equation}
    \label{eq:sampling}
    S_i^t = \{(e_1^{i,t},c_1^{i,t}),...,(e_m^{i,t},c_m^{i,t})\}.
\end{equation}
The sampling $S_i^t$ is done using  the probability measure $\rho^{i,t}$ defined  in \eqref{rho}. For each individual $i$, the components of $S_i^t$ represent the influences that will shape the knowledge of $i$ at the next step. The elements of $S_i^t$ can come from social or individual learning. 

\begin{definition}
    Let us denote by $\tau \geq 0$ the \emph{proportion} of {individual learning}, fixed and independent on $i$. The \emph{sampling} measure is defined by
    \begin{equation}
        \rho^{i,t}= (1-\tau) \rho_\Lambda^{i,t} + \tau \rho_\I^{i,t},
        \label{rho}
    \end{equation}
    where $\rho_\Lambda^{i,t}$ and $\rho_\I^{i,t}$ are two probability measures representing the effects of social and individual learning respectively, and defined below.
\end{definition}


\begin{definition}
The probability measure $\rho_\Lambda^{i,t}$ representing \emph{social learning} is defined by:
\begin{equation}
    \label{def:socialLearningMeasure}
    \rho_\Lambda^{i,t}(e,c) \propto \sum_{j=1}^N \lambda^t_{ij}\mathds{1}_{\{k_j(e)=c\}},
\end{equation}
where $\lambda_{ij}^t$ describes the influence of the individual j on the individual i through \ref{def:socialLearningMat} at time $t$.
\end{definition}

According to \eqref{def:socialLearningMeasure}, drawing an element $(e,k_i^t(e))$ using the probability measure $\rho_\Lambda^{i,t}$ is equivalent   to randomly drawing an individual $i$ weighted by the coefficient $(\lambda^t_{ij})_{1 \leq j \leq N}$, and randomly choosing an experience $e$ in $\mathcal{E}$.

\begin{definition}
    The probability measure $\rho_\I^{i,t}$ representing \emph{individual learning} is given by
    \begin{equation}
        \label{def:individualLearningMeasure}
        \left \{ \begin{aligned}
            & \int_{c \in C} \rho_\I^{i,t}(e,c) \, \mathrm{d}c \propto \int_{e' \in E} \mathds{1}_{\{k_i^t(e') \neq 0\}} \mathds{1}_{\{e \in E\}} \exp{-\frac{\|e - e'\|^2}{2 \sigma_E^2}} \, \mathrm{d}e', \\
            & \rho_\I^{i,t}(c|e) \propto \mathds{1}_{\{c\in C\}} \exp{-\frac{\|c - f_i^t(e)\|^2}{2 \sigma_C^2}},
        \end{aligned} \right.
    \end{equation}
    where $\rho(c|e)$ denotes the conditional probability measure on $\mathcal{C}$, defined  for every $(e,c) \in \mathcal E \times \mathcal C$,  and every integrable function $\phi$ by
     \begin{equation*}
        \int_{\mathcal{E}\times \mathcal{C}}\phi(e,c) \, \mathrm{d}\rho = \int_{\mathcal{E}}\left(\int_{\mathcal{C}}\phi(e,c) \, \mathrm{d}\rho(c|e)\right) \, \mathrm{d}\rho_\mathcal E.
    \end{equation*}
    In this last expression, $\rho_\mathcal{E} $ denotes the \emph{marginal} probability measure on $\mathcal{E}$, namely
    \[ 
        \rho_\mathcal{E}(e) := \rho(\pi^{-1}(e)), \quad \forall e\in \mathcal{E},
    \] 
    where $\pi : \mathcal{E} \times \mathcal{C} \to \mathcal{E}$ is the projection on $\mathcal{E}$.
\end{definition}

The individual learning phase is then equivalent for each individual $i$ to draw an element $e'$ experienced by $i$ and to draw an experience $e$ following a normal law centered and concentrated on $e'$. Because of the shape of this probability law, individuals tend to explore the set $\mathcal{E}$ close to the elements they  already explored (namely, innovating). The concept $c$ is drawn following a probability law centered and concentrated on $k_i^t(e)$.

\begin{definition}
    The learning algorithm finally computes the knowledge-like function at the next time step using a least-square procedure \cite{CuckerSmale2001} 
    \begin{equation*}
        k_i^{t+1} \in \argmin_k \sum_{(e,c)\in S_i^t} (k(e)-c)^2.
    \end{equation*}
\end{definition}



\section{The case of globally shared knowledge: convergence without individual learning}
\label{sec:mathRes}

In this section, we are interested in the convergence of the learning dynamics with high probability  to a common shared conceptualization among individuals, i.e. when everybody carries the same knowledge-like function $k$. This result is obtained assuming no individual learning : in all this section, we shall assume that the rate of individual learning $\tau = 0$. The more realistic case where individuals also learn individually is explored below using numerical simulations.

    \medskip
    \paragraph{One step idealistic processes.}
    In order to establish the main result, let us decompose the stochastic process $k^t$ into two other processes that we shall analyze separately. 
    
    \begin{definition}
        Let us define the application $\mathcal{T} : \mathcal{F} \times \mathbb{N} \mapsto \mathcal{F}$ by 
        \begin{equation}
            \label{def:Tcal}
            \mathcal{T}(f,t) = \Lambda^t f.
        \end{equation}
        We can then define the \emph{one step} deterministic \emph{idealistic process} as
        \begin{equation} 
            \label{def:idealProc}
            K^t_{\mathcal{T}} := \mathcal{T}(k^t,t), \quad  K^0_{\mathcal{T}} = k^0 \in \mathcal F^N.
        \end{equation}
    \end{definition}
    
Using these idealistic processes, the time evolution of the knowledge-like function is given by 
\begin{equation}
    \label{eq:decompkt}
    k^t = \Delta k^t + K_{\mathcal{T}}^t,
\end{equation}
where $\Delta k^t = k^t - K_{\mathcal{T}}^t$. 

\begin{definition}
    Let $\mathcal M_\mathcal{F} = \{(k, ...,k), k \in \mathcal{F}\}$ be the space of all the common shared conceptualizations.
\end{definition}

We first prove the contraction of the idealized process $K_{\mathcal{T}}^t$ from \eqref{def:Tcal} in the space $\MM_\mathcal{F}$ using some algebraic properties of primitive matrices, as well as results about inhomogeneous Markov chains. 
Secondly, we shall prove the convergence of the process $\Delta k^t$ with high probability using learning theory. Then, under certain hypothesis (such as $\tau = 0$), we prove the convergence of  $k^t$ towards the set $\MM_\mathcal{F}$ with high probability. 

%
%
%
%

\subsection{Primitive matrices and their applications}

{The behavior of processes is mainly driven by the influence matrix $\Lambda$. In this Section, we study the relationship between the properties of the influence matrix and the interactions taking place within the population.}

\begin{definition}
A matrix $A \in \mathcal{M}_N(\mathbb{R})$ is said to be primitive if $A\geq0$ and if $\exists k \in \mathbb{N}^*$ such as $A^k>0$.
\end{definition}

\begin{definition}
\label{def:communicates}
Let $A \in \mathcal{M}_N(\mathbb{R})$. Let $i$,$j$ be in $\{1,...,N\}$, 
\begin{itemize}
    \item We say that $i$ \emph{communicates} with $j$ (denoted $i \xrightarrow{A} j$) if there exists $n\geq0$ and $i_1,...,i_n \in \{1,...,N\}$ such that $$a_{ii_1}\prod_{l=1}^{n-1} (a_{i_{l}i_{l+1}}) a_{i_nj}>0.$$
    If $i$ does not communicate with $j$ we write $i \nrightarrow j$.
    \item We say that $i$ communicates with $j$ with $k$ \emph{intermediates} if there exists $i_1,...,i_k \in \{1,...,N\}$ such that $$a_{ii_1}\prod_{l=1}^{k} (a_{i_{l}i_{l+1}}) a_{i_kj}>0,$$
    and if there exists not $i_1,...,i_{k-1} \in \{1,...,N\}$ such that $$a_{ii_1}\prod_{l=1}^{k-1} (a_{i_{l}i_{l+1}}) a_{i_{k-1}j}>0,$$
    \item Let $I,J \in \mathscr{P}(\{1,...,N\})$. We say that $I \xrightarrow{A} J$ if
    \begin{equation*}
        \exists i,j \in I \times J \text{ such that } i \xrightarrow{A} j.
    \end{equation*}
\end{itemize}
\end{definition}

If $A$ is a primitive matrix of size $N$, for each $i$, $j$ in $\{1,...,N\}, i\rightarrow j$. Moreover if there is $k$ such that $A^k >0$, then for each $i$, $j$ in $\{1,...,N\}$ $i$ communicates with $j$ with at most $k$ intermediates.  

\medskip
\paragraph{Examples.} Let us consider a matrix $A$ defined by
\begin{equation*}
    A = \begin{pmatrix}
    1&1&0&0\\1&1&1&0\\0&1&1&1\\0&0&1&1
    \end{pmatrix}.
\end{equation*}
As we can see on the graph of the matrix $A$ (Fig. \ref{fig:graphprimitive matrix}), every individual communicates with each other with at most 3 intermediates.
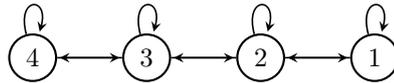
\begin{figure}[ht]
\centering
\begin{tikzpicture}[
            > = stealth, 
            shorten > = 1pt, 
            auto,
            node distance = 1.5cm, 
            semithick 
        ]

        \tikzstyle{every state}=[
            draw = black,
            thick,
            fill = white,
            minimum size = 4mm
        ]
        \node[state] (1) {1};
        \node[state] (2) [left of=1] {2};
        \node[state] (3) [left of=2] {3};
        \node[state] (4) [left of=3] {4};

        \path[->] (1) edge [loop above] node {} (1);
        \path[->] (2) edge [loop above] node {} (2);
        \path[->] (3) edge [loop above] node {} (3);
        \path[->] (4) edge [loop above] node {} (4);
        \path[->] (1) edge node {} (2);
        \path[->] (2) edge node {} (1);
        \path[->] (3) edge node {} (2);
        \path[->] (2) edge node {} (3);
        \path[->] (3) edge node {} (4);
        \path[->] (4) edge node {} (3);
    \end{tikzpicture}
    \caption{Graph representing the matrix A.}
    \label{fig:graphprimitive matrix}

\end{figure}
Then, $A$ is a primitive matrix, because
\begin{equation*}
    A^3 = \begin{pmatrix}
    4&5&3&1\\5&7&6&3\\3&6&5&7\\1&3&5&4
    \end{pmatrix} >0.
\end{equation*}

However the converse is not true. Indeed if one considers the matrix 
\begin{equation*}
    B = 
    \begin{pmatrix}
        0&1\\1&0
    \end{pmatrix},
\end{equation*}
then $1\rightarrow2$ and $2\rightarrow1$ (see fig. \ref{fig:graph_non_primitive_matrix}).
\begin{figure}[ht]
\centering
\begin{tikzpicture}[
            > = stealth, 
            shorten > = 1pt, 
            auto,
            node distance = 1.5cm, 
            semithick 
        ]

        \tikzstyle{every state}=[
            draw = black,
            thick,
            fill = white,
            minimum size = 4mm
        ]
        \node[state] (1) {1};
        \node[state] (2) [left of=1] {2};
        
        \path[->] (1) edge node {} (2);
        \path[->] (2) edge node {} (1);
        
    \end{tikzpicture}
    \caption{Graph representing the matrix $B$.}
    \label{fig:graph_non_primitive_matrix}
    
\end{figure}
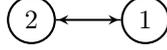
Nevertheless, $B$ is not a primitive matrix because for all $k \in \mathbb{N}$,
\begin{equation*}
    B =
\left\{
	\begin{array}{ll}
		B  & \mbox{if $k$ is odd}, \\
		I_2 & \mbox{if $k$ is even}.
	\end{array}
\right.
\end{equation*}

%

According to Perron-Frobenius theorem, one has
\begin{prop}
Let $A \in \mathcal{M}_N(\mathbb{R})$ be a primitive stochastic matrix. $1$ is an eigenvalues of A, and all the other eigenvalues are less than 1 in modulus.
\label{p<1}
\end{prop}

We can now establish results relating  graphs and eigenvalues of matrices:

\begin{prop}
Let $A$ be a stochastic matrix of size $N$, if 
\begin{equation*}
    \forall i,j \in \{1,...,N\}, i\xrightarrow{A} j \text{ or } j \xrightarrow{A} i,
\end{equation*}
holds then $1$ is an eigenvalue of $A$, and its multiplicity is $1$.
\end{prop}

\begin{proof}
We suppose that $1$ is not an eigenvalue of multiplicity $1$ of $A$. Since $A$ is a stochastic matrix one has $A \mathbf e = \mathbf e$ with $\mathbf e$ given by \eqref{def:eqVector}, so $1$ is an eigenvalue of $A$. In particular, its order of multiplicity is bigger than 1, and there exists $X\in \mathbb{R}^N\setminus\RR \mathbf e$ such that $A X = X$. 

Let P be a permutation matrix such that the coordinates of the vector $PX$ are ranked from the lowest to the highest. Let $A' = P^{T}A P$ and $X' = P^{T}X$. Let $n_l$ and $n_h$ be respectively the number of coordinates equals to the lowest and to the highest coordinates values. We have $n_l \geq 1$, $n_h \geq 1$ and $n_l + n_h \leq N$.

\smallskip
\paragraph{First case : $n_l + n_h = N$}

\begin{align*}
    A X = X & \iff A'X' = X'. \\
    &\implies 
		\forall i \in \{1,...,N\},\sum_{j=1}^N a_{ij}' X_j' = X_i'.\\
	&\implies \left\{ \begin{array}{l}
		\forall i \in \{1,...,n_l\},\forall j \in \{n_l + 1, ..., N\}, a_{ij}' = 0, \\\forall i \in \{n_l+1,...,N\},\forall j \in \{1, ..., n_l\}, a_{ij}' = 0.
	\end{array} \right.\\
	&\implies \exists (B,C) \in \mathcal{M}_{n_l}(\mathbb{R})\times\mathcal{M}_{n_h}(\mathbb{R}),  A = \left (\begin{array}{c|c}
B & 0\\
\hline
0 & C 
\end{array}\right).
\end{align*}

So $1 \nrightarrow N$ et $N \nrightarrow 1$ with the matrix $A'$. Let $i = \sigma^{-1}(1)$ and $j = \sigma^{-1}(N)$, we have $i\nrightarrow j$ and  $j\nrightarrow i$ with the matrix $A$.

\smallskip
\paragraph{Second case : $n_l + n_h < N$}
Let $G_1 = \{1,...,n_l\}$, $G_2 = \{n_l + 1,...,N - n_l\}$, and $G_3 = \{N - n_h + 1,...,N\}$.

\begin{align*}
    A X = X & \iff A'X' = X'.\\
    &\implies 
		\forall i \in \{1,...,N\},\sum_{j=1}^N a_{ij}' X_j' = X_i'.\\
	&\implies \left\{ \begin{array}{l}
		\forall i \in \{1,...,n_l\},\forall j \in \{n_l + 1, ..., N\}, a_{ij}' = 0 ,\\\forall i \in \{N-n_h+1,...,N\},\forall j \in \{1, ..., n_l\}, a_{ij}' = 0.
	\end{array} \right.\\
	&\implies \exists (B,C,D,E,F),  A = \left (\begin{array}{c|c|c}
B & 0 & 0\\
\hline
C & D & E \\
\hline
0 & 0 & F 
\end{array}\right).
\end{align*}
\begin{figure}[ht]
\centering
\begin{tikzpicture}[
            > = stealth, 
            shorten > = 1pt, 
            auto,
            node distance = 1.5cm, 
            semithick 
        ]

        \tikzstyle{every state}=[
            draw = black,
            thick,
            fill = white,
            minimum size = 4mm
        ]
        \node[state] (2) {$G_2$};
        \node[state] (1) [below left of=2] {$G_1$};
        \node[state] (3) [below right of=2] {$G_3$};

        \path[->] (1) edge [loop left] node {} (1);
        \path[->] (2) edge [loop above] node {} (2);
        \path[->] (3) edge [loop right] node {} (3);
        \path[->] (1) edge node {} (2);
        \path[->] (3) edge node {} (2);
    \end{tikzpicture}
    \caption{Illustration of the influence relationship between the three clusters $G_1$, $G_2$, and $G_3$}
    \label{fig:graphproof}
    
\end{figure}
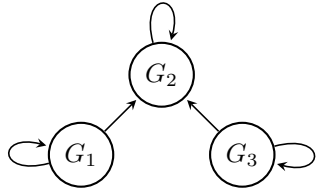
As shown on the figure \ref{fig:graphproof}, $G_1\nrightarrow G_3$ and $G_3\nrightarrow G_1$.

We conclude as before.

\end{proof}

\subsection{Eigenvalues of the matrix of influence $\Lambda$}
{In this Section we study the the quantitative properties of the eigenvalues of the influence matrix, in order to understand the dynamics of the idealized process $K_{\mathcal{T}}^t$.}

Let $\Gamma$ and $C$ be respectively the structure and credibility matrices defined in Def. \ref{def:Gamma} and equation \eqref{def:credibilityMat}.
By construction, $C$ is a stochastic matrix. The structure matrix $\Lambda$ is defined according to \eqref{def:socialLearningMat} by  
\[\lambda_{ij} = \left\{ \begin{array}{ll}
		\frac{\gamma_{ij}c_{ij}}{\sum_{l=1}^N\gamma_{il}C_{il}} & \mbox{if } \sum_{l=1}^N\gamma_{il}C_{il} \neq 0,\\
		1/N & \mbox{otherwise.}\\
	\end{array}
\right.\]

\begin{prop}
 If for all $i,j \in \{1,...,N\}$, one has $i\xrightarrow{\Gamma} j$ or $j \xrightarrow{\Gamma} i$, and the credibility matrix $C > 0$, then  for all $i,j \in \{1,...,N\}$, $i\xrightarrow{\Lambda} j$ or $j \xrightarrow{\Lambda} i$.
\end{prop}

\begin{proof}

$\forall i,j \in \{1,...,N\}, \gamma_{ij} > 0 \implies \lambda_{ij} > 0$ using \eqref{def:socialLearningMat}.
Thus $\forall i,j \in \{1,...,N\}, i\xrightarrow{\Gamma} j \text{ or } j \xrightarrow{\Gamma}  > 0 \implies i\xrightarrow{\Lambda} j \text{ or } j \xrightarrow{\Lambda} i$.

\end{proof}

\begin{lemma}
If $\Gamma$ is primitive and $c_{\min}>0$, then $\Lambda$ is a primitive matrix.
\label{lemma:primitive}
\end{lemma}

\begin{proof}
$\Gamma$ is primitive so there exists $k \in \mathbb{N}$ such that $\Gamma^k > 0$. In particular,
\begin{align*}
   \forall i,j \in \{1,...,N\}, \gamma^k_{ij} > 0 \iff&  \forall i,j \in \{1,...,N\}, \\
   &\sum_{0\leq l_1,...,l_{k-1}\leq N}\gamma_{il_1}\gamma_{l_1l_2}...\gamma_{l_{k-1}j} > 0 ,\\
   \iff& \forall i,j \in \{1,...,N\},  \exists l_1,...,l_{k-1} \in \{1,...,N\}, \\&\gamma_{il_1}\gamma_{l_1l_2}...\gamma_{l_{k-1}j} > 0.
\end{align*}
Moreover $C>0$ implies that $\forall i,j \in \{1,...,N\}, \gamma_{ij} > 0$. One can then conclude that $\lambda_{ij} > 0$.
\end{proof}

\begin{lemma}
 Let us assume that $\Gamma > 0$. If $c_{\min}>0$, the social learning matrix $\Lambda$ is \emph{bounded by below}: there exists $\underline m_\lambda > 0 $ such that $\forall i, j \in \{1, ..., N\}$, $\lambda_{ij} \geq \underline m_\lambda$.
\label{lemma:lambda>eps}
\end{lemma}

\begin{proof}
Let us set 
$\underline{m}_\gamma = \min_{i,j} \gamma_{ij}$.
Since $\gamma_{ij} > \underline m_\gamma$ for all $i,j \in \{1,...,N\}$, one has that
\[\lambda_{ij} > \mathds{1}_{\sum_{l=0}^N\gamma_{il}C_{il} = 0} \frac{1}{N} + \mathds{1}_{\sum_{l=0}^N\gamma_{il}C_{il} \neq 0} 	\frac{\underline m_\gamma c_{\min}}{N(1 - N\underline m_\gamma)}.\]
It is then enough to choose
\begin{equation*}
\underline m_\lambda= \min \left(\frac{1}{N},\frac{\underline m_\gamma c_{\min}}{N(1 - N\underline m_\gamma)}\right).
\end{equation*}
\end{proof}

\begin{lemma}
\label{lemma:upperBound}
Let $\Lambda$ be a stochastic matrix of size N that is bounded by below by $\underline m_\lambda$ as in Lemma \ref{lemma:lambda>eps}.
Let us consider the sorted collection $(\alpha_i)_{i=1\ldots N}$ of its eigenvalues:
\[\alpha_1 = 1 > |\alpha_2| \geq ... \geq |\alpha_N|.\] 
Then there exists a universal constant $0< \overline M_\lambda < 1$ such that 
\[\forall i \in \{2,...,N\}, |\alpha_i| \leq 1 - \overline M_\lambda.\] 
\label{alpha2<}
\end{lemma}

\begin{proof}
Let $\mathcal{A} = \{A \in \mathcal{M}^N(\mathbb{R}), \forall i \in \{1,...,N\}, \sum a_{ij} = 1, \forall i,j \in \{1,...,N\}, a_{ij} \geq \underline m_\lambda \}$. 
Being composed of stochastic matrices, the set $\mathcal{A}$ is bounded (by $1$) for the norms induced by both the $1$ and $\infty$ vector norms on $\mathbb R^N$. Moreover, it is closed by construction. In particular, $\mathcal{A}$ is  a compact subset of $\mathcal{M}^N(\mathbb{R})$.

Let
\begin{align*}
  \mathcal{L} \colon &\mathcal{M}^N(\mathbb{R})\to \mathbb{C}^N\\
  &A \mapsto Sp(A),
\end{align*}
the application that returns the eigenvalues of a matrix, sorted in a nonincreasing (in modulus) order.
According to the Theorem II.5.1 of \cite{kato2013perturbation}, $\mathcal{L}$ is a continuous function on the set of stochastic matrices. In particular, $\mathcal A$ being compact, the numerical range of $\mathcal L$, 
\begin{equation*}
R(\mathcal{L}) := \{\mathcal{L}(A), \forall A \in \mathcal{A} \} \subset \mathbb C^N,
\end{equation*}
is also compact.
Continuous function reach their bounds on compact sets, so that one can take 
\begin{align*}
    \overline M_\lambda &= 1 - \sup_{\mathcal L \in \mathcal A}\{|\mu_2|, \mu \in R(\mathcal{L})\}\\
    & = 1 - \max\{|\mu_2|, \mu \in R(\mathcal{L})\}.
\end{align*}

\end{proof}

\subsection{Contraction of a stochastic primitive matrix}

    We shall now study in this section  properties of stochastic primitive matrices, in order to understand the behavior of the process $K_{\mathcal{T}}^t$.

\begin{lemma}
    \label{lem:contraction}
    Let $A \in \mathcal{M}_N(\mathbb{R})$ be a stochastic matrix that is bounded by below by $\underline m_A$ as in \ref{lemma:lambda>eps}. Let $ \mathcal M := \RR \mathbf e$ be the eigenspace associated to the eigenvalue $1$ and $\mathcal W$ the eigenspace associated to the remaining eigenvalues. One has
    \begin{enumerate}
        \item  $\mathbb{R}^N = \mathcal M \oplus \mathcal W$, both spaces being stable by $A$;
        \item There is a norm $\|\cdot\|$ on $\mathcal W$ and a distance $d$ on $\mathbb{R}^N$ such that for all $(x_\MM,x_\mathcal{W})\in \MM \oplus \mathcal W$,
        \begin{gather}
            d(x_\mathcal{W},\MM) = \| x_\mathcal{W}\|,  \label{eq:projection1}\\
            d(x_\MM+x_\mathcal{W}) = d(x_\mathcal{W},\MM), \label{eq:projection2} \\
            d( A(x_\MM+x_\mathcal{W}),\MM ) \leq (1 - N \underline{m}_A) d(x_\MM+x_\mathcal{W},\MM). \label{eq:projection 3}
        \end{gather}
        
    \end{enumerate}
\end{lemma}

    \begin{proof}
        \begin{enumerate}
        \item Classical decomposition result.
        \item We take $\|x_\mathcal{W}\| = \max_{i,j\in\{1,...,N\}}\{|x_{\mathcal{W},i} - x_{\mathcal{W},j}|\}$ for all $x_\mathcal{W} \in \WW$.
        The matrix $A$ being stochastic, one can use results on inhomogeneous Markov chains together (namely Theorem 3.1 in \cite{Seneta}) together with the upper bound  \eqref{alpha2<} on the eigenvalues of $A$  to have that
        \begin{equation*}
        \|A x_\WW\| \geq \left (1 - N \underline m_A \right ) \|x_\mathcal{W}\|.
        \end{equation*}
    
        Let $x = x_\MM + x_\mathcal{W}$ and $y = y_\MM + y_\mathcal{W}$ be in $\mathcal{M} \oplus \mathcal{W}$, we define $d$ as
        \begin{equation*}
        d(x,y) = \|x_\MM - y_\MM\|_2 + \|x_\WW - y_\WW\|,
        \end{equation*}
        with $\|\cdot\|_2$ being the euclidean norm on $\R^N$.
        \end{enumerate}
    \end{proof}

    \begin{rem}
      \label{rem:errorCuckerSmale}
        This result is inspired from the  Lemma 1 from \cite{CuckerSmale2004}, and has similar conclusions. Nevertheless, one has to bear in mind that with its set of hypotheses, the original result from \cite{CuckerSmale2004} is wrong. Indeed, being only stochastic and weakly irreducible is not enough to have the existence of a norm with the desired, precise contraction property (ii)-(c). For example,  
        \[
            \Lambda:= \begin{pmatrix} 0 & 1/2 & 1/2 \\ 3/4 & 0 & 1/4 \\ 1/8 & 7/8 & 0 \end{pmatrix} 
        \]
        is a stochastic, weakly irreducible matrix which has 2 complex eigenvalues, preventing the validity of (ii)-(c). Our set of hypotheses, as well as our new proof, prevent this.
    \end{rem}

\begin{corollary}
    If $A \in \mathcal{M}_N(\mathbb{R})$ is as in Lemma \ref{lem:contraction},  $ \mathcal M := \RR \mathbf e$ be the eigenspace associated to the eigenvalue $1$ and $\mathcal W$ the eigenspace associated to the remaining eigenvalues:
\begin{enumerate}
    \item  Then $\mathcal{F}^N = \mathcal M \oplus \mathcal W$, both spaces being stable by $A$;
    \item There is a norm $\|\cdot\|$ on $\mathcal W$ and a distance $d$ on $\mathbb{R}^N$ such that, for all $f_\MM\in\MM$ and for all $f_\WW\in \WW$,
    \begin{gather*}
        d(f_\WW,\MM) =\sqrt{ \int_\mathcal{E}\sum_{i=1}^l\| f_{\WW,i}(e)\|^2} \, \mathrm{d}e, \\
        d(f_\MM+f_\WW,\MM) = d(f_\WW,\MM), \\
        d(A(f_\MM+f_\WW),\MM) \leq \left (1 - N \underline{m}_A \right ) d(f_\MM+g_\WW,\MM).
    \end{gather*}
\end{enumerate}
\label{cor:function}
\end{corollary}

\begin{proof}
Let $f = f_\MM + f_\WW$ and $g = g_\MM + g_\WW$ in $\MM \oplus \WW$, we define as
\begin{equation*}
d(g,\MM) =\sqrt{ \int_\mathcal{E}\sum_{i=1}^l(\|f_{\MM,i}(e)-g_{\MM(e),i}\|^2_2+\|f_{W,i}(e)-g_{W,i}(e)\|^2)}\, \mathrm{d}e,
\end{equation*}
\end{proof}

\begin{corollary}
Let $A \in \mathcal{M}_N(\mathbb{R})$ as in Lemma \ref{lem:contraction}, and $f\in \mathcal{F}^N$ such that $f = f_\MM + f_\WW \in \MM \oplus \WW$. There exists a distance $d$ on $\mathcal{F}$ such that 
\begin{equation*}
    d(A f,f_\MM) \leq \left (1 - N \underline m_A\right ) d(f,f_\MM).
\end{equation*}
\label{lemma:contr}
\end{corollary}

\begin{proof}
Consequence of the previous corollary.
\end{proof}

We recall that the application $\mathcal{T} : \mathcal{F} \times \mathbb{N} \mapsto \mathcal{F}$ and the process $K^t_{\mathcal{T}}$ are defined by $\mathcal{T}(f,t) = \Lambda^t f$ and $K^t_{\mathcal{T}} = \mathcal{T}(k^t,t)$. This may be interpreted as an idealistic step of learning. Moreover, we assumed that $\tau = 0$, namely no individual learning occurs in the model.

\begin{theorem}
    \label{thm:contractionK}
    If $\Gamma > 0$, and the minimum credibility $c_{\min} >0$ is fixed, then for all times $t$, there exists a distance $d_{\Lambda^t}$ and $\underline m > 0$ independent on $t$ such that
\begin{equation*}
    d(K_{\mathcal{T}}^{t+1},k^{t}_\MM) \leq \left(1 - \underline m\right) d(k^t,k^{t}_\MM),
\end{equation*}  
where $k^t_\mathcal{M}$ is the projection of $k^t$ on $\mathcal{M}$.
\end{theorem}

\begin{proof}
Consequence of lemmas \ref{lemma:primitive} and \ref{lemma:lambda>eps}, and corollary \ref{lemma:contr}.
\end{proof}

We define an idealistic deterministic process $K^t$ by $K^{t+1} = \Lambda_t K^t$ and $K^0 = k^0$. Theorem \ref{thm:contractionK} implies that the idealistic, deterministic process converges to a common shared knowledge:

\begin{corollary}
    Under the hypotheses of Theorem \ref{thm:contractionK}, there exist $(K^0_\MM, K^0_\WW)$ in $\MM \otimes \WW$ such that $K^0 = K^0_\MM + K^0_\WW $. Then,
    \begin{equation*}
        \lim_{t\rightarrow+\infty}(K^t,K^0_\MM) = 0.
    \end{equation*}
\end{corollary}

\subsection{Learning theory}

The results presented in this part are inspired by \cite{CuckerSmale2001}. This article deals with the inference of functions to match with random samples. In our case, the functions are knowledge-like functions and samples come from social learning and individual learning. Nevertheless, our theoretical results shall only deal  with the case where individuals learn from social sources only (under the hypothesis $\tau = 0$). The results of this theory implies the convergence of the process $\Delta k^t$ with high probability.

\subsubsection{Sample error}

We study the learning process from random samples governed by the probability measure $\rho$ on $\mathcal{Z} = \mathcal{E} \times \mathcal{C}$. We recall that $\mathcal E$ is a compact subset of  $\mathbb{R}^n$, and that  $\mathcal{C}$ is a subset of an euclidean space containing zero. 



 

\begin{definition}
We define the \emph{least square error} of f as
 \begin{equation*}
    \varepsilon(f) = \int_\mathcal{Z} \|f(e) - c\|_\mathcal{C}^2\, \mathrm{d}\rho,
\end{equation*}
for $f : \mathcal{E} \to \mathcal{C}$,  where $\|\cdot\|_\mathcal{C}$ is a norm on $\mathcal{C}$ associated with the inner product $\langle\cdot,\cdot\rangle_\mathcal{C}$ of the ambient euclidean space $\mathbb{E}^l$.
\end{definition}


\begin{prop}
For every $f : \mathcal{E} \to \mathcal{C}$,
\begin{equation*}
    \varepsilon(f) = \int_\mathcal{E} \|f(e) - f_\rho(e)\|_\mathcal{C}^2 \, \mathrm{d}\rho_\mathcal{E} + \varepsilon(f_\rho),
\end{equation*}
where  $f_\rho(e) := \int_\mathcal{C} c \, \mathrm{d}\rho(c|e)$, for any $e \in \mathcal E$.
\label{prop:sample}
\end{prop}

\begin{proof} Adding and subtracting $f_\rho$ yields
    \begin{align*}
        \varepsilon(f) 
        &= \int_\mathcal{Z} \|f(e) - f_\rho(e)\|^2\, \mathrm{d}\rho + \int_\mathcal{Z} \|f_\rho(e) - c\|^2\, \mathrm{d}\rho + 2 \int_\mathcal{Z} \langle f(e) - f_\rho(e) , f_\rho(e) - c \rangle_\mathcal{C}\, \mathrm{d}\rho\\
        &= A + \varepsilon(f_\rho) + 2B.
    \end{align*}
    We have \begin{align*}
        A& = \int_{\mathcal{E}}\left(\int_\mathcal{C}\|f(e) - f_\rho(e)\|^2 \, \mathrm{d}\rho(c|e)\right) \, \mathrm{d}\rho_\mathcal{E}\\
        & = \int_\mathcal{E}\left(\|f(e) - f_\rho(e)\|^2 \int_\mathcal{C}\, \mathrm{d}\rho(c|e)\right) \, \mathrm{d}\rho_\mathcal{E}\\
        & = \int_\mathcal{E} \|f(e) - f_\rho(e)\|^2 \, \mathrm{d}\rho_\mathcal{E}.\\
    \end{align*}
    For the second term we have 
    \begin{align*}
        B& =  \int_\mathcal{E}\left(\int_\mathcal{C}\langle f(e) - f_\rho(e) , f_\rho(e) - c \rangle_\mathcal{C} \, \mathrm{d}\rho(c|e)\right) \, \mathrm{d}\rho_\mathcal{E}\\
        & =  \int_\mathcal{E}\langle f(e) - f_\rho(e) ,f_\rho(e) - \int_\mathcal{C}c \, \mathrm{d}\rho(c|e) \rangle_\mathcal{C}  \, \mathrm{d}\rho_\mathcal{E}\\
        & =  \int_\mathcal{E}\langle f(e) - f_\rho(e) ,f_\rho(e) - f_\rho(e) \rangle  \, \mathrm{d}\rho_\mathcal{E} = 0.
    \end{align*}
\end{proof}

As a consequence of the proposition \ref{prop:sample}, the \emph{regression function} $f_\rho$ minimizes the mean square error $\varepsilon$. 

\begin{definition}
    Let $f_\mathcal{F}$ be the \emph{target function} that minimizes $\varepsilon$:
\begin{equation*}
    f_\mathcal{F} \in \argmin_{f\in\mathcal{F}} \varepsilon(f).
\end{equation*}
\end{definition}

During the learning phase, the probability measure $\rho$ is not assumed to be known. The learning process is a minimisation procedure on a sample $S = ((e_1,c_1),...,(e_m,c_m))$, $m\in\mathbb{N}^*$. 

\begin{definition}
    We define the \emph{empirical error} $\varepsilon_S$ of f on the sample S by 
    \begin{equation*}
        \varepsilon_S(f) = \frac{1}{m}\sum_{i=1}^m \|f(e_i)-c_i\|^2\, \mathrm{d}\rho,
    \end{equation*}
    and $f_S$ the \emph{empirical target} function, namely a minimizer of $\varepsilon_S$:
    \begin{equation*}
        f_S \in \argmin_{f\in\mathcal{F}} \varepsilon_S(f).
    \end{equation*}
\end{definition}

This minimizer is of course not unique. Nevertheless, when the size $m$ of the sample is large enough, the empirical target function will approximate the target function. More precisely, one has the following classical concentration inequality from \cite{Pollard84}: 

\begin{prop}
We assume that:
\begin{enumerate}
    \item $\mathcal{F}$ is a compact and convex set;
    \item there exists $M\in\mathbb{R}^*_+$, such that for all $f\in \mathcal{F}, \|f(e)-c\|_\mathcal{C}\leq M$ almost everywhere;
    \item $\rho$ is a probability measure on $\mathcal{Z}$.
\end{enumerate}
Then for all $\eta > 0$,

\begin{equation*}
    \mathrm{Prob}\left\{\int_\mathcal{E} \|f_S(e) - f_\mathcal{F}(e)\|^2_\mathcal{C}\, \mathrm{d}\rho_\mathcal{E}\leq \eta\right\} \geq 1 - \mathcal{N}(\mathcal{F},\frac{\eta}{24M})e^{\frac{-m\eta}{288 M^2}}
\end{equation*}

where $\mathcal{N}(\mathcal F,s)$ is the so-called \emph{covering number}, namely the minimal $l \in \mathbb{N}$ such that there exist $l$ disks in $\mathcal F$ with radius $s$ covering $\mathcal F$.
    Since $\mathcal F$ is compact,  this number is finite.
\label{prop:proba}
\end{prop}

We can now get back to our model. We recall that the probability measure $\rho^{i,t}$ allows to draw the sample for the learning of the individual $i$ at time $t$, and that $\tau = 0$. The probability measure $\rho^{i,t}$ then depends only on social learning. During the learning phase of our model we have :
\begin{align*}
    f_{\rho^{i,t}}(e) &= \int_\mathcal{C} c \, \mathrm{d}\rho(c|e)\\
    &= \sum_{j=1}^N \Lambda^t_{ij} \int_\mathcal{C} c \mathds{1}_{k_j(e)=c}\, \mathrm{d}c\\
    &= \sum_{j=1}^N \Lambda^t_{ij} k_j(e) \, \mathrm{d}c.\\
\end{align*}

Since $\mathcal{F}$ is convex, $f_{\rho^{i,t}} \in \mathcal{F}$. If $\mathcal E$ is finite, or in the other case, if $\mathcal{F}$ is a set a continuous functions, we have $f^{i,t}_\mathcal{F} = f_{\rho^{i,t}}$ with $f^{i,t}_\mathcal{F}$ being the minimiser of the error $\varepsilon$ with $\rho = \rho^{i,t}$.

\subsection{Main result} 
{Combining our results on the one-step idealistic process, together with the ones on learning theory, we are able to study the convergence of the full process $k^t$ with high probability.}

We recall that $\MM_\mathcal{F} = \{(k,...,k),k \in \mathcal{F}\}$. 

\begin{theorem} 
We suppose that $\tau =0$, $\Gamma > 0$, $c_{\min} >0$, and $\mathcal{F}$ is compact and convex.
There exist some constants $\alpha_* < 1$, $M > 0$, and $A\geq0$ such that for each $0<\delta<1$, and $t \geq 0$, if the sample size $m\geq m_t = m_t\left(\MM_\mathcal{F}, k^0, M, \alpha_*, \delta\right)$, then
\begin{equation*}
    d(k^t, \MM_\mathcal{F}) \leq A \, \alpha^t_* \,d(k^0, \MM_\mathcal{F}),
\end{equation*}
with confidence at least $1 - \delta.$
\label{th:main}
\end{theorem}

\begin{proof}

Let $d$ be the distance defined in corollary \ref{cor:function}.

We recall that the application $\mathcal{T} : \mathcal{F} \times \mathbb{N} \mapsto \mathcal{F}$ is defined by $\mathcal{T}(f,t) = \Lambda^t f$. Let $K^t_{\mathcal{T}} = \mathcal{T}(k^t,t)$. Notice in particular that the process $K^t_{\mathcal{T}}$ is different from $k^t$.
By the triangle inequality we have, using \eqref{eq:decompkt} and \eqref{def:idealProc}, that
\begin{equation*}
    d(k^t,\MM_\mathcal{F}) \leq d(k^t,K_{\mathcal{T}}^t)  + d(K_\mathcal{T}^t,\MM_\mathcal{F}).
\end{equation*}
The contractivity of the second term is yielded by Theorem \ref{thm:contractionK}: there exists $\alpha_t <1 $ such that
\begin{equation*}
    d(K_{\mathcal{T}}^t,\MM_\mathcal{F}) \leq \alpha_t d(k_{\mathcal{T}}^{t-1},\MM_\mathcal{F}).
\end{equation*}

Now, we need to estimate the other term. We recall that $\mathcal E$ is compact in $\RR^N$. By the compactness of $\mathcal{E}$ and $\mathcal{F}$ we have that
\begin{equation*}
    \sup_{f\in\mathcal{F}, e\in \mathcal{E}} \|f(e)\|_\mathcal{C} < \infty.
\end{equation*} 
In particular, there exists $M>0$ such that
\begin{equation*}
    \max_{(e,c)\in \mathcal{E}\times \mathcal{C}, f \in \mathcal{F}} \| f(e) - c \|_\mathcal{C} \leq M
\end{equation*}
Using Proposition \ref{prop:proba}, for each $\eta > 0$ and $i\in \{1,...,N\}$,

\begin{equation}
    \label{eq:concentrationKT}
    \mathrm{Prob}\left\{\int_\mathcal{E} \|k_i^t(e) - K_{\mathcal{T},i}^t(e)\|^2_\mathcal{C}\, \mathrm{d}\rho_i\leq \eta\right\} \geq 1 - \mathcal{N}(\mathcal{F},\frac{\eta}{24M})e^{\frac{-m\eta}{288 M^2}}.
\end{equation}

Let us now define the norm $\|\cdot\|_{\mathcal{F}^N_\rho}$ on $\mathcal{F}^N$ by:
\begin{equation*}
    \|F\|_{\mathcal{F}^N_\rho} = \sqrt{\sum_{i=1}^N\int_\mathcal{E} \|F_i(e)\|^2_\mathcal{C}\, \mathrm{d}\rho_i},\qquad \text{for all }F \in \mathcal{F}^N.
\end{equation*}
For all $1\leq i \leq N$, one has:
\begin{equation}
    \label{eq:ineqNeta}
    \int_\mathcal{E} \|k_i^t(e) - K_{\mathcal{T},i}^t(e)\|^2_\mathcal{C}\, \mathrm{d}\rho_i\leq \eta \implies \sum_{i=1}^N\int_\mathcal{E} \|k_i^t(e) - K_{\mathcal{T},i}^t(e)\|^2_\mathcal{C}\, \mathrm{d}\rho_i\leq N\eta.
\end{equation}
In particular,
\begin{equation*}
    \cup_{i=1}^N\left\{\int_\mathcal{E} \|k_i^t(e) - K_{\mathcal{T},i}^t(e)\|^2_\mathcal{C}\, \mathrm{d}\rho_i\leq \eta\right\} \subset \{\|k^t - K_{\mathcal{T}}^t\|_{\mathcal{F}^N_\rho}\leq N\eta\}.
\end{equation*}
Thus, gathering \eqref{eq:concentrationKT} and \eqref{eq:ineqNeta}, and using the convexity of the exponential,
\begin{align}
    \mathrm{Prob}\left\{\|k^t - K_{\mathcal{T}}^t\|_{\mathcal{F}^N_\rho}\leq N\eta\}\right\} &\geq \mathrm{Prob}\left\{\cup_{i=1}^N\{\int_\mathcal{E} \|k_i^t(e) - K_{\mathcal{T},i}^t(e)\|^2_\mathcal{C}\, \mathrm{d}\rho_i\leq \eta\}\right\} \notag \\
    &\geq (1 - \mathcal{N}(\mathcal{F},\frac{\eta}{24M})e^{\frac{-m\eta}{288 M^2}})^N \notag \\
    &\geq 1 - N\mathcal{N}(\mathcal{F},\frac{\eta}{24M})e^{\frac{-m\eta}{288 M^2}}. \label{eq:probaBiggerThan}
\end{align}

Let $d_{\mathcal{F}^N_\rho}$ be the distance on $\mathcal{F}^N$ defined by the norm $\|\cdot\|_{\mathcal{F}^N_\rho}$. For all f and g in $\mathcal{F}^N$, we have:
\begin{align*}
    d_{\mathcal{F}^N_\rho}(f,g)^2 = {\sum_{i=1}^N\int_\mathcal{E} \|f_i(e)-g_i(e)\|^2_\mathcal{C}\, \mathrm{d}\rho_i} = {\int_\mathcal{E} \|f(e)-g(e)\|_A^2\, \mathrm{d}\rho_i},
\end{align*}
with 
\begin{equation*}
    \|x\|_A^2 = \sum_{i=1}^N \|x_i\|^2_\mathcal{C} \quad \forall x \in (\mathbb{R}^l)^N.
\end{equation*}
For all f and g in $\mathcal{F}^N$ we also have:
\begin{align*}
    d(f,g)^2 &= {\int_\mathcal{E} \sum_{i=1}^l \|((f_1(e)-g_1(e))_l, \ldots,(f_N(e)-g_N(e))_l)\|^2_\mathcal{C}\, \mathrm{d}\rho_i}\\
    &= {\int_\mathcal{E} \|f(e)-g(e)\|_B^2\, \mathrm{d}\rho_i},
\end{align*}
with 
\begin{equation*}
    \|x\|_B^2 = \sum_{i=1}^l \|((x_1(e))_l, \ldots,(f_N(e)-g_N(e))_l)\|^2, \quad \forall x \in (\mathbb{R}^l)^N.
\end{equation*}

All the norm being equivalent on $(\mathbb{R}^l)^N$, there exist $C_A'$ and $C_A$ such that
\begin{equation*}
    C_A' \|x\|_B \leq \|x\|_A \leq C_A \|x\|_B, \quad \forall x\in (\mathbb{R}^l)^N.
\end{equation*}
Hence,
\begin{equation*}
    C_A' d(f,g) \leq d_{\mathcal{F}^N(f,g)} \leq C_A d(f,g), \quad \forall lf,g\in \mathcal{F}.
\end{equation*}
Using \eqref{eq:probaBiggerThan} with confidence at least $1 - N\mathcal{N}(\mathcal{F},\frac{\eta}{24M})e^{\frac{-m\eta}{288 M^2}}$, we finally have
\begin{equation*}
    d(k^t,\MM_\mathcal{F}) \leq C_A \sqrt{N\eta} + \alpha_td(k^{t-1},\MM_\mathcal{F}).
\end{equation*}

Iterating on the discrete times, one has with confidence at least $1 - t N\mathcal{N}(\mathcal{F},\frac{\eta}{24M})e^{\frac{-m\eta}{288 M^2}}$ that
\begin{equation*}
    d(k^t,\MM_\mathcal{F}) \leq C_A \sqrt{N\eta}\left(\sum_{i=0}^{t-1} \prod_{j=1}^i\alpha_j \right) + \prod_{i=1}^t\alpha_i d(k^{0},\MM_\mathcal{F}).
\end{equation*}
Let $\alpha_* = \max_{i = 0,\ldots, N} \max_{t} \alpha_i(t)$. According to Lemma \ref{lemma:upperBound}, one has $\alpha_* < 1$, yielding that 
\begin{align*}
    d(k^t,\MM_\mathcal{F}) &\leq C_A \sqrt{N\eta}\left(1 + \alpha_ + \ldots + \alpha_*^{t-1}\right) + \alpha_*^t d(k^{0},\MM_\mathcal{F}).\\
    &\leq \frac{C_A}{1-\alpha_*} \sqrt{N\eta} + \alpha_*^t d(k^{0},\MM_\mathcal{F}).
\end{align*}
Thus, for any $0<\delta<1$, choosing the parameter $m$ such that
\begin{equation}
    \label{eq:mBiggerThan}
    m \geq \frac{288M^2}{\eta}\left(\ln{\left(\frac{t N\mathcal{N}(\mathcal{F},\frac{\eta}{24M}}{\delta}\right)}\right),
\end{equation}
yields with confidence at least $1-\delta$ that
\begin{equation*}
    d(k^t,\MM_\mathcal{F}) \leq \frac{C_A}{1-\alpha_*} \sqrt{N\eta} + \alpha_*^t d(k^{0},\MM_\mathcal{F}).
\end{equation*}
Taking $\eta = {\alpha_*^{2t}d(k^0,\MM_\mathcal{F})^2}/{N}$ finishes the proof.
\end{proof}

\begin{rem}
 When time $t$ goes to infinity, so does the number of sample $m_t$ needed for the convergence in Theorem \ref{th:main} to occur. 

Indeed, using Section 7.1 of \cite{CuckerSmale2004}, there exists $C_\mathcal{F}>0$ and $a>0$ such that 
\begin{equation*}
    \ln{\mathcal{N}(\mathcal{F},\epsilon)}\leq C_\mathcal{F}\left(\frac{1}{\epsilon}\right)^a.
\end{equation*}
Plugging this into \eqref{eq:mBiggerThan} yields that 
\begin{align*}
    m_t &\leq \frac{288NM^2}{(1-\alpha_*)^2\alpha_*^{2t}d(k^0,\MM_\mathcal{F})^2} \ln{\left(tNC_\mathcal{F}\left(\frac{24NM}{(1-\alpha_*)^2\alpha_*^{2t}d(k^0,\MM_\mathcal{F})^2}\right)^a+\ln{\left(\frac{1}{\delta}\right)}\right)}\\
   &\leq \frac{288NM^2}{(1-\alpha_*)^2\alpha_*^{2t}d(k^0,\MM_\mathcal{F})^2} \left(\ln{(tN)}+C_\mathcal{F}\left(\frac{24NM}{(1-\alpha_*)^2\alpha_*^{2t}d(k^0,\MM_\mathcal{F})^2}\right)^a+\ln{\left(\frac{1}{\delta}\right)}\right)\\
\end{align*}
Choosing appropriately $\delta$ as a function of $t$, one can show that $f^t$ tends to $\MM_\mathcal{F}$ almost surely. 
One can then define the minimal sampling size $m(t)$ by
\begin{align*}
    m(t) & = \frac{288NM^2}{(1-\alpha_*)^2\alpha_*^{2t}d(k^0,\MM_\mathcal{F})^2} \left(\ln{(t^2N)}+C_\mathcal{F}\left(\frac{24NM}{(1-\alpha_*)^2\alpha_*^{2t}d(k^0,\MM_\mathcal{F})^2}\right)^a\right),
\end{align*}
which tends to $ +\infty$ when $t  \to +\infty$.

\end{rem}

\begin{corollary}
    Let $f_m^t$ be the process at the  $t$ when the size of the sample in the dynamics is $m$. One has that
    \begin{equation*}
        \sup_{\epsilon>0} \lim_{t\rightarrow\infty} \mathrm{Prob}\left\{d(f^t_{m(t)},\MM_\mathcal{F})\leq \epsilon\right\} = 1
    \end{equation*}
\end{corollary}

\begin{proof}
Let $\epsilon>0$. For all $t$ big enough, one has
\begin{equation*}
    A \alpha_*^td(f^0_{m(t)},\MM_\mathcal{F}) < \epsilon.
\end{equation*}
Taking $\delta = \frac{1}{t}$ yields 
\begin{equation*}
    \mathrm{Prob}\left\{d(f^t_{m(t)},\MM_\mathcal{F})\leq \epsilon\right\} \geq 1 - \frac{1}{t}.
\end{equation*}
\end{proof}

\section{Numerical simulations}
\label{sec:numRes}
%
  Let us now both illustrate the mathematical results of the paper, such as Theorem \ref{th:main}, and show that some generalizations also hold when individual learning is possible ($\tau>0$). Individual learning allows innovations, new experiences and observations, and original conceptualizations  which make possible the evolution of knowledge for both the individuals and the population. By analogy, individual learning plays the same role for the evolution of knowledge than genetic mutations for the biological evolution of species \cite{Mesoudi2011}.
  
    \subsection{Illustration of the main theorem}
{Our model aims to be used by theoretical anthropologists. To show its usefulness, we illustrate the results of our main theorem in specific cases.}

\medskip

\paragraph{Test 1. Impact of self-inertia.}
   As a first numerical test, we aim to illustrate Theorem \ref{th:main}.  Let $\mathcal{E} = \{1,...,5\}$ and $\mathcal{C} = [-10,10]$. We consider a relationship between two individuals (labeled $1$ and $2$). The structure matrix \ref{def:Gamma} is given by
\begin{equation*}
    \Gamma = \begin{pmatrix}
    \alpha & 1 - \alpha \\
    1 - \alpha & \alpha
    \end{pmatrix},
\end{equation*}
where $\alpha\in[0,1].$

The parameter $\alpha$ can be interpreted as cognitive (see Remark \ref{rem:inertia}) or self-inertia. The higher the $\alpha$, the less individuals' knowledge-like functions change along the dynamics. As likelihood landscape (\ref{def:likelihoodLandscape}), we take $L(e,c) = 1$ for all $e \in \mathcal{E}$ and $c\in \mathcal{C}\backslash\{0\}$, and take $c_{\min} = 0.1$. 
We define $\mathcal{F}$ as the set of continuous functions from $\mathcal{E}$ to $\mathcal{C}$ so $\mathcal{F}$ is convex. As $\mathcal{E}$ contains a finite number of elements and $\mathcal{C}$ is compact, then $\mathcal{F}$ is compact. We set $\tau = 0$ so the dynamics is only driven by social learning.

When $\alpha$ varies in $(0,1)$, all the hypotheses of Theorem \ref{th:main} are met (even though, strictly speaking, we do not illustrate exactly the theorem because we cannot compute $m_t$). Our numerical simulations show that according to this result the population converges to a common shared knowledge.

At the initial state, the knowledge-like functions of individuals $1$ and $2$ are   $k_1^0$ and $k_2^0$, respectively, given by 
\begin{equation*}
    k_1^0(e) = 2 \quad \forall e\in \mathcal{E}, \qquad
    k_2^0(e) = 6 \quad \forall e \in \mathcal{E}.
\end{equation*}

Let $d$ be the distance defined in  Corollary \ref{cor:function}.
By using numerical simulations we follow the evolution of $d(k^t,\MM_\mathcal{F})$ through time for different values of the parameter $\alpha$. We ran $100$ simulation replicates. The mean dynamics is presented in Figure \ref{fig:d_self_inertia}.

\begin{figure}[ht]
\centering
\input{mytikz_self_inertia_fig_1_log.tex}
\caption{\label{fig:d_self_inertia}
\textbf{Test 1.} Evolution of the distance between $k^t$ and the set $\mathcal{F}$ through time.}
\end{figure}
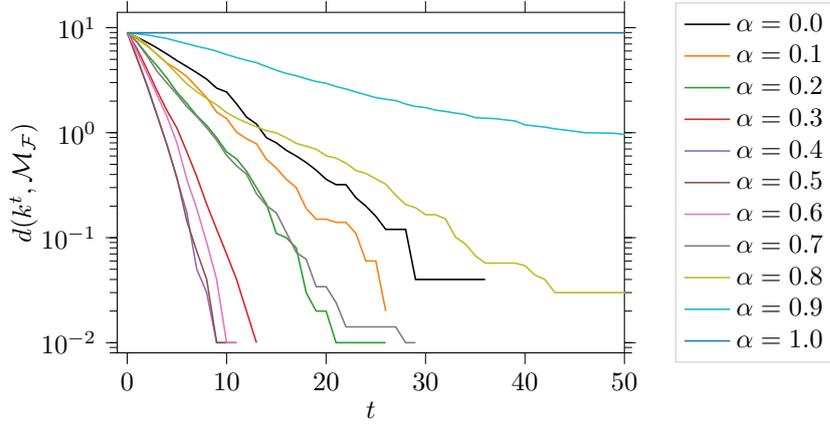

When $\alpha \neq 1$ the population rapidly converges  to a common shared knowledge (Fig. \ref{fig:d_self_inertia}) as predicted by Theorem \ref{th:main}. We notice that this convergence is exponential, as expected given that the process is driven by an inhomogeneous Markov chain.
We notice that the convergence is faster when $\alpha = 0.5$.

When $\alpha = 1$ the matrix does not respect the hypothesis of Theorem \ref{th:main} since $\gamma_{12}=\gamma_{21}=0$. It corresponds to the case where the individuals do not communicate with each other. Thus individuals knowledge-like functions do not vary through time, and the process does not converge towards a common shared knowledge.

\medskip
\paragraph{Test 2. A professor and its audience.}

Now let us consider a population of 5 individuals: 1 professor ($1$) and 4 students ($2$, $3$, $4$, $5$). We keep the same setting as previously, namely no individual learning ($\tau = 0$), since it is a purely teaching situation. 

Let the structure matrix be
\begin{equation*}
        \Gamma = \begin{pmatrix}1&0.01&0.01&0.01&0.01\\1&0.1&0.1&0.1&0.1\\1&0.1&0.1&0.1&0.1\\1&0.1&0.1&0.1&0.1\\1&0.1&0.1&0.1&0.1\end{pmatrix}.
    \end{equation*}
At the initial state, knowledge-like functions are defined by:
\begin{equation*}
    k_1^0(e) = 5 \quad \forall e\in \mathcal{E},
\end{equation*}
and 
\begin{equation*}
    k_i^0(e) = 1 \quad \forall e\in \mathcal{E}, \forall i \in \{2,\ldots,5\},
\end{equation*}
so at the initial time, all students have the same knowledge.

We consider two cases. First, the likelihood landscape is assumed constant. Second, it is considered concave assuming the concept $c$ is fixed: we set for all $(e,c) \in \mathcal{E} \times\mathcal{C}\backslash\{0\}$
\begin{equation*}
L(e,c) = e^{\frac{(e-6)^2}{10}},
\end{equation*}
so the professor has a knowledge-like function that is more likely than that one of her students.

We call $k_{eq}$ the common shared knowledge at the equilibrium. We define $\Delta_i$ as the distance between the initial knowledge of individual $i$ and the common shared knowledge. We have
\begin{equation*}
    \Delta_i = d_\mathcal{C}(k_i^0,k_{eq}),
\end{equation*}
with $d_\mathcal{C}$ the distance induced by the inner product on $\mathcal C$.

We ran 100 numerical simulations as previously. Results are shown in Figures \ref{fig:d_professor} and \ref{fig:Delta_professor}.

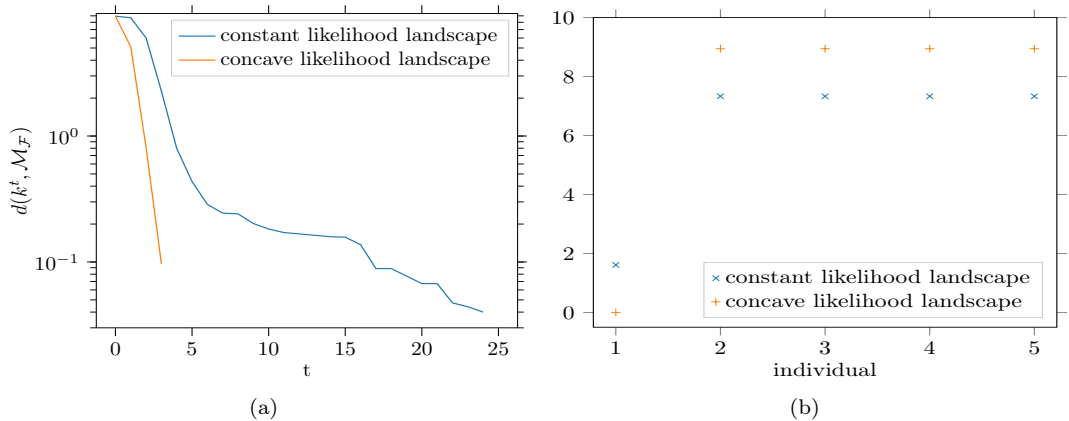
\begin{figure}[ht]
    \centering
    \subfigure[]{\resizebox{7cm}{5cm}{\input{mytikz_professor_fig_1.tex}}\label{fig:d_professor}}
    \subfigure[]{\resizebox{7cm}{5cm}{\input{mytikz_professor_fig_2.tex}}\label{fig:Delta_professor}}
    \caption{\textbf{Test 2.}(a) Evolution of the distance between $k^t$ and the set $\mathcal{F}$ through time. (b) $\Delta_i$ for each individual $i$ in the population. The blue crosses represent the case where the likelihood is constant, and the orange pluses show the case where the likelihood is concave.}
\end{figure}

Figure \ref{fig:d_professor} shows the evolution of the distance to space $\mathcal{F}$ with time. In both cases, whether the likelihood landscape is fixed or concave, the population rapidly converges to a common shared knowledge. When the likelihood landscape is concave, the professor has a strong influence on her students and the convergence to a common shared knowledge is faster. 

Figure \ref{fig:Delta_professor} shows the values of $\Delta_\cdot$. In both cases, the common shared knowledge is close to the professor's initial one. This common shared knowledge is farther from the students' initial knowledge than from the professor's. When the likelihood landscape favors the professor influence, the common shared knowledge is closer to the initial professor knowledge.

    \subsection{Creation of knowledge}
    
  We now consider the case where individual learning is present, namely $\tau >0$. Although we couldn't prove a convergence result for this case, we can still use numerical approaches when the parameters of the model do not allow analytical resolution.

  \medskip
  \paragraph{Test 3. Creation of knowledge among interacting individuals.} In this part we set $\mathcal E =\{1,...,25\}$ and $\mathcal C = \mathbb{R}$. We define the likelihood landscape as 
 \begin{equation*}
        L(e, c) = \left\{ \begin{array}{ll}
		\frac{1}{2} & \mbox{if } c = 0,\\
		\exp{-(x-1)^2}	 & \mbox{otherwise},\\
	\end{array}
    \right.
\end{equation*}
such that the function $1_{\mathcal{F}}$ is defined as:
\begin{equation*}
    \forall e \in E, 1_{\mathcal{F}}(e) = 1,
\end{equation*}
is the most likely function. We consider a population of ten individuals and we set a initial state where $K^0 = (0_{\mathcal{F}}, ...,0_{\mathcal{F}})$. Let $\Gamma$ be the square matrix of size $N$ full of $1$. Thus at the initial state,  individuals are "newborn", that is, they do have not conceptualized any experiences. We investigate convergence of knowledge  towards the function $1_{\mathcal{F}}$ and its dynamics by simulation runs.

We define the relative entropy (RE) of the population as 
\begin{equation}
    \text{RE(t)} = - \frac{1}{N} \sum_{i = 1}^N d_{\mathcal{F}}(k_i^t,1_{\mathcal{F}}),
\end{equation}
with for all $f$, $g$ in $\mathcal{F}$,
\begin{equation}
    d_{\mathcal{F}}(f,g) = \sqrt{\sum_{e=1}^{N_E} \frac{(f(e)-g(e))^2}{N_E}}.
\end{equation}
When every individual in the population has $1_{\mathcal{F}}$ as knowledge-like function, the relative entropy is maximal and equals 0. We use the relative entropy as a measure of knowledge in the population \textit{i.e.} the higher the relative entropy, the more likely the individuals' knowledge. This allows us to quantify the effect of parameters on the evolution of knowledge.
Figure \ref{fig:RE_ev}  shows that the relative entropy increases with time.

\begin{figure}
    \centering
    \resizebox{9cm}{5cm}{\input{mytikz_creation_RE.tex}}
    \caption{\label{fig:RE_ev} \textbf{Test 3.} Evolution of the relative entropy through time.}
\end{figure}
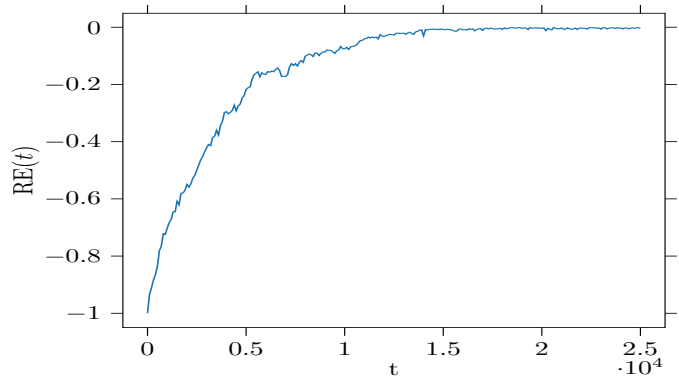

In simulations, the individual learning rate $\tau=0.02$. Individual learning results in new experiences and observations, while social learning promotes the spread of adequate conceptualizations. The combined effect of individual and social learning allows the population to evolve towards better solutions (Figure \ref{fig:knowledge_crea}). 

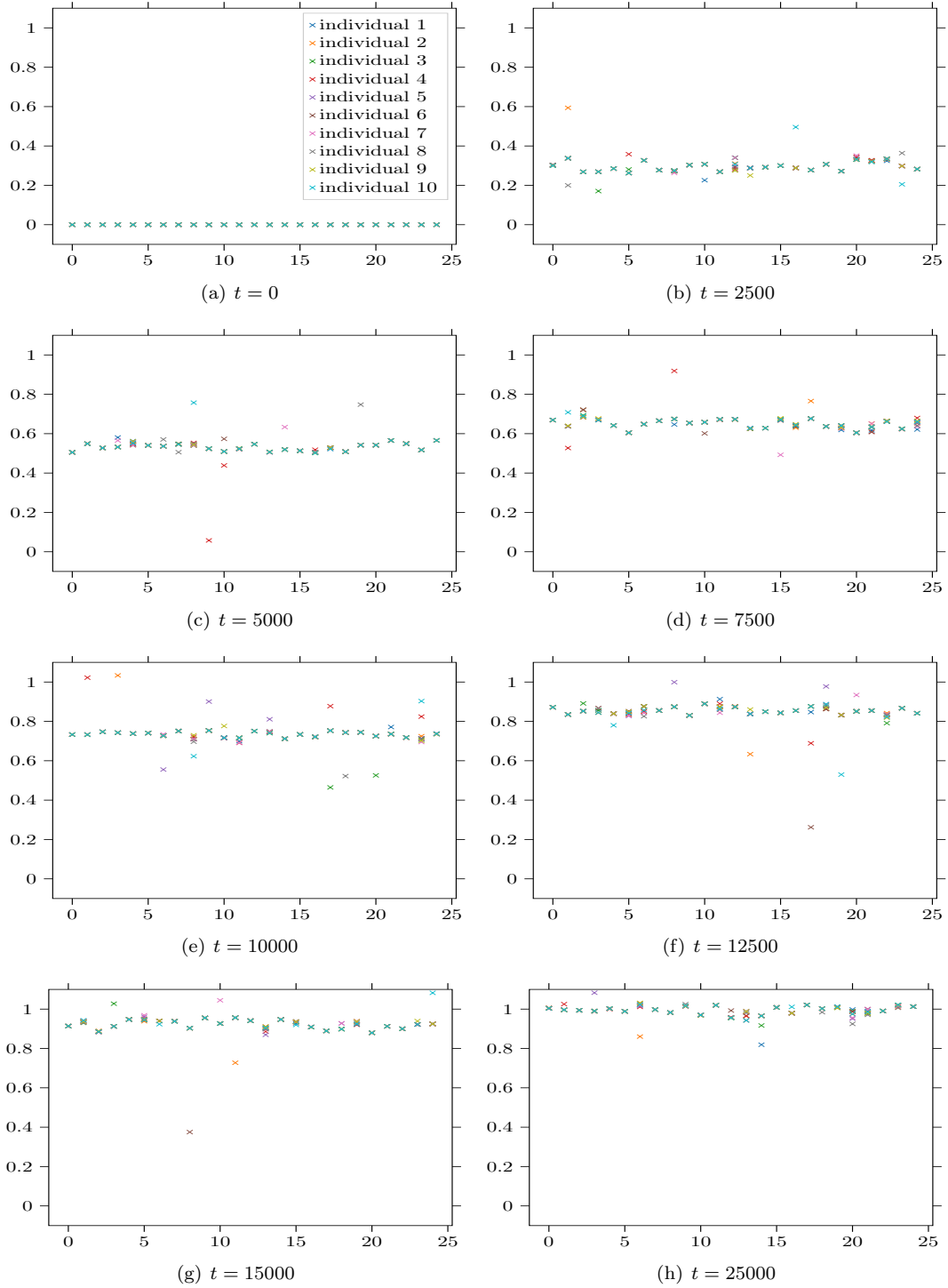
\begin{figure}
    \centering
    \subfigure[$t=0$]{\resizebox{7cm}{4cm}{\input{mytikz_creation_t_0.tex}}\label{fig:knowledge_crea_0}}
    \subfigure[$t=2500$]{\resizebox{7cm}{4cm}{\input{mytikz_creation_t_2500.tex}}\label{fig:knowledge_crea_2500}}
    \subfigure[$t=5000$]{\resizebox{7cm}{4cm}{\input{mytikz_creation_t_5000.tex}}\label{fig:knowledge_crea_5000}}
    \subfigure[$t=7500$]{\resizebox{7cm}{4cm}{\input{mytikz_creation_t_7500.tex}}\label{fig:knowledge_crea_7500}}
    \subfigure[$t=10000$]{\resizebox{7cm}{4cm}{\input{mytikz_creation_t_10000.tex}}\label{fig:knowledge_crea_10000}}
    \subfigure[$t=12500$]{\resizebox{7cm}{4cm}{\input{mytikz_creation_t_12500.tex}}\label{fig:knowledge_crea_12500}}
    \subfigure[$t=15000$]{\resizebox{7cm}{4cm}{\input{mytikz_creation_t_15000.tex}}\label{fig:knowledge_crea_15000}}
    \subfigure[$t=25000$]{\resizebox{7cm}{4cm}{\input{mytikz_creation_t_25000.tex}}\label{fig:knowledge_crea_25000}}
    \caption{\label{fig:knowledge_crea}\textbf{Test 3.} Knowledge-like functions for every individuals in the population at different given times.}
\end{figure}

    \subsection{Comparison with a language model}

Cucker, Smale and Zhou developed a model to describe the evolution of language \cite{CuckerSmale2004}. Our work is stongly inspired by their work. For our purpose, we needed to substantially modify this model by introducing the credibility matrix and individual learning. However, interpretation of the variables of the model is different: in their model a language-like function is a function from a space of objects to a space of signals. As in our case, they proved the convergence of the languages of different individuals to a common shared language, although under different hypotheses (see also Remark \ref{rem:errorCuckerSmale}). 

In their model, influences between individuals do not vary through time. In reality, we expect  influences between individuals to be dynamic, and that is why we introduced the credibility matrix which changes at each time step and modifies the interactions within the population. 

\medskip
\paragraph{Test 4. On the evolution of language.}
We modified our numerical method in order to simulate the model of language evolution developed in \cite{CuckerSmale2004}. We consider two different linguistic communities of two individuals with few interactions.  We take $\mathcal{E} = \{1,...,5\}$ and $\mathcal{C} = [-10,10]$. The individuals of the first and the second communities have the language-like function $k_1$ and $k_2$, respectively. Where $k_1$ and $k_2$ correspond to two different languages. This language-like function is defined as
\begin{equation*}
\forall e \in \mathcal{E}, k_1(e) = 5,
\end{equation*}
and 
\begin{equation*}
\forall e \in \mathcal{E}, k_2(e) = 7.
\end{equation*}
We take 
\begin{equation*}
\Gamma = \begin{pmatrix}1&1&0.01&0.01\\
1&1&0.01&0.01\\
0.01&0.01&1&1\\
0.01&0.01&1&1\\
\end{pmatrix},
\end{equation*}
so the two linguistic communities hardly interact. Numerical simulations show that the two communities converge to a common shared language: Figure \ref{fig:language} shows that the  distance between the process $k^t$ and the set $\mathcal{F}$ tends to 0.

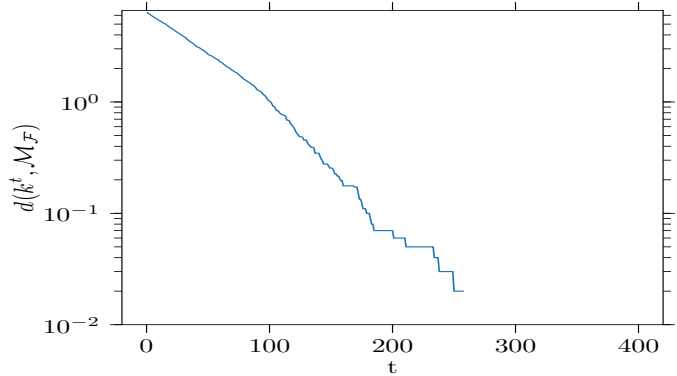
\begin{figure}[ht]
    \centering
    \resizebox{9cm}{5cm}{\input{mytikz_language_fig_1.tex}}
    \caption{\textbf{Test 4.} Evolution of the distance between $k^t$ and the set $\mathcal{F}$ through time.}
    \label{fig:language}
\end{figure}

\section{Conclusion}

The aim of this work was to develop a more general mathematical model of knowledge evolution than the existing ones e.g. \cite{Henrich2014, Powell2009}. Existing models have been widely used to investigate the impact of population size on the evolution of knowledge. However, they rely on strong assumptions and omit important aspects of social dynamics. Here, we developed a hybrid model, between an individual based stochastic model and a learning algorithm, that relaxes hypotheses and incorporates various forms of social interaction dynamics.

Analytical results show that interacting individuals converge with high probability towards a common shared knowledge, when no innovation occurs (i.e. no individual learning). Numerical simulations show that these results hold when individuals combine individual and social learning and that conceptualizations that appropriately reflect the structure of the world emerge across time. 
This model can be used to investigate knowledge evolution in hierarchically or spatially structured populations of variable sizes.   


\bibliographystyle{acm}
\bibliography{refs}

\end{document}

%% file: mytikz_self_inertia_fig_1_log.tex
\begin{tikzpicture}

\definecolor{color0}{rgb}{0.12156862745098,0.466666666666667,0.705882352941177}
\definecolor{color1}{rgb}{1,0.498039215686275,0.0549019607843137}
\definecolor{color2}{rgb}{0.172549019607843,0.627450980392157,0.172549019607843}
\definecolor{color3}{rgb}{0.83921568627451,0.152941176470588,0.156862745098039}
\definecolor{color4}{rgb}{0.580392156862745,0.403921568627451,0.741176470588235}
\definecolor{color5}{rgb}{0.549019607843137,0.337254901960784,0.294117647058824}
\definecolor{color6}{rgb}{0.890196078431372,0.466666666666667,0.76078431372549}
\definecolor{color7}{rgb}{0.737254901960784,0.741176470588235,0.133333333333333}
\definecolor{color8}{rgb}{0.0901960784313725,0.745098039215686,0.811764705882353}
\definecolor{color9}{rgb}{0.,0.,0.}

\begin{axis}[
height=2.4in,
width=3.25in,
legend cell align={left},
legend style={at={(1.1,0.5)}, anchor=west, draw=white!80.0!black},
tick align=outside,
tick pos=both,
x grid style={white!69.01960784313725!black},
xlabel={$t$},
xmin=-1, xmax=50,
xtick style={color=black},
y grid style={white!69.01960784313725!black},
ylabel={$d(k^t,\mathcal{M}_\mathcal{F})$},
ymin= 0.008, ymax=14,
ytick style={color=black},
ymode=log
]
\addplot [semithick, color9]
table {%
0 8.94427190999917
1 8.10214707614972
2 7.26926239441469
3 6.41539047566064
4 5.61224647322253
5 4.86697556877192
6 4.26886425766279
7 3.73232487348223
8 3.23477314908919
9 2.65637423433212
10 2.44081803204967
11 1.88284271247462
12 1.40970562748477
13 1.20970562748477
14 0.896568542494924
15 0.8
16 0.68
17 0.6
18 0.52
19 0.44
20 0.36
21 0.32
22 0.32
23 0.24
24 0.2
25 0.16
26 0.12
27 0.12
28 0.12
29 0.04
30 0.04
31 0.04
32 0.04
33 0.04
34 0.04
35 0.04
36 0.04
37 0
38 0
39 0
40 0
41 0
42 0
43 0
44 0
45 0
46 0
47 0
48 0
49 0
50 0
51 0
52 0
53 0
54 0
55 0
56 0
57 0
58 0
59 0
60 0
61 0
62 0
63 0
64 0
65 0
66 0
67 0
68 0
69 0
70 0
71 0
72 0
73 0
74 0
75 0
76 0
77 0
78 0
79 0
80 0
81 0
82 0
83 0
84 0
85 0
86 0
87 0
88 0
89 0
90 0
91 0
92 0
93 0
94 0
95 0
96 0
97 0
98 0
99 0
100 0
101 0
102 0
103 0
104 0
105 0
106 0
107 0
108 0
109 0
110 0
111 0
112 0
113 0
114 0
115 0
116 0
117 0
118 0
119 0
120 0
121 0
122 0
123 0
124 0
125 0
126 0
127 0
128 0
129 0
130 0
131 0
132 0
133 0
134 0
135 0
136 0
137 0
138 0
139 0
140 0
141 0
142 0
143 0
144 0
145 0
146 0
147 0
148 0
149 0
150 0
151 0
152 0
153 0
154 0
155 0
156 0
157 0
158 0
159 0
160 0
161 0
162 0
163 0
164 0
165 0
166 0
167 0
168 0
169 0
170 0
171 0
172 0
173 0
174 0
175 0
176 0
177 0
178 0
179 0
180 0
181 0
182 0
183 0
184 0
185 0
186 0
187 0
188 0
189 0
190 0
191 0
192 0
193 0
194 0
195 0
196 0
197 0
198 0
199 0
200 0
201 0
202 0
203 0
204 0
205 0
206 0
207 0
208 0
209 0
210 0
211 0
212 0
213 0
214 0
215 0
216 0
217 0
218 0
219 0
220 0
221 0
222 0
223 0
224 0
225 0
226 0
227 0
228 0
229 0
230 0
231 0
232 0
233 0
234 0
235 0
236 0
237 0
238 0
239 0
240 0
241 0
242 0
243 0
244 0
245 0
246 0
247 0
248 0
249 0
};
\addlegendentry{$\alpha = 0.0$}
\addplot [semithick, color1]
table {%
0 8.94427190999917
1 7.82565774431776
2 6.70817430173675
3 5.59334551468725
4 4.60601199229256
5 4.02128861597636
6 3.42645409708852
7 2.71967071098657
8 2.15207483189128
9 1.57065227774629
10 1.36408373525137
11 1.01358485472372
12 0.875952415806172
13 0.785952415806172
14 0.558284271247462
15 0.46
16 0.36
17 0.3
18 0.19
19 0.15
20 0.15
21 0.14
22 0.14
23 0.11
24 0.06
25 0.06
26 0.02
27 0
28 0
29 0
30 0
31 0
32 0
33 0
34 0
35 0
36 0
37 0
38 0
39 0
40 0
41 0
42 0
43 0
44 0
45 0
46 0
47 0
48 0
49 0
50 0
51 0
52 0
53 0
54 0
55 0
56 0
57 0
58 0
59 0
60 0
61 0
62 0
63 0
64 0
65 0
66 0
67 0
68 0
69 0
70 0
71 0
72 0
73 0
74 0
75 0
76 0
77 0
78 0
79 0
80 0
81 0
82 0
83 0
84 0
85 0
86 0
87 0
88 0
89 0
90 0
91 0
92 0
93 0
94 0
95 0
96 0
97 0
98 0
99 0
100 0
101 0
102 0
103 0
104 0
105 0
106 0
107 0
108 0
109 0
110 0
111 0
112 0
113 0
114 0
115 0
116 0
117 0
118 0
119 0
120 0
121 0
122 0
123 0
124 0
125 0
126 0
127 0
128 0
129 0
130 0
131 0
132 0
133 0
134 0
135 0
136 0
137 0
138 0
139 0
140 0
141 0
142 0
143 0
144 0
145 0
146 0
147 0
148 0
149 0
150 0
151 0
152 0
153 0
154 0
155 0
156 0
157 0
158 0
159 0
160 0
161 0
162 0
163 0
164 0
165 0
166 0
167 0
168 0
169 0
170 0
171 0
172 0
173 0
174 0
175 0
176 0
177 0
178 0
179 0
180 0
181 0
182 0
183 0
184 0
185 0
186 0
187 0
188 0
189 0
190 0
191 0
192 0
193 0
194 0
195 0
196 0
197 0
198 0
199 0
200 0
201 0
202 0
203 0
204 0
205 0
206 0
207 0
208 0
209 0
210 0
211 0
212 0
213 0
214 0
215 0
216 0
217 0
218 0
219 0
220 0
221 0
222 0
223 0
224 0
225 0
226 0
227 0
228 0
229 0
230 0
231 0
232 0
233 0
234 0
235 0
236 0
237 0
238 0
239 0
240 0
241 0
242 0
243 0
244 0
245 0
246 0
247 0
248 0
249 0
};
\addlegendentry{$\alpha = 0.1$}
\addplot [semithick, color2]
table {%
0 8.94427190999917
1 7.34378299071165
2 5.55183346222596
3 4.3371836056665
4 3.30819471199996
5 2.42860168889261
6 1.90056012453304
7 1.45089085668238
8 1.172995466397
9 0.894888519383546
10 0.656502815398729
11 0.562360679774998
12 0.43
13 0.3
14 0.21
15 0.11
16 0.1
17 0.08
18 0.03
19 0.02
20 0.02
21 0.01
22 0.01
23 0.01
24 0.01
25 0.01
26 0.01
27 0
28 0
29 0
30 0
31 0
32 0
33 0
34 0
35 0
36 0
37 0
38 0
39 0
40 0
41 0
42 0
43 0
44 0
45 0
46 0
47 0
48 0
49 0
50 0
51 0
52 0
53 0
54 0
55 0
56 0
57 0
58 0
59 0
60 0
61 0
62 0
63 0
64 0
65 0
66 0
67 0
68 0
69 0
70 0
71 0
72 0
73 0
74 0
75 0
76 0
77 0
78 0
79 0
80 0
81 0
82 0
83 0
84 0
85 0
86 0
87 0
88 0
89 0
90 0
91 0
92 0
93 0
94 0
95 0
96 0
97 0
98 0
99 0
100 0
101 0
102 0
103 0
104 0
105 0
106 0
107 0
108 0
109 0
110 0
111 0
112 0
113 0
114 0
115 0
116 0
117 0
118 0
119 0
120 0
121 0
122 0
123 0
124 0
125 0
126 0
127 0
128 0
129 0
130 0
131 0
132 0
133 0
134 0
135 0
136 0
137 0
138 0
139 0
140 0
141 0
142 0
143 0
144 0
145 0
146 0
147 0
148 0
149 0
150 0
151 0
152 0
153 0
154 0
155 0
156 0
157 0
158 0
159 0
160 0
161 0
162 0
163 0
164 0
165 0
166 0
167 0
168 0
169 0
170 0
171 0
172 0
173 0
174 0
175 0
176 0
177 0
178 0
179 0
180 0
181 0
182 0
183 0
184 0
185 0
186 0
187 0
188 0
189 0
190 0
191 0
192 0
193 0
194 0
195 0
196 0
197 0
198 0
199 0
200 0
201 0
202 0
203 0
204 0
205 0
206 0
207 0
208 0
209 0
210 0
211 0
212 0
213 0
214 0
215 0
216 0
217 0
218 0
219 0
220 0
221 0
222 0
223 0
224 0
225 0
226 0
227 0
228 0
229 0
230 0
231 0
232 0
233 0
234 0
235 0
236 0
237 0
238 0
239 0
240 0
241 0
242 0
243 0
244 0
245 0
246 0
247 0
248 0
249 0
};
\addlegendentry{$\alpha = 0.2$}
\addplot [semithick, color3]
table {%
0 8.94427190999917
1 6.27307815897779
2 4.00247048635709
3 2.51382295849671
4 1.62731920020625
5 1.10869666167251
6 0.65685339390596
7 0.38128990204492
8 0.208284271247462
9 0.12
10 0.07
11 0.04
12 0.02
13 0.01
14 0
15 0
16 0
17 0
18 0
19 0
20 0
21 0
22 0
23 0
24 0
25 0
26 0
27 0
28 0
29 0
30 0
31 0
32 0
33 0
34 0
35 0
36 0
37 0
38 0
39 0
40 0
41 0
42 0
43 0
44 0
45 0
46 0
47 0
48 0
49 0
50 0
51 0
52 0
53 0
54 0
55 0
56 0
57 0
58 0
59 0
60 0
61 0
62 0
63 0
64 0
65 0
66 0
67 0
68 0
69 0
70 0
71 0
72 0
73 0
74 0
75 0
76 0
77 0
78 0
79 0
80 0
81 0
82 0
83 0
84 0
85 0
86 0
87 0
88 0
89 0
90 0
91 0
92 0
93 0
94 0
95 0
96 0
97 0
98 0
99 0
100 0
101 0
102 0
103 0
104 0
105 0
106 0
107 0
108 0
109 0
110 0
111 0
112 0
113 0
114 0
115 0
116 0
117 0
118 0
119 0
120 0
121 0
122 0
123 0
124 0
125 0
126 0
127 0
128 0
129 0
130 0
131 0
132 0
133 0
134 0
135 0
136 0
137 0
138 0
139 0
140 0
141 0
142 0
143 0
144 0
145 0
146 0
147 0
148 0
149 0
150 0
151 0
152 0
153 0
154 0
155 0
156 0
157 0
158 0
159 0
160 0
161 0
162 0
163 0
164 0
165 0
166 0
167 0
168 0
169 0
170 0
171 0
172 0
173 0
174 0
175 0
176 0
177 0
178 0
179 0
180 0
181 0
182 0
183 0
184 0
185 0
186 0
187 0
188 0
189 0
190 0
191 0
192 0
193 0
194 0
195 0
196 0
197 0
198 0
199 0
200 0
201 0
202 0
203 0
204 0
205 0
206 0
207 0
208 0
209 0
210 0
211 0
212 0
213 0
214 0
215 0
216 0
217 0
218 0
219 0
220 0
221 0
222 0
223 0
224 0
225 0
226 0
227 0
228 0
229 0
230 0
231 0
232 0
233 0
234 0
235 0
236 0
237 0
238 0
239 0
240 0
241 0
242 0
243 0
244 0
245 0
246 0
247 0
248 0
249 0
};
\addlegendentry{$\alpha = 0.3$}
\addplot [semithick, color4]
table {%
0 8.94427190999917
1 5.1717109884569
2 2.76202831083263
3 1.4698943418555
4 0.769295990260371
5 0.364787086646191
6 0.174142135623731
7 0.05
8 0.03
9 0.01
10 0.01
11 0
12 0
13 0
14 0
15 0
16 0
17 0
18 0
19 0
20 0
21 0
22 0
23 0
24 0
25 0
26 0
27 0
28 0
29 0
30 0
31 0
32 0
33 0
34 0
35 0
36 0
37 0
38 0
39 0
40 0
41 0
42 0
43 0
44 0
45 0
46 0
47 0
48 0
49 0
50 0
51 0
52 0
53 0
54 0
55 0
56 0
57 0
58 0
59 0
60 0
61 0
62 0
63 0
64 0
65 0
66 0
67 0
68 0
69 0
70 0
71 0
72 0
73 0
74 0
75 0
76 0
77 0
78 0
79 0
80 0
81 0
82 0
83 0
84 0
85 0
86 0
87 0
88 0
89 0
90 0
91 0
92 0
93 0
94 0
95 0
96 0
97 0
98 0
99 0
100 0
101 0
102 0
103 0
104 0
105 0
106 0
107 0
108 0
109 0
110 0
111 0
112 0
113 0
114 0
115 0
116 0
117 0
118 0
119 0
120 0
121 0
122 0
123 0
124 0
125 0
126 0
127 0
128 0
129 0
130 0
131 0
132 0
133 0
134 0
135 0
136 0
137 0
138 0
139 0
140 0
141 0
142 0
143 0
144 0
145 0
146 0
147 0
148 0
149 0
150 0
151 0
152 0
153 0
154 0
155 0
156 0
157 0
158 0
159 0
160 0
161 0
162 0
163 0
164 0
165 0
166 0
167 0
168 0
169 0
170 0
171 0
172 0
173 0
174 0
175 0
176 0
177 0
178 0
179 0
180 0
181 0
182 0
183 0
184 0
185 0
186 0
187 0
188 0
189 0
190 0
191 0
192 0
193 0
194 0
195 0
196 0
197 0
198 0
199 0
200 0
201 0
202 0
203 0
204 0
205 0
206 0
207 0
208 0
209 0
210 0
211 0
212 0
213 0
214 0
215 0
216 0
217 0
218 0
219 0
220 0
221 0
222 0
223 0
224 0
225 0
226 0
227 0
228 0
229 0
230 0
231 0
232 0
233 0
234 0
235 0
236 0
237 0
238 0
239 0
240 0
241 0
242 0
243 0
244 0
245 0
246 0
247 0
248 0
249 0
};
\addlegendentry{$\alpha = 0.4$}
\addplot [semithick, color5]
table {%
0 8.94427190999917
1 5.00955016693965
2 2.83799220573501
3 1.47615800454319
4 0.763480545535475
5 0.374694134495336
6 0.144142135623731
7 0.074142135623731
8 0.04
9 0.01
10 0.01
11 0.01
12 0
13 0
14 0
15 0
16 0
17 0
18 0
19 0
20 0
21 0
22 0
23 0
24 0
25 0
26 0
27 0
28 0
29 0
30 0
31 0
32 0
33 0
34 0
35 0
36 0
37 0
38 0
39 0
40 0
41 0
42 0
43 0
44 0
45 0
46 0
47 0
48 0
49 0
50 0
51 0
52 0
53 0
54 0
55 0
56 0
57 0
58 0
59 0
60 0
61 0
62 0
63 0
64 0
65 0
66 0
67 0
68 0
69 0
70 0
71 0
72 0
73 0
74 0
75 0
76 0
77 0
78 0
79 0
80 0
81 0
82 0
83 0
84 0
85 0
86 0
87 0
88 0
89 0
90 0
91 0
92 0
93 0
94 0
95 0
96 0
97 0
98 0
99 0
100 0
101 0
102 0
103 0
104 0
105 0
106 0
107 0
108 0
109 0
110 0
111 0
112 0
113 0
114 0
115 0
116 0
117 0
118 0
119 0
120 0
121 0
122 0
123 0
124 0
125 0
126 0
127 0
128 0
129 0
130 0
131 0
132 0
133 0
134 0
135 0
136 0
137 0
138 0
139 0
140 0
141 0
142 0
143 0
144 0
145 0
146 0
147 0
148 0
149 0
150 0
151 0
152 0
153 0
154 0
155 0
156 0
157 0
158 0
159 0
160 0
161 0
162 0
163 0
164 0
165 0
166 0
167 0
168 0
169 0
170 0
171 0
172 0
173 0
174 0
175 0
176 0
177 0
178 0
179 0
180 0
181 0
182 0
183 0
184 0
185 0
186 0
187 0
188 0
189 0
190 0
191 0
192 0
193 0
194 0
195 0
196 0
197 0
198 0
199 0
200 0
201 0
202 0
203 0
204 0
205 0
206 0
207 0
208 0
209 0
210 0
211 0
212 0
213 0
214 0
215 0
216 0
217 0
218 0
219 0
220 0
221 0
222 0
223 0
224 0
225 0
226 0
227 0
228 0
229 0
230 0
231 0
232 0
233 0
234 0
235 0
236 0
237 0
238 0
239 0
240 0
241 0
242 0
243 0
244 0
245 0
246 0
247 0
248 0
249 0
};
\addlegendentry{$\alpha = 0.5$}
\addplot [semithick, color6]
table {%
0 8.94427190999917
1 5.79524890171628
2 3.67309333098075
3 2.29095856694446
4 1.42644988298416
5 0.802754145023388
6 0.358031186194344
7 0.188284271247462
8 0.09
9 0.04
10 0.01
11 0.01
12 0
13 0
14 0
15 0
16 0
17 0
18 0
19 0
20 0
21 0
22 0
23 0
24 0
25 0
26 0
27 0
28 0
29 0
30 0
31 0
32 0
33 0
34 0
35 0
36 0
37 0
38 0
39 0
40 0
41 0
42 0
43 0
44 0
45 0
46 0
47 0
48 0
49 0
50 0
51 0
52 0
53 0
54 0
55 0
56 0
57 0
58 0
59 0
60 0
61 0
62 0
63 0
64 0
65 0
66 0
67 0
68 0
69 0
70 0
71 0
72 0
73 0
74 0
75 0
76 0
77 0
78 0
79 0
80 0
81 0
82 0
83 0
84 0
85 0
86 0
87 0
88 0
89 0
90 0
91 0
92 0
93 0
94 0
95 0
96 0
97 0
98 0
99 0
100 0
101 0
102 0
103 0
104 0
105 0
106 0
107 0
108 0
109 0
110 0
111 0
112 0
113 0
114 0
115 0
116 0
117 0
118 0
119 0
120 0
121 0
122 0
123 0
124 0
125 0
126 0
127 0
128 0
129 0
130 0
131 0
132 0
133 0
134 0
135 0
136 0
137 0
138 0
139 0
140 0
141 0
142 0
143 0
144 0
145 0
146 0
147 0
148 0
149 0
150 0
151 0
152 0
153 0
154 0
155 0
156 0
157 0
158 0
159 0
160 0
161 0
162 0
163 0
164 0
165 0
166 0
167 0
168 0
169 0
170 0
171 0
172 0
173 0
174 0
175 0
176 0
177 0
178 0
179 0
180 0
181 0
182 0
183 0
184 0
185 0
186 0
187 0
188 0
189 0
190 0
191 0
192 0
193 0
194 0
195 0
196 0
197 0
198 0
199 0
200 0
201 0
202 0
203 0
204 0
205 0
206 0
207 0
208 0
209 0
210 0
211 0
212 0
213 0
214 0
215 0
216 0
217 0
218 0
219 0
220 0
221 0
222 0
223 0
224 0
225 0
226 0
227 0
228 0
229 0
230 0
231 0
232 0
233 0
234 0
235 0
236 0
237 0
238 0
239 0
240 0
241 0
242 0
243 0
244 0
245 0
246 0
247 0
248 0
249 0
};
\addlegendentry{$\alpha = 0.6$}
\addplot [semithick, white!49.80392156862745!black]
table {%
0 8.94427190999917
1 7.2520352805794
2 5.37550227278696
3 3.85317373691014
4 2.97458384356632
5 2.31691155748173
6 1.77017466872237
7 1.42668327191416
8 1.11099008522865
9 0.847943290844233
10 0.608891911236297
11 0.483811198926014
12 0.402360679774998
13 0.262360679774998
14 0.202360679774998
15 0.172360679774998
16 0.112360679774998
17 0.0723606797749979
18 0.0623606797749979
19 0.0341421356237309
20 0.0341421356237309
21 0.024142135623731
22 0.014142135623731
23 0.014142135623731
24 0.014142135623731
25 0.014142135623731
26 0.014142135623731
27 0.014142135623731
28 0.01
29 0.01
30 0
31 0
32 0
33 0
34 0
35 0
36 0
37 0
38 0
39 0
40 0
41 0
42 0
43 0
44 0
45 0
46 0
47 0
48 0
49 0
50 0
51 0
52 0
53 0
54 0
55 0
56 0
57 0
58 0
59 0
60 0
61 0
62 0
63 0
64 0
65 0
66 0
67 0
68 0
69 0
70 0
71 0
72 0
73 0
74 0
75 0
76 0
77 0
78 0
79 0
80 0
81 0
82 0
83 0
84 0
85 0
86 0
87 0
88 0
89 0
90 0
91 0
92 0
93 0
94 0
95 0
96 0
97 0
98 0
99 0
100 0
101 0
102 0
103 0
104 0
105 0
106 0
107 0
108 0
109 0
110 0
111 0
112 0
113 0
114 0
115 0
116 0
117 0
118 0
119 0
120 0
121 0
122 0
123 0
124 0
125 0
126 0
127 0
128 0
129 0
130 0
131 0
132 0
133 0
134 0
135 0
136 0
137 0
138 0
139 0
140 0
141 0
142 0
143 0
144 0
145 0
146 0
147 0
148 0
149 0
150 0
151 0
152 0
153 0
154 0
155 0
156 0
157 0
158 0
159 0
160 0
161 0
162 0
163 0
164 0
165 0
166 0
167 0
168 0
169 0
170 0
171 0
172 0
173 0
174 0
175 0
176 0
177 0
178 0
179 0
180 0
181 0
182 0
183 0
184 0
185 0
186 0
187 0
188 0
189 0
190 0
191 0
192 0
193 0
194 0
195 0
196 0
197 0
198 0
199 0
200 0
201 0
202 0
203 0
204 0
205 0
206 0
207 0
208 0
209 0
210 0
211 0
212 0
213 0
214 0
215 0
216 0
217 0
218 0
219 0
220 0
221 0
222 0
223 0
224 0
225 0
226 0
227 0
228 0
229 0
230 0
231 0
232 0
233 0
234 0
235 0
236 0
237 0
238 0
239 0
240 0
241 0
242 0
243 0
244 0
245 0
246 0
247 0
248 0
249 0
};
\addlegendentry{$\alpha = 0.7$}
\addplot [semithick, color7]
table {%
0 8.94427190999917
1 8.11873741572643
2 6.88836059724254
3 5.5315977831697
4 4.57264436009007
5 3.64537209030673
6 2.91525449748815
7 2.46152505570391
8 2.11227761349523
9 1.84198327726456
10 1.55386343399247
11 1.37614046707958
12 1.24390536330014
13 1.13908613639108
14 1.04645208168266
15 0.99479185013593
16 0.895519524006262
17 0.791866520123466
18 0.744437821707322
19 0.696983803660324
20 0.604136367906083
21 0.572309609769267
22 0.513365117201687
23 0.437088763925122
24 0.409914374572979
25 0.363570238421299
26 0.323570238421299
27 0.253570238421299
28 0.20706742302257
29 0.192925287398839
30 0.165604779323151
31 0.165604779323151
32 0.15146264369942
33 0.10146264369942
34 0.0873205080756888
35 0.0673205080756888
36 0.0573205080756888
37 0.0573205080756888
38 0.0573205080756888
39 0.0573205080756888
40 0.0541421356237309
41 0.0441421356237309
42 0.04
43 0.03
44 0.03
45 0.03
46 0.03
47 0.03
48 0.03
49 0.03
50 0.03
51 0.03
52 0.03
53 0.02
54 0.02
55 0.02
56 0.02
57 0.02
58 0.02
59 0.02
60 0.01
61 0.01
62 0.01
63 0.01
64 0.01
65 0.01
66 0
67 0
68 0
69 0
70 0
71 0
72 0
73 0
74 0
75 0
76 0
77 0
78 0
79 0
80 0
81 0
82 0
83 0
84 0
85 0
86 0
87 0
88 0
89 0
90 0
91 0
92 0
93 0
94 0
95 0
96 0
97 0
98 0
99 0
100 0
101 0
102 0
103 0
104 0
105 0
106 0
107 0
108 0
109 0
110 0
111 0
112 0
113 0
114 0
115 0
116 0
117 0
118 0
119 0
120 0
121 0
122 0
123 0
124 0
125 0
126 0
127 0
128 0
129 0
130 0
131 0
132 0
133 0
134 0
135 0
136 0
137 0
138 0
139 0
140 0
141 0
142 0
143 0
144 0
145 0
146 0
147 0
148 0
149 0
150 0
151 0
152 0
153 0
154 0
155 0
156 0
157 0
158 0
159 0
160 0
161 0
162 0
163 0
164 0
165 0
166 0
167 0
168 0
169 0
170 0
171 0
172 0
173 0
174 0
175 0
176 0
177 0
178 0
179 0
180 0
181 0
182 0
183 0
184 0
185 0
186 0
187 0
188 0
189 0
190 0
191 0
192 0
193 0
194 0
195 0
196 0
197 0
198 0
199 0
200 0
201 0
202 0
203 0
204 0
205 0
206 0
207 0
208 0
209 0
210 0
211 0
212 0
213 0
214 0
215 0
216 0
217 0
218 0
219 0
220 0
221 0
222 0
223 0
224 0
225 0
226 0
227 0
228 0
229 0
230 0
231 0
232 0
233 0
234 0
235 0
236 0
237 0
238 0
239 0
240 0
241 0
242 0
243 0
244 0
245 0
246 0
247 0
248 0
249 0
};
\addlegendentry{$\alpha = 0.8$}
\addplot [semithick, color8]
table {%
0 8.94427190999917
1 8.75745183951313
2 8.53639020041645
3 8.20685463883213
4 7.90165084591222
5 7.44752386315757
6 7.04015320305632
7 6.64991830767067
8 6.33589975708656
9 5.97551979477835
10 5.55081321027286
11 5.19487166265088
12 4.89612727818329
13 4.64724387355141
14 4.26754435399299
15 3.93398470389686
16 3.67202848551826
17 3.48732766553344
18 3.26278668235562
19 3.07188101897672
20 2.95554622508835
21 2.74916205799471
22 2.61351486863055
23 2.44310463110954
24 2.29703236339495
25 2.15743121931046
26 2.08330556196269
27 2.02381465561888
28 1.87686108859173
29 1.77883527239853
30 1.73714523623477
31 1.64066821586292
32 1.59526814939443
33 1.53952051605341
34 1.48917797056525
35 1.39183371801858
36 1.37474479738613
37 1.36231839051494
38 1.32519160279683
39 1.29424924123454
40 1.18823613844036
41 1.15921775927537
42 1.13087861600048
43 1.08817744969786
44 1.06169567103957
45 1.03146981104111
46 0.99928595336377
47 0.996107580911812
48 0.992929208459854
49 0.987135411193798
50 0.962767415571608
51 0.941144638969925
52 0.916853892868628
53 0.900071041989318
54 0.866452498224599
55 0.832310362600868
56 0.820176144948034
57 0.800914048121349
58 0.784476959818815
59 0.766258415667548
60 0.764124198014714
61 0.736884081566376
62 0.730176172078136
63 0.717815492303138
64 0.710641102950996
65 0.708280423175998
66 0.704138287552267
67 0.689996151928536
68 0.682821762576393
69 0.666318947177664
70 0.666318947177664
71 0.643140574725706
72 0.617476430707022
73 0.617476430707022
74 0.607476430707022
75 0.59333429508329
76 0.58919215945956
77 0.585050023835829
78 0.576514181897396
79 0.576514181897396
80 0.576514181897396
81 0.562372046273665
82 0.548121720309479
83 0.531764975405564
84 0.531764975405564
85 0.531764975405564
86 0.517622839781833
87 0.517622839781833
88 0.50651844238227
89 0.49651844238227
90 0.492376306758539
91 0.492376306758539
92 0.482376306758539
93 0.463704956282817
94 0.449562820659086
95 0.419562820659086
96 0.401278549411624
97 0.401278549411624
98 0.391278549411624
99 0.381278549411624
100 0.381278549411624
101 0.381278549411624
102 0.367788246117805
103 0.367788246117805
104 0.347788246117805
105 0.340613856765662
106 0.331005577111169
107 0.331005577111169
108 0.331005577111169
109 0.331005577111169
110 0.327827204659211
111 0.327827204659211
112 0.31368506903548
113 0.31368506903548
114 0.290506696583523
115 0.284741784358108
116 0.284741784358108
117 0.250599648734377
118 0.240599648734377
119 0.240599648734377
120 0.240599648734377
121 0.240599648734377
122 0.230599648734377
123 0.230599648734377
124 0.230599648734377
125 0.226457513110646
126 0.226457513110646
127 0.216457513110646
128 0.204494897427832
129 0.204494897427832
130 0.204494897427832
131 0.204494897427832
132 0.204494897427832
133 0.184494897427832
134 0.184494897427832
135 0.184494897427832
136 0.184494897427832
137 0.184494897427832
138 0.184494897427832
139 0.174494897427832
140 0.164494897427832
141 0.164494897427832
142 0.164494897427832
143 0.164494897427832
144 0.164494897427832
145 0.154494897427832
146 0.144494897427832
147 0.144494897427832
148 0.144494897427832
149 0.134494897427832
150 0.134494897427832
151 0.134494897427832
152 0.134494897427832
153 0.134494897427832
154 0.134494897427832
155 0.124494897427832
156 0.124494897427832
157 0.124494897427832
158 0.114494897427832
159 0.107320508075689
160 0.107320508075689
161 0.107320508075689
162 0.107320508075689
163 0.0973205080756888
164 0.0973205080756888
165 0.0973205080756888
166 0.0973205080756888
167 0.0973205080756888
168 0.0973205080756888
169 0.0973205080756888
170 0.0973205080756888
171 0.0973205080756888
172 0.0973205080756888
173 0.0973205080756888
174 0.0873205080756888
175 0.0873205080756888
176 0.0773205080756888
177 0.0773205080756888
178 0.0773205080756888
179 0.0773205080756888
180 0.0773205080756888
181 0.0773205080756888
182 0.0773205080756888
183 0.0773205080756888
184 0.0773205080756888
185 0.0773205080756888
186 0.0573205080756888
187 0.0573205080756888
188 0.0541421356237309
189 0.0541421356237309
190 0.0541421356237309
191 0.0541421356237309
192 0.0541421356237309
193 0.0541421356237309
194 0.0541421356237309
195 0.0541421356237309
196 0.0541421356237309
197 0.0541421356237309
198 0.0541421356237309
199 0.0541421356237309
200 0.0541421356237309
201 0.0441421356237309
202 0.0441421356237309
203 0.0441421356237309
204 0.0441421356237309
205 0.0441421356237309
206 0.0441421356237309
207 0.0441421356237309
208 0.04
209 0.04
210 0.04
211 0.04
212 0.04
213 0.04
214 0.04
215 0.04
216 0.04
217 0.04
218 0.04
219 0.04
220 0.04
221 0.04
222 0.04
223 0.04
224 0.04
225 0.04
226 0.04
227 0.04
228 0.04
229 0.04
230 0.03
231 0.03
232 0.03
233 0.03
234 0.03
235 0.03
236 0.03
237 0.03
238 0.03
239 0.03
240 0.03
241 0.03
242 0.03
243 0.03
244 0.03
245 0.03
246 0.03
247 0.03
248 0.03
249 0.03
};
\addlegendentry{$\alpha = 0.9$}
\addplot [semithick, color0]
table {%
0 8.94427190999917
1 8.94427190999917
2 8.94427190999917
3 8.94427190999917
4 8.94427190999917
5 8.94427190999917
6 8.94427190999917
7 8.94427190999917
8 8.94427190999917
9 8.94427190999917
10 8.94427190999917
11 8.94427190999917
12 8.94427190999917
13 8.94427190999917
14 8.94427190999917
15 8.94427190999917
16 8.94427190999917
17 8.94427190999917
18 8.94427190999917
19 8.94427190999917
20 8.94427190999917
21 8.94427190999917
22 8.94427190999917
23 8.94427190999917
24 8.94427190999917
25 8.94427190999917
26 8.94427190999917
27 8.94427190999917
28 8.94427190999917
29 8.94427190999917
30 8.94427190999917
31 8.94427190999917
32 8.94427190999917
33 8.94427190999917
34 8.94427190999917
35 8.94427190999917
36 8.94427190999917
37 8.94427190999917
38 8.94427190999917
39 8.94427190999917
40 8.94427190999917
41 8.94427190999917
42 8.94427190999917
43 8.94427190999917
44 8.94427190999917
45 8.94427190999917
46 8.94427190999917
47 8.94427190999917
48 8.94427190999917
49 8.94427190999917
50 8.94427190999917
51 8.94427190999917
52 8.94427190999917
53 8.94427190999917
54 8.94427190999917
55 8.94427190999917
56 8.94427190999917
57 8.94427190999917
58 8.94427190999917
59 8.94427190999917
60 8.94427190999917
61 8.94427190999917
62 8.94427190999917
63 8.94427190999917
64 8.94427190999917
65 8.94427190999917
66 8.94427190999917
67 8.94427190999917
68 8.94427190999917
69 8.94427190999917
70 8.94427190999917
71 8.94427190999917
72 8.94427190999917
73 8.94427190999917
74 8.94427190999917
75 8.94427190999917
76 8.94427190999917
77 8.94427190999917
78 8.94427190999917
79 8.94427190999917
80 8.94427190999917
81 8.94427190999917
82 8.94427190999917
83 8.94427190999917
84 8.94427190999917
85 8.94427190999917
86 8.94427190999917
87 8.94427190999917
88 8.94427190999917
89 8.94427190999917
90 8.94427190999917
91 8.94427190999917
92 8.94427190999917
93 8.94427190999917
94 8.94427190999917
95 8.94427190999917
96 8.94427190999917
97 8.94427190999917
98 8.94427190999917
99 8.94427190999917
100 8.94427190999917
101 8.94427190999917
102 8.94427190999917
103 8.94427190999917
104 8.94427190999917
105 8.94427190999917
106 8.94427190999917
107 8.94427190999917
108 8.94427190999917
109 8.94427190999917
110 8.94427190999917
111 8.94427190999917
112 8.94427190999917
113 8.94427190999917
114 8.94427190999917
115 8.94427190999917
116 8.94427190999917
117 8.94427190999917
118 8.94427190999917
119 8.94427190999917
120 8.94427190999917
121 8.94427190999917
122 8.94427190999917
123 8.94427190999917
124 8.94427190999917
125 8.94427190999917
126 8.94427190999917
127 8.94427190999917
128 8.94427190999917
129 8.94427190999917
130 8.94427190999917
131 8.94427190999917
132 8.94427190999917
133 8.94427190999917
134 8.94427190999917
135 8.94427190999917
136 8.94427190999917
137 8.94427190999917
138 8.94427190999917
139 8.94427190999917
140 8.94427190999917
141 8.94427190999917
142 8.94427190999917
143 8.94427190999917
144 8.94427190999917
145 8.94427190999917
146 8.94427190999917
147 8.94427190999917
148 8.94427190999917
149 8.94427190999917
150 8.94427190999917
151 8.94427190999917
152 8.94427190999917
153 8.94427190999917
154 8.94427190999917
155 8.94427190999917
156 8.94427190999917
157 8.94427190999917
158 8.94427190999917
159 8.94427190999917
160 8.94427190999917
161 8.94427190999917
162 8.94427190999917
163 8.94427190999917
164 8.94427190999917
165 8.94427190999917
166 8.94427190999917
167 8.94427190999917
168 8.94427190999917
169 8.94427190999917
170 8.94427190999917
171 8.94427190999917
172 8.94427190999917
173 8.94427190999917
174 8.94427190999917
175 8.94427190999917
176 8.94427190999917
177 8.94427190999917
178 8.94427190999917
179 8.94427190999917
180 8.94427190999917
181 8.94427190999917
182 8.94427190999917
183 8.94427190999917
184 8.94427190999917
185 8.94427190999917
186 8.94427190999917
187 8.94427190999917
188 8.94427190999917
189 8.94427190999917
190 8.94427190999917
191 8.94427190999917
192 8.94427190999917
193 8.94427190999917
194 8.94427190999917
195 8.94427190999917
196 8.94427190999917
197 8.94427190999917
198 8.94427190999917
199 8.94427190999917
200 8.94427190999917
201 8.94427190999917
202 8.94427190999917
203 8.94427190999917
204 8.94427190999917
205 8.94427190999917
206 8.94427190999917
207 8.94427190999917
208 8.94427190999917
209 8.94427190999917
210 8.94427190999917
211 8.94427190999917
212 8.94427190999917
213 8.94427190999917
214 8.94427190999917
215 8.94427190999917
216 8.94427190999917
217 8.94427190999917
218 8.94427190999917
219 8.94427190999917
220 8.94427190999917
221 8.94427190999917
222 8.94427190999917
223 8.94427190999917
224 8.94427190999917
225 8.94427190999917
226 8.94427190999917
227 8.94427190999917
228 8.94427190999917
229 8.94427190999917
230 8.94427190999917
231 8.94427190999917
232 8.94427190999917
233 8.94427190999917
234 8.94427190999917
235 8.94427190999917
236 8.94427190999917
237 8.94427190999917
238 8.94427190999917
239 8.94427190999917
240 8.94427190999917
241 8.94427190999917
242 8.94427190999917
243 8.94427190999917
244 8.94427190999917
245 8.94427190999917
246 8.94427190999917
247 8.94427190999917
248 8.94427190999917
249 8.94427190999917
};
\addlegendentry{$\alpha = 1.0$}
\end{axis}

\end{tikzpicture}

%% file: mytikz_professor_fig_1.tex
\begin{tikzpicture}

\definecolor{color0}{rgb}{0.12156862745098,0.466666666666667,0.705882352941177}
\definecolor{color1}{rgb}{1,0.498039215686275,0.0549019607843137}

\begin{axis}[
legend cell align={left},
legend style={draw=white!80.0!black},
tick align=outside,
tick pos=both,
x grid style={white!69.01960784313725!black},
xlabel={t},
xmin=-1.25, xmax=26.25,
xtick style={color=black},
y grid style={white!69.01960784313725!black},
ylabel={$d(k^t,\mathcal{M}_\mathcal{F})$},
ymin=0.03, ymax=9.39148550549913,
ytick style={color=black},
ymode=log
]
\addplot [semithick, color0]
table {%
0 8.94427190999917
1 8.68827346009904
2 6.01127827227921
3 2.27002232510496
4 0.797497500304315
5 0.435184586967827
6 0.286400503629287
7 0.244281462044708
8 0.24110308959275
9 0.202007498706004
10 0.182779016934951
11 0.171156240333267
12 0.16706742302257
13 0.162925287398839
14 0.158783151775109
15 0.157308383527285
16 0.137308383527285
17 0.0886370330515627
18 0.0886370330515627
19 0.0773205080756888
20 0.0673205080756888
21 0.0673205080756888
22 0.0473205080756888
23 0.0441421356237309
24 0.04
25 0
};
\addlegendentry{constant likelihood landscape}
\addplot [semithick, color1]
table {%
0 8.94427190999917
1 5.11468567414611
2 0.809705627484771
3 0.0965685424949238
4 0
5 0
6 0
7 0
8 0
9 0
10 0
11 0
12 0
13 0
14 0
15 0
16 0
17 0
18 0
19 0
20 0
21 0
22 0
23 0
24 0
25 0
};
\addlegendentry{concave likelihood landscape}
\end{axis}

\end{tikzpicture}

%% file: mytikz_professor_fig_2.tex
\begin{tikzpicture}

\definecolor{color0}{rgb}{0.12156862745098,0.466666666666667,0.705882352941177}
\definecolor{color1}{rgb}{1,0.498039215686275,0.0549019607843137}

\begin{axis}[
legend cell align={left},
legend style={at={(0.97,0.03)}, anchor=south east, draw=white!80.0!black},
tick align=outside,
tick pos=both,
x grid style={white!69.01960784313725!black},
xlabel={individual},
xmin=0.785340770518991, xmax=5.21465922948101,
y grid style={white!69.01960784313725!black},
ymin=-0.5, ymax=10,
]
\addplot [only marks, mark=x, draw=color0, fill=color0]
table{%
1 1.61357428350635
2 7.33069762649282
3 7.33069762649282
4 7.33069762649282
5 7.33069762649282
};
\addlegendentry{constant likelihood landscape}
\addplot [only marks, mark=+, draw=color1, fill=color1]
table{%
1 0
2 8.94427190999917
3 8.94427190999917
4 8.94427190999917
5 8.94427190999917
};
\addlegendentry{concave likelihood landscape}
\end{axis}

\end{tikzpicture}

%% file: mytikz_creation_RE.tex
\begin{tikzpicture}

\definecolor{color0}{rgb}{0.12156862745098,0.466666666666667,0.705882352941177}

\begin{axis}[
tick align=outside,
tick pos=both,
x grid style={white!69.01960784313725!black},
xlabel={t},
xmin=-1250, xmax=26250,
xtick style={color=black},
y grid style={white!69.01960784313725!black},
ylabel={RE$(t)$},
ymin=-1.049945470, ymax=0.048854889480199,
ytick style={color=black}
]
\addplot [semithick, color0]
table {%
x y
0 -1
100 -0.934268419240421
200 -0.911286488449256
300 -0.883738074906486
400 -0.865015346427878
500 -0.835956238527923
600 -0.781265046075365
700 -0.764959659800397
800 -0.721921015877556
900 -0.724506238191151
1000 -0.701898343758877
1100 -0.683408188920054
1200 -0.671523070554184
1300 -0.645313596941414
1400 -0.645115651116012
1500 -0.608993003057616
1600 -0.62072663883469
1700 -0.580406760829873
1800 -0.579547099215114
1900 -0.569902646440715
2000 -0.549575786849573
2100 -0.559169045656163
2200 -0.547211368930378
2300 -0.52709908949842
2400 -0.518047531070436
2500 -0.499555780246023
2600 -0.480828564500994
2700 -0.463682448870824
2800 -0.451566263708635
2900 -0.435108734200811
3000 -0.420846533128241
3100 -0.409896727876108
3200 -0.413796079889011
3300 -0.385924629750861
3400 -0.380885834443576
3500 -0.360643579140288
3600 -0.376607093740463
3700 -0.343834946234218
3800 -0.328508086332861
3900 -0.298066237181222
4000 -0.294965406973774
4100 -0.302700516151538
4200 -0.297677951141994
4300 -0.290959353809312
4400 -0.272778541503929
4500 -0.292083112201729
4600 -0.274713533664202
4700 -0.268651389847664
4800 -0.247579250615475
4900 -0.239154680747619
5000 -0.217045408783007
5100 -0.211580453395134
5200 -0.208554217037661
5300 -0.182730775575824
5400 -0.16631752923011
5500 -0.16096497269439
5600 -0.155694962321215
5700 -0.173321506950872
5800 -0.159105860832693
5900 -0.163838809934772
6000 -0.165140914747158
6100 -0.154811484925143
6200 -0.157436952965722
6300 -0.153144965363207
6400 -0.155368672388287
6500 -0.147422255290245
6600 -0.142329481108616
6700 -0.151463694254237
6800 -0.172453215373156
6900 -0.171275068409997
7000 -0.172398553021751
7100 -0.166049623418758
7200 -0.139512333591678
7300 -0.127104496847621
7400 -0.133673966799998
7500 -0.127354890189364
7600 -0.135423522741045
7700 -0.122450284845457
7800 -0.116344287968334
7900 -0.122643155831337
8000 -0.101215146126482
8100 -0.0975193136623344
8200 -0.0937503799869673
8300 -0.0962473448195604
8400 -0.102145457124856
8500 -0.0902237392250528
8600 -0.0903448957106575
8700 -0.0990088191249674
8800 -0.0907941809583688
8900 -0.0860146593671266
9000 -0.0866204940340296
9100 -0.0792309420625507
9200 -0.080433603658422
9300 -0.0813624848854838
9400 -0.0859013723444297
9500 -0.0915227887552324
9600 -0.0823834463716252
9700 -0.0792087596066443
9800 -0.0672818224620071
9900 -0.0751863804911561
10000 -0.0754784883178257
10100 -0.0699734738717651
10200 -0.0778759876107529
10300 -0.0687067616088582
10400 -0.0670435511930941
10500 -0.0657128831959094
10600 -0.061689997089925
10700 -0.0538231697859909
10800 -0.0463529531085176
10900 -0.0457326399646088
11000 -0.0403768926601289
11100 -0.0394729621891155
11200 -0.0343710650511289
11300 -0.0375964599974587
11400 -0.0338789548054329
11500 -0.0368546746363921
11600 -0.034596333566998
11700 -0.0403838160169986
11800 -0.0260784821179297
11900 -0.0309396854909959
12000 -0.0324801313730575
12100 -0.0292452821777341
12200 -0.0268386261759652
12300 -0.0245287514085859
12400 -0.025971293920381
12500 -0.0275365679058356
12600 -0.0212765783464551
12700 -0.0207369256224864
12800 -0.0216982059789009
12900 -0.0212710897038967
13000 -0.0210820265065006
13100 -0.0247250844438288
13200 -0.0198758214914059
13300 -0.017290168373423
13400 -0.021760075956532
13500 -0.0250080493664831
13600 -0.0161080259438215
13700 -0.0140703109844616
13800 -0.0106381855285254
13900 -0.00883940562661125
14000 -0.0297134608046915
14100 -0.00827644846710967
14200 -0.00949116020852211
14300 -0.00716674820125043
14400 -0.00743739159786567
14500 -0.00621293925088924
14600 -0.00826451523622895
14700 -0.00689263755937874
14800 -0.00807370407809672
14900 -0.0059224239333709
15000 -0.00754168064960963
15100 -0.00873681935712801
15200 -0.00720211745486501
15300 -0.00821943684772155
15400 -0.00868827364538424
15500 -0.0115266633909252
15600 -0.0140337721034486
15700 -0.0134368739065704
15800 -0.00643811456079428
15900 -0.00620426484847369
16000 -0.0092337617854485
16100 -0.00842288249552056
16200 -0.00964983686946558
16300 -0.00507197826764797
16400 -0.00720827250404584
16500 -0.00551042199582502
16600 -0.0124654836777357
16700 -0.0082252902180649
16800 -0.0054944237991189
16900 -0.00457565317973723
17000 -0.0110365467987884
17100 -0.00433898343891062
17200 -0.00300634512014573
17300 -0.00593541823992807
17400 -0.00284746433639273
17500 -0.00490640456110845
17600 -0.00609058426925986
17700 -0.0058350323927689
17800 -0.00496358302492131
17900 -0.0036519939688967
18000 -0.00775527365413372
18100 -0.00225956100377044
18200 -0.00557026084432158
18300 -0.00330310676520911
18400 -0.00169759909795556
18500 -0.00138457822184365
18600 -0.00242842339552386
18700 -0.00284301293970157
18800 -0.00127986004625294
18900 -0.00134690763471123
19000 -0.00385726430410771
19100 -0.0019594280094004
19200 -0.00308924546917852
19300 -0.0078311103746652
19400 -0.00109058144742954
19500 -0.00344419092320299
19600 -0.00335010512733102
19700 -0.00176237905744021
19800 -0.00367456330237712
19900 -0.00196641380594603
20000 -0.00406780078321375
20100 -0.00218838636193495
20200 -0.0116938115780037
20300 -0.00418619934522044
20400 -0.0069321509233816
20500 -0.00893957293712252
20600 -0.00124859474922159
20700 -0.00340125544799651
20800 -0.00532272855707437
20900 -0.0052256005608809
21000 -0.0025718446971112
21100 -0.00706508455876704
21200 -0.00687790244558056
21300 -0.00203229838192251
21400 -0.00552684241130594
21500 -0.00734792517953869
21600 -0.00582928329112765
21700 -0.00294869783580114
21800 -0.00809546215889478
21900 -0.00311597745279702
22000 -0.00255667505527318
22100 -0.00389959208217072
22200 -0.00354002066213887
22300 -0.00658042188064167
22400 -0.00147797664513913
22500 -0.00156468272447034
22600 -0.00203047519578133
22700 -0.00184055751137314
22800 -0.003464894510478
22900 -0.00246358592156281
23000 -0.00680737466921153
23100 -0.00157702877354992
23200 -0.0017275840460276
23300 -0.00580433545630617
23400 -0.00318438548597166
23500 -0.00216072211932912
23600 -0.00199162285257407
23700 -0.00434701728182789
23800 -0.00156029563923555
23900 -0.00353656124708077
24000 -0.00231203044795323
24100 -0.00608497942746034
24200 -0.00269288015052121
24300 -0.00269784776694034
24400 -0.00556543981834075
24500 -0.00329389830323101
24600 -0.00200870657298194
24700 -0.00348816881314451
24800 -0.0016656545140415
24900 -0.00183495556064686
25000 -0.00346761093608336
};
\end{axis}

\end{tikzpicture}

%% file: mytikz_creation_t_0.tex
\begin{tikzpicture}

\definecolor{color0}{rgb}{0.12156862745098,0.466666666666667,0.705882352941177}
\definecolor{color1}{rgb}{1,0.498039215686275,0.0549019607843137}
\definecolor{color2}{rgb}{0.172549019607843,0.627450980392157,0.172549019607843}
\definecolor{color3}{rgb}{0.83921568627451,0.152941176470588,0.156862745098039}
\definecolor{color4}{rgb}{0.580392156862745,0.403921568627451,0.741176470588235}
\definecolor{color5}{rgb}{0.549019607843137,0.337254901960784,0.294117647058824}
\definecolor{color6}{rgb}{0.890196078431372,0.466666666666667,0.76078431372549}
\definecolor{color7}{rgb}{0.737254901960784,0.741176470588235,0.133333333333333}
\definecolor{color8}{rgb}{0.0901960784313725,0.745098039215686,0.811764705882353}

\begin{axis}[
legend cell align={left},
legend style={draw=white!80.0!black},
tick align=outside,
tick pos=both,
x grid style={white!69.01960784313725!black},
xmin=-1.28841079764563, xmax=25.2884107976456,
xtick style={color=black},
y grid style={white!69.01960784313725!black},
ymin=-0.1, ymax=1.1,
ytick style={color=black}
]
\addplot [only marks, mark=x,  mark=x, draw=color0, fill=color0]
table{%
x                      y
0 0
1 0
2 0
3 0
4 0
5 0
6 0
7 0
8 0
9 0
10 0
11 0
12 0
13 0
14 0
15 0
16 0
17 0
18 0
19 0
20 0
21 0
22 0
23 0
24 0
};
\addlegendentry{individual 1}
\addplot [only marks, mark=x,  draw=color1, fill=color1]
table{%
x                      y
0 0
1 0
2 0
3 0
4 0
5 0
6 0
7 0
8 0
9 0
10 0
11 0
12 0
13 0
14 0
15 0
16 0
17 0
18 0
19 0
20 0
21 0
22 0
23 0
24 0
};
\addlegendentry{individual 2}
\addplot [only marks, mark=x,  draw=color2, fill=color2]
table{%
x                      y
0 0
1 0
2 0
3 0
4 0
5 0
6 0
7 0
8 0
9 0
10 0
11 0
12 0
13 0
14 0
15 0
16 0
17 0
18 0
19 0
20 0
21 0
22 0
23 0
24 0
};
\addlegendentry{individual 3}
\addplot [only marks, mark=x,  draw=color3, fill=color3]
table{%
x                      y
0 0
1 0
2 0
3 0
4 0
5 0
6 0
7 0
8 0
9 0
10 0
11 0
12 0
13 0
14 0
15 0
16 0
17 0
18 0
19 0
20 0
21 0
22 0
23 0
24 0
};
\addlegendentry{individual 4}
\addplot [only marks, mark=x,  draw=color4, fill=color4]
table{%
x                      y
0 0
1 0
2 0
3 0
4 0
5 0
6 0
7 0
8 0
9 0
10 0
11 0
12 0
13 0
14 0
15 0
16 0
17 0
18 0
19 0
20 0
21 0
22 0
23 0
24 0
};
\addlegendentry{individual 5}
\addplot [only marks, mark=x,  draw=color5, fill=color5]
table{%
x                      y
0 0
1 0
2 0
3 0
4 0
5 0
6 0
7 0
8 0
9 0
10 0
11 0
12 0
13 0
14 0
15 0
16 0
17 0
18 0
19 0
20 0
21 0
22 0
23 0
24 0
};
\addlegendentry{individual 6}
\addplot [only marks, mark=x,  draw=color6, fill=color6]
table{%
x                      y
0 0
1 0
2 0
3 0
4 0
5 0
6 0
7 0
8 0
9 0
10 0
11 0
12 0
13 0
14 0
15 0
16 0
17 0
18 0
19 0
20 0
21 0
22 0
23 0
24 0
};
\addlegendentry{individual 7}
\addplot [only marks, mark=x,  draw=white!49.80392156862745!black, fill=white!49.80392156862745!black]
table{%
x                      y
0 0
1 0
2 0
3 0
4 0
5 0
6 0
7 0
8 0
9 0
10 0
11 0
12 0
13 0
14 0
15 0
16 0
17 0
18 0
19 0
20 0
21 0
22 0
23 0
24 0
};
\addlegendentry{individual 8}
\addplot [only marks, mark=x,  draw=color7, fill=color7]
table{%
x                      y
0 0
1 0
2 0
3 0
4 0
5 0
6 0
7 0
8 0
9 0
10 0
11 0
12 0
13 0
14 0
15 0
16 0
17 0
18 0
19 0
20 0
21 0
22 0
23 0
24 0
};
\addlegendentry{individual 9}
\addplot [only marks, mark=x,  draw=color8, fill=color8]
table{%
x                      y
0 0
1 0
2 0
3 0
4 0
5 0
6 0
7 0
8 0
9 0
10 0
11 0
12 0
13 0
14 0
15 0
16 0
17 0
18 0
19 0
20 0
21 0
22 0
23 0
24 0
};
\addlegendentry{individual 10}
\end{axis}

\end{tikzpicture}

%% file: mytikz_creation_t_2500.tex
\begin{tikzpicture}

\definecolor{color0}{rgb}{0.12156862745098,0.466666666666667,0.705882352941177}
\definecolor{color1}{rgb}{1,0.498039215686275,0.0549019607843137}
\definecolor{color2}{rgb}{0.172549019607843,0.627450980392157,0.172549019607843}
\definecolor{color3}{rgb}{0.83921568627451,0.152941176470588,0.156862745098039}
\definecolor{color4}{rgb}{0.580392156862745,0.403921568627451,0.741176470588235}
\definecolor{color5}{rgb}{0.549019607843137,0.337254901960784,0.294117647058824}
\definecolor{color6}{rgb}{0.890196078431372,0.466666666666667,0.76078431372549}
\definecolor{color7}{rgb}{0.737254901960784,0.741176470588235,0.133333333333333}
\definecolor{color8}{rgb}{0.0901960784313725,0.745098039215686,0.811764705882353}

\begin{axis}[
tick align=outside,
tick pos=both,
x grid style={white!69.01960784313725!black},
xmin=-1.28841079764563, xmax=25.2884107976456,
xtick style={color=black},
y grid style={white!69.01960784313725!black},
ymin=-0.1, ymax=1.1,
ytick style={color=black}
]
\addplot [only marks, mark=x,  draw=color0, fill=color0]
table{%
x                      y
0 0.302370505195585
1 0.337477921408621
2 0.268399995443248
3 0.268867656800274
4 0.285295420721746
5 0.262482644192981
6 0.326616960881526
7 0.277340766901799
8 0.275098995850919
9 0.302651421679462
10 0.225925461659429
11 0.268862577679877
12 0.340286582997736
13 0.287916711383383
14 0.29204966012937
15 0.299968983298272
16 0.288256402764058
17 0.27748212880294
18 0.307227498093394
19 0.272062821109299
20 0.330991293328025
21 0.32738951547619
22 0.32455751121748
23 0.298138379952352
24 0.282344379826703
};
\addplot [only marks, mark=x,  draw=color1, fill=color1]
table{%
x                      y
0 0.301627883156114
1 0.593752275304433
2 0.268399653289341
3 0.268836070659378
4 0.285295421195511
5 0.262491848631964
6 0.32668385860636
7 0.277101330409858
8 0.275098995850919
9 0.303017488671179
10 0.307103726076183
11 0.269076426865752
12 0.310300379608422
13 0.287916723434364
14 0.292079782541436
15 0.299968983298272
16 0.288256432890574
17 0.277561611267872
18 0.306918501055446
19 0.272056745791085
20 0.331204129982204
21 0.327202784222468
22 0.331359373481582
23 1.1118280314257
24 0.282344379826703
};
\addplot [only marks, mark=x,  draw=color2, fill=color2]
table{%
x                      y
0 0.303855749274529
1 0.337472101753332
2 0.268399653289341
3 0.170708009889987
4 0.28529541932278
5 0.280657226282614
6 0.326550063156692
7 0.27746048514777
8 0.271062614564751
9 0.303549489154736
10 0.307114400574423
11 0.269627633861857
12 0.340286582997736
13 0.287916712278046
14 0.291575756855892
15 0.299968983298272
16 0.288256331463775
17 0.277561611267872
18 0.30662387908204
19 0.272056610046811
20 0.331575956933829
21 0.320121464377902
22 0.334785556865433
23 0.298138386380289
24 0.282344379826703
};
\addplot [only marks, mark=x,  draw=color3, fill=color3]
table{%
x                      y
0 0.301627883156114
1 0.33747863526997
2 0.268399961227857
3 0.268838090409699
4 0.285295420721746
5 0.358270293863809
6 0.326894399243727
7 0.277266833692062
8 0.272071707392516
9 0.302379566613371
10 0.307121815464507
11 0.267903001055364
12 0.282800282101978
13 0.28791666547701
14 0.292006050318844
15 0.299968983298272
16 0.288256355086087
17 0.277561611267872
18 0.306964014665417
19 0.272067122557064
20 0.345159747710687
21 0.326767073938752
22 0.334778692175624
23 0.298138386750205
24 0.282344379826703
};
\addplot [only marks, mark=x,  draw=color4, fill=color4]
table{%
x                      y
0 0.303113127235057
1 0.337468006309475
2 0.268400063874029
3 0.268817546721309
4 0.285295416258888
5 0.262481879076025
6 0.326550063156692
7 0.277266833692062
8 0.275099000044466
9 0.302950777504554
10 0.307118849508473
11 0.268980497413904
12 0.2965503308552
13 0.287916712278046
14 0.291954343069696
15 0.299968983298272
16 0.288256351062866
17 0.27748212880294
18 0.307105332328024
19 0.272054358410618
20 0.330089603328412
21 0.321691963191408
22 0.334725735997092
23 0.298138386368792
24 0.282344379826702
};
\addplot [only marks, mark=x,  draw=color5, fill=color5]
table{%
x                      y
0 0.302370505195585
1 0.337471423917961
2 0.268399892797076
3 0.26885826083924
4 0.285295417210051
5 0.26247114440313
6 0.326694177579464
7 0.277352738726396
8 0.272071714476844
9 0.302282038147581
10 0.3071154988612
11 0.268914709206927
12 0.29157811908946
13 0.287916693194627
14 0.292024783742547
15 0.299968983298272
16 0.288256368604381
17 0.277323163873076
18 0.307016411684807
19 0.272067122557064
20 0.339033523732889
21 0.321681590574545
22 0.334699257907826
23 0.298138385605966
24 0.282344379826702
};
\addplot [only marks, mark=x,  draw=color6, fill=color6]
table{%
x                      y
0 0.301256572136378
1 0.33747863526997
2 0.268399653289341
3 0.268877294502632
4 0.285295418864666
5 0.262487246412472
6 0.326701498890024
7 0.277221048655829
8 0.262989843605323
9 0.30277307808088
10 0.307121815464507
11 0.26912988916222
12 0.340286582997736
13 0.28791675619622
14 0.292002445902183
15 0.299968983298272
16 0.288256447583294
17 0.277442387570474
18 0.307045448605341
19 0.272061749567773
20 0.351215873512999
21 0.321567475688926
22 0.334799286245053
23 0.298138385750064
24 0.282344379826702
};
\addplot [only marks, mark=x,  draw=white!49.80392156862745!black, fill=white!49.80392156862745!black]
table{%
x                      y
0 0.30355870045874
1 0.199885237829513
2 0.268399892797076
3 0.268817546721309
4 0.285295421195511
5 0.262489356540051
6 0.326838292002237
7 0.277331384177298
8 0.272071696114642
9 0.303549489154736
10 0.30711727843482
11 0.26838278936762
12 0.276910657622606
13 0.287916746231075
14 0.291575756855892
15 0.299968983298272
16 0.288256427838566
17 0.277561611267872
18 0.306925681159975
19 0.272047588251673
20 0.331451536986512
21 0.321757645288937
22 0.334732600686901
23 0.363530681616195
24 0.282344379826703
};
\addplot [only marks, mark=x,  draw=color7, fill=color7]
table{%
x                      y
0 0.299400017037699
1 0.337476762303165
2 0.268399919077096
3 0.268827827307731
4 0.285295413569161
5 0.262492712713586
6 0.326644609152609
7 0.277101330409858
8 0.275099014213122
9 0.302814038631138
10 0.307121815464507
11 0.268862577679877
12 0.275341964453368
13 0.250134019520407
14 0.29174099854655
15 0.299968983298272
16 0.288256463040226
17 0.27748212880294
18 0.30718916953978
19 0.272055718384019
20 0.330089603328412
21 0.321629719440167
22 0.334792421555243
23 0.298138385894162
24 0.282344379826703
};
\addplot [only marks, mark=x,  draw=color8, fill=color8]
table{%
x                      y
0 0.299400017037699
1 0.337475243416004
2 0.268399961227857
3 0.268817546721309
4 0.285295418201643
5 0.262482360599169
6 0.326674676890772
7 0.277101330409858
8 0.275099014213122
9 0.302814038631138
10 0.307121815464507
11 0.268897332031243
12 0.307154514056116
13 0.287916601002639
14 0.292084055495216
15 0.299968983298272
16 0.495779981280438
17 0.277561611267872
18 0.307356843963292
19 0.272062051607935
20 0.331575956933829
21 0.319014359377444
22 0.334699257907826
23 0.205233539474156
24 0.282344379826703
};
\end{axis}

\end{tikzpicture}

%% file: mytikz_creation_t_5000.tex
\begin{tikzpicture}

\definecolor{color0}{rgb}{0.12156862745098,0.466666666666667,0.705882352941177}
\definecolor{color1}{rgb}{1,0.498039215686275,0.0549019607843137}
\definecolor{color2}{rgb}{0.172549019607843,0.627450980392157,0.172549019607843}
\definecolor{color3}{rgb}{0.83921568627451,0.152941176470588,0.156862745098039}
\definecolor{color4}{rgb}{0.580392156862745,0.403921568627451,0.741176470588235}
\definecolor{color5}{rgb}{0.549019607843137,0.337254901960784,0.294117647058824}
\definecolor{color6}{rgb}{0.890196078431372,0.466666666666667,0.76078431372549}
\definecolor{color7}{rgb}{0.737254901960784,0.741176470588235,0.133333333333333}
\definecolor{color8}{rgb}{0.0901960784313725,0.745098039215686,0.811764705882353}

\begin{axis}[
tick align=outside,
tick pos=both,
x grid style={white!69.01960784313725!black},
xmin=-1.28841079764563, xmax=25.2884107976456,
xtick style={color=black},
y grid style={white!69.01960784313725!black},
ymin=-0.1, ymax=1.1,
ytick style={color=black}
]
\addplot [only marks, mark=x,  draw=color0, fill=color0]
table{%
x                      y
0 0.505137233705389
1 0.548939891321602
2 0.527446818973448
3 0.580647041206634
4 0.562293950449495
5 -0.64598450189629
6 0.536184459795229
7 0.545500061643522
8 0.541697281552832
9 0.523852537231873
10 0.50945346996891
11 0.523015074995684
12 0.546871125900864
13 0.506028045207311
14 0.519714854005718
15 0.513006547656565
16 0.503481850349701
17 0.52924054449132
18 0.508551883921367
19 0.541785743661301
20 0.541649592769917
21 0.565405879514814
22 0.549558829486937
23 0.517105048055714
24 0.565710931839941
};
\addplot [only marks, mark=x,  draw=color1, fill=color1]
table{%
x                      y
0 0.50513723370225
1 0.549329910387172
2 0.527348602461404
3 -0.424250201413975
4 0.545564034055838
5 0.540415970341507
6 0.53617877891605
7 0.548051567677956
8 0.540450750777696
9 0.523852537231873
10 0.509453184397088
11 0.523192051114966
12 0.546565730899254
13 0.507672399837475
14 0.519559953625115
15 0.513061392082541
16 0.503325760154737
17 0.525813658543239
18 0.508552191441258
19 0.541793607658418
20 0.541649635770153
21 0.565405879514814
22 0.548800235467291
23 0.517105051271758
24 0.565711067396231
};
\addplot [only marks, mark=x,  draw=color2, fill=color2]
table{%
x                      y
0 0.505137233711667
1 0.549248928246648
2 0.527175031566302
3 0.531989079074867
4 0.541220341266057
5 0.540587121953525
6 0.536185812295543
7 0.548745954464468
8 0.540450750777696
9 0.523857638446352
10 0.50945346996891
11 0.521333684839315
12 0.546798683615175
13 0.506850222522393
14 0.519536718568025
15 0.512896858804614
16 0.50463870435697
17 0.52924054449132
18 0.508552191441258
19 0.541797539656977
20 0.541649682858817
21 0.565405879514814
22 0.54916085189633
23 0.51710504394451
24 0.565710942748178
};
\addplot [only marks, mark=x,  draw=color3, fill=color3]
table{%
x                      y
0 0.505137233617007
1 0.549284154059525
2 0.527446818973448
3 0.531989079074867
4 0.544116136459244
5 0.540587121953525
6 0.53618436197689
7 0.546775814660739
8 0.553322296672052
9 0.0576195599594125
10 0.438651467767992
11 0.523133059075205
12 0.546809126416185
13 0.505971987629797
14 0.519630765227676
15 0.513061392082541
16 0.518441099969608
17 0.529240544695367
18 0.508550314116147
19 0.541794397449318
20 0.541649538792827
21 0.565405879514814
22 0.549429419086726
23 0.517105055039391
24 0.565710988709849
};
\addplot [only marks, mark=x,  draw=color4, fill=color4]
table{%
x                      y
0 0.505137233664337
1 0.548939891321602
2 0.527305986915693
3 0.532038431701953
4 0.556009052750282
5 0.540587121953525
6 0.536184910628667
7 0.545500061643522
8 0.540450750777696
9 0.523852537231873
10 0.509453319371953
11 0.522749593874765
12 0.546821434606914
13 0.506999709229688
14 0.519916224500501
15 0.512979125443577
16 0.506356831115138
17 0.522386771945279
18 0.508551783336689
19 0.541785743661301
20 0.541649592769917
21 0.565405879514814
22 0.549760254879734
23 0.517105046658979
24 0.565710975794619
};
\addplot [only marks, mark=x,  draw=color5, fill=color5]
table{%
x                      y
0 0.505137233711667
1 0.549248928246648
2 0.527446818973448
3 0.532008915634433
4 0.547564076078907
5 0.540550827188435
6 0.536183989832661
7 0.546775814660739
8 0.54813305450001
9 0.523862739660831
10 0.573539697022359
11 0.523428052110813
12 0.546809126416185
13 0.506028045207311
14 0.519559953625115
15 0.512942605206654
16 0.505019530575006
17 0.529240544775962
18 0.50855174951612
19 0.541785743661301
20 0.541649660778025
21 0.565405879514814
22 0.549362277289127
23 0.51710505614257
24 0.565710987376961
};
\addplot [only marks, mark=x,  draw=color6, fill=color6]
table{%
x                      y
0 0.505137233711667
1 0.549256374065172
2 0.527137041940844
3 0.565150799695453
4 0.541220341266057
5 0.540522024913708
6 0.536184385994575
7 -0.288888550093594
8 0.541697281552832
9 0.523862739660831
10 0.509453369570939
11 0.523192051114966
12 0.546821434606914
13 0.505542213196123
14 0.633560125087302
15 0.513140583333131
16 0.505019530575006
17 0.529240544937152
18 0.508551875395579
19 0.541787777073316
20 0.541649721770624
21 0.565405879514814
22 0.549956807077544
23 0.517105054671664
24 0.56571106077775
};
\addplot [only marks, mark=x,  draw=white!49.80392156862745!black, fill=white!49.80392156862745!black]
table{%
x                      y
0 0.505137233697547
1 0.548994557024166
2 0.527348602461404
3 0.532048588753567
4 0.551757145857776
5 0.540525225166455
6 0.570709668018987
7 0.505972162796703
8 0.549031781373933
9 0.523852537231873
10 0.50945313315663
11 0.525315930824747
12 0.546769931503487
13 0.506028045207311
14 0.519559953625115
15 0.512823818330768
16 0.507025481385203
17 0.529240544021816
18 0.508551826216485
19 0.747665106743646
20 0.541649611525821
21 0.565405879514814
22 0.549956807077544
23 0.517105053936211
24 0.565711132845651
};
\addplot [only marks, mark=x,  draw=color7, fill=color7]
table{%
x                      y
0 0.505137233711667
1 0.549248928246648
2 0.527235570886816
3 0.531972738868341
4 0.561036970909652
5 0.540550124432475
6 0.536186024369104
7 0.545274020531911
8 0.543975549322519
9 0.523859338851178
10 0.509453168774997
11 0.525315930824747
12 0.546781034311988
13 0.507672399837475
14 0.519559953625115
15 0.512903009581359
16 0.501754362192976
17 0.529240544775962
18 0.50855194333775
19 0.541788793779323
20 0.541649650684158
21 0.565405879514814
22 0.549956807077544
23 0.517105027749771
24 0.565710984297529
};
\addplot [only marks, mark=x,  draw=color8, fill=color8]
table{%
x                      y
0 0.505137233692834
1 0.549098132693387
2 0.527348602461404
3 0.532031899106041
4 0.551757145857776
5 0.540483398764971
6 0.536184910628667
7 0.546010362850409
8 0.757579096774679
9 0.523852537231873
10 0.50945346996891
11 0.52363453298277
12 0.546746788112427
13 0.506999709229688
14 0.519312113016151
15 0.513061392082541
16 0.501631989734468
17 0.522386772168195
18 0.508551775093481
19 0.541785743661301
20 0.541649721770624
21 0.565405879514814
22 0.54916085189633
23 0.517105041072038
24 0.565710975472888
};
\end{axis}

\end{tikzpicture}

%% file: mytikz_creation_t_7500.tex
\begin{tikzpicture}

\definecolor{color0}{rgb}{0.12156862745098,0.466666666666667,0.705882352941177}
\definecolor{color1}{rgb}{1,0.498039215686275,0.0549019607843137}
\definecolor{color2}{rgb}{0.172549019607843,0.627450980392157,0.172549019607843}
\definecolor{color3}{rgb}{0.83921568627451,0.152941176470588,0.156862745098039}
\definecolor{color4}{rgb}{0.580392156862745,0.403921568627451,0.741176470588235}
\definecolor{color5}{rgb}{0.549019607843137,0.337254901960784,0.294117647058824}
\definecolor{color6}{rgb}{0.890196078431372,0.466666666666667,0.76078431372549}
\definecolor{color7}{rgb}{0.737254901960784,0.741176470588235,0.133333333333333}
\definecolor{color8}{rgb}{0.0901960784313725,0.745098039215686,0.811764705882353}

\begin{axis}[
tick align=outside,
tick pos=both,
x grid style={white!69.01960784313725!black},
xmin=-1.28841079764563, xmax=25.2884107976456,
xtick style={color=black},
y grid style={white!69.01960784313725!black},
ymin=-0.1, ymax=1.1,
ytick style={color=black}
]
\addplot [only marks, mark=x,  draw=color0, fill=color0]
table{%
x                      y
0 0.670109318169486
1 0.637069087733836
2 0.721394789283164
3 0.669362470156643
4 0.641697848474105
5 0.604686601239122
6 0.648333091351768
7 0.665723807822452
8 0.6468311403259
9 0.654182447732562
10 0.658313764061192
11 0.672256370756405
12 0.673287006272307
13 0.628860861518286
14 0.628355940292043
15 0.673057211938799
16 0.640889871744551
17 0.676697772539842
18 0.636852784477109
19 0.618830959134273
20 0.604799893945734
21 0.617172845816442
22 0.663426242178518
23 0.624689770748298
24 0.62126371361181
};
\addplot [only marks, mark=x,  draw=color1, fill=color1]
table{%
x                      y
0 0.669042628992458
1 0.638107265428862
2 0.721394789283164
3 0.672257820167945
4 0.641697848474105
5 0.604686601239122
6 0.648333091351768
7 0.665723807822452
8 0.67440612095887
9 0.654249680271428
10 0.658313764061192
11 0.672927406008593
12 0.672255523922007
13 0.625532715357379
14 0.628355919755024
15 0.671026045311821
16 0.630474045475204
17 0.765764151691221
18 0.636841757377954
19 0.629553226325975
20 0.604798941913183
21 0.651393754644472
22 0.662857438047518
23 0.622836869614056
24 0.65250567255253
};
\addplot [only marks, mark=x,  draw=color2, fill=color2]
table{%
x                      y
0 0.669246586848179
1 0.637918505847976
2 0.693369312869409
3 0.669102614856946
4 0.641697848474105
5 0.604686601239122
6 0.648333091351768
7 0.665723807822452
8 0.674595891475973
9 0.654174977739782
10 0.658313764061192
11 0.671919984533471
12 0.672771265097157
13 0.624492137932274
14 0.628356953770072
15 0.673057211938799
16 0.637628791231568
17 0.676697772539842
18 0.636852784477109
19 0.642323376742235
20 0.604799178512104
21 0.633461991630024
22 0.665169236245041
23 0.624689770748298
24 0.655303346290499
};
\addplot [only marks, mark=x,  draw=color3, fill=color3]
table{%
x                      y
0 0.668898527151521
1 0.5271353848362
2 0.684027485001176
3 0.669362470156643
4 0.641697848474105
5 0.604686601239122
6 0.648333091351768
7 0.665723807822452
8 0.918821109175517
9 0.654249680416716
10 0.658313764061192
11 0.67091256305835
12 0.673287006272307
13 0.62550171518152
14 0.628356536343212
15 0.669117979692539
16 0.642600228430121
17 0.676697772539842
18 0.636845433077672
19 0.639500483012382
20 0.604798527210165
21 0.609026231897666
22 0.663434371439301
23 0.624226545464737
24 0.679721266968662
};
\addplot [only marks, mark=x,  draw=color4, fill=color4]
table{%
x                      y
0 0.668898527151521
1 0.637210657419501
2 0.684027485335959
3 0.671139682543296
4 0.641697848474105
5 0.604686601239122
6 0.648333091351768
7 0.665723807822452
8 0.674600399299155
9 0.654249680416716
10 0.658313764061192
11 0.67192319012268
12 0.672426604154398
13 0.625704786897661
14 0.628356118916352
15 0.66936418170793
16 0.644499456257103
17 0.676697772539842
18 0.636852784477109
19 0.627396425100702
20 0.604799047054942
21 0.633461991630024
22 0.663877233059116
23 0.62478241580501
24 0.655303346290499
};
\addplot [only marks, mark=x,  draw=color5, fill=color5]
table{%
x                      y
0 0.669352573783258
1 0.637918505847979
2 0.721394789283164
3 0.669791812403222
4 0.641697848474105
5 0.604686601239122
6 0.648333091351768
7 0.665723807822452
8 0.674586875829608
9 0.654234739850002
10 0.601388596703812
11 0.672423695271032
12 0.671996403095444
13 0.625255666146181
14 0.628357132394381
15 0.667148363569409
16 0.63657920592246
17 0.676697772539842
18 0.636841969735048
19 0.62963069161599
20 0.604800072453526
21 0.609028272909651
22 0.66169544200317
23 0.624380953892591
24 0.662510095270564
};
\addplot [only marks, mark=x,  draw=color6, fill=color6]
table{%
x                      y
0 0.668898527151521
1 0.638201645219308
2 0.68402748550369
3 0.670538755571653
4 0.641697848474105
5 0.604686601239122
6 0.648333091351768
7 0.665723807822452
8 0.67445243663235
9 0.654204859297726
10 0.658313764061192
11 0.670918105639928
12 0.67388995454897
13 0.627181288349903
14 0.62835799049155
15 0.49253808928142
16 0.645610739683951
17 0.676697772539842
18 0.636830730278799
19 0.627396425100702
20 0.604798981326361
21 0.651396223610583
22 0.665207869227358
23 0.622836869614056
24 0.63688469308217
};
\addplot [only marks, mark=x,  draw=white!49.80392156862745!black, fill=white!49.80392156862745!black]
table{%
x                      y
0 0.669045300997271
1 0.638201645219308
2 0.684027486160759
3 0.669133139809424
4 0.641697848474105
5 0.604686601239122
6 0.648333091351768
7 0.665723807822452
8 0.674521703902294
9 0.654249680416716
10 0.658313764061192
11 0.672927840307013
12 0.67303038592372
13 0.62505405651485
14 0.628358189812079
15 0.667886969615583
16 0.641875776152451
17 0.676697772539842
18 0.636852784477109
19 0.640275493517677
20 0.604798705314263
21 0.637434710620282
22 0.662579132927605
23 0.622836869614056
24 0.641420625484859
};
\addplot [only marks, mark=x,  draw=color7, fill=color7]
table{%
x                      y
0 0.669050645006896
1 0.640183620818833
2 0.684027485619776
3 0.677128785792435
4 0.641697848474105
5 0.604686601239122
6 0.648333091351768
7 0.665723807822452
8 0.674447928809168
9 0.654227269857221
10 0.658313764061192
11 0.671856234893596
12 0.672513811255912
13 0.625165189431771
14 0.628357043082226
15 0.678227454262016
16 0.648133522052076
17 0.676697772539842
18 0.636841757377954
19 0.630918481809013
20 0.604799178512104
21 0.633461991630024
22 0.665157042353867
23 0.625152996031859
24 0.665106878629449
};
\addplot [only marks, mark=x,  draw=color8, fill=color8]
table{%
x                      y
0 0.669420616696508
1 0.708527790700875
2 0.690255369279807
3 0.669232542506794
4 0.641697848474105
5 0.604686601239122
6 0.648333091351768
7 0.665723807822452
8 0.674408595842185
9 0.654234739946861
10 0.658313764061192
11 0.672928274605433
12 0.673487989031194
13 0.62775147946465
14 0.628357788623792
15 0.670841393800278
16 0.643388172830254
17 0.676697772539842
18 0.636814720522801
19 0.640275493517677
20 0.604799770083029
21 0.633461991630024
22 0.661118508611387
23 0.624689770748298
24 0.660708408025547
};
\end{axis}

\end{tikzpicture}

%% file: mytikz_creation_t_10000.tex
\begin{tikzpicture}

\definecolor{color0}{rgb}{0.12156862745098,0.466666666666667,0.705882352941177}
\definecolor{color1}{rgb}{1,0.498039215686275,0.0549019607843137}
\definecolor{color2}{rgb}{0.172549019607843,0.627450980392157,0.172549019607843}
\definecolor{color3}{rgb}{0.83921568627451,0.152941176470588,0.156862745098039}
\definecolor{color4}{rgb}{0.580392156862745,0.403921568627451,0.741176470588235}
\definecolor{color5}{rgb}{0.549019607843137,0.337254901960784,0.294117647058824}
\definecolor{color6}{rgb}{0.890196078431372,0.466666666666667,0.76078431372549}
\definecolor{color7}{rgb}{0.737254901960784,0.741176470588235,0.133333333333333}
\definecolor{color8}{rgb}{0.0901960784313725,0.745098039215686,0.811764705882353}

\begin{axis}[
tick align=outside,
tick pos=both,
x grid style={white!69.01960784313725!black},
xmin=-1.28841079764563, xmax=25.2884107976456,
xtick style={color=black},
y grid style={white!69.01960784313725!black},
ymin=-0.1, ymax=1.1,
ytick style={color=black}
]
\addplot [only marks, mark=x,  draw=color0, fill=color0]
table{%
x                      y
0 0.732899323862076
1 0.733108955882305
2 0.747437776876372
3 0.742534858033335
4 0.737928532124705
5 0.740701208495008
6 0.727148052172403
7 0.751072142451243
8 0.70995071443788
9 0.753409357072323
10 0.716158947497578
11 0.699511147061984
12 0.750601149148748
13 0.747109217908695
14 0.71183160248047
15 0.733654698149134
16 0.720931824174486
17 0.752951144448069
18 0.743400489784845
19 0.744411654820408
20 0.725344196646716
21 0.771422812419652
22 0.718065903041676
23 0.712996758961224
24 0.736039433981727
};
\addplot [only marks, mark=x,  draw=color1, fill=color1]
table{%
x                      y
0 0.732899323862076
1 0.732413734110335
2 0.74698123574197
3 1.03355560470872
4 0.737928360922365
5 0.740701208495008
6 0.727148052172403
7 0.751085387879173
8 0.70995071443788
9 0.753409357072323
10 0.716158931999795
11 0.716379509601455
12 0.750600234288706
13 0.740681314659474
14 0.712016527418032
15 0.733654698149134
16 0.721443064504271
17 0.752951144447802
18 0.743400489779174
19 0.744409941014011
20 0.725344196646716
21 0.735586841375797
22 0.717649143058599
23 0.724861929895899
24 0.736039433981727
};
\addplot [only marks, mark=x,  draw=color2, fill=color2]
table{%
x                      y
0 0.732899323862076
1 0.733340696472962
2 0.74698123574197
3 0.742534868109283
4 0.737928703327044
5 0.740701208495008
6 0.727929599088482
7 0.751058897023314
8 0.70995071443788
9 0.753409238418748
10 0.716158931999795
11 0.716379509601455
12 0.75059688722679
13 0.749251852325102
14 0.711873252241182
15 0.733654698149134
16 0.722109127438473
17 0.464546656968964
18 0.743400489779174
19 0.744406513401219
20 0.525713507066109
21 0.735586840606857
22 0.717974688379952
23 0.707266286919423
24 0.736039433981727
};
\addplot [only marks, mark=x,  draw=color3, fill=color3]
table{%
x                      y
0 0.732899323862076
1 1.02226318542016
2 0.746497057047502
3 0.742534868109283
4 0.737928446523535
5 0.740701208495008
6 0.727408567811096
7 0.750901741706379
8 0.721832591316582
9 0.753409238418748
10 0.716158908753122
11 0.691076965792248
12 0.750601149148748
13 0.744966583492288
14 0.711884081178967
15 0.733654698149134
16 0.720420583844702
17 0.877419493552717
18 0.743400489790515
19 0.744411654820408
20 0.725344636795953
21 0.735586839068977
22 0.718065903041676
23 0.824304022797082
24 0.73691379783175
};
\addplot [only marks, mark=x,  draw=color4, fill=color4]
table{%
x                      y
0 0.732899323862076
1 0.732413734110335
2 0.747437776876372
3 0.742534865422364
4 0.737929473737572
5 0.740701208495008
6 0.55526084680333
7 0.751058897023314
8 0.70995071443788
9 0.900715278595601
10 0.716158978493143
11 0.716379509601455
12 0.750589479070181
13 0.81079832738422
14 0.711881582193325
15 0.733656828161224
16 0.722775190372676
17 0.752951144448516
18 0.743400489790515
19 0.744411654820408
20 0.725298127271918
21 0.735586839068977
22 0.717805363498352
23 0.712996758961224
24 0.738662525531797
};
\addplot [only marks, mark=x,  draw=color5, fill=color5]
table{%
x                      y
0 0.732899323862076
1 0.732413734110335
2 0.74698123574197
3 0.742534868109283
4 0.737930244148099
5 0.740701208495008
6 0.727148052172403
7 0.751085387879173
8 0.70995071443788
9 0.753409086570423
10 0.716158978493143
11 0.716379509601455
12 0.750592625304833
13 0.74639500643656
14 0.71183160248047
15 0.733655763155179
16 0.720420583844702
17 0.752951144448516
18 0.743400489779174
19 0.744411654820408
20 0.725345517094425
21 0.735586841375797
22 0.717856748403329
23 0.712996758961224
24 0.736039433981727
};
\addplot [only marks, mark=x,  draw=color6, fill=color6]
table{%
x                      y
0 0.732899323862076
1 0.733804177654275
2 0.746054334693133
3 0.742534864963361
4 0.737928189720026
5 0.740701208495008
6 0.73396097848168
7 0.751058897023314
8 0.70995071443788
9 0.753409238418748
10 0.716158978493143
11 0.691076965792248
12 0.750601149148748
13 0.749251852325102
14 0.71183160248047
15 0.733664283203539
16 0.721443064504271
17 0.752951144448516
18 0.743400489779174
19 0.744406513401219
20 0.725344856870571
21 0.735586839068977
22 0.718065903041676
23 0.69517277824295
24 0.736039433981727
};
\addplot [only marks, mark=x,  draw=white!49.80392156862745!black, fill=white!49.80392156862745!black]
table{%
x                      y
0 0.732899323862076
1 0.732413734110335
2 0.747157806531573
3 0.742534855133688
4 0.737928703327044
5 0.740701208495008
6 0.727148052172403
7 0.75118490744161
8 0.696793492884378
9 0.753409238418748
10 0.717857519974047
11 0.716379509601455
12 0.75059516163279
13 0.740681314659474
14 0.71183160248047
15 0.733654698149134
16 0.721443064504271
17 0.752951144447802
18 0.522058645075289
19 0.744409084110813
20 0.725344196646716
21 0.735586840799092
22 0.718065903041676
23 0.712996758961224
24 0.736585911387991
};
\addplot [only marks, mark=x,  draw=color7, fill=color7]
table{%
x                      y
0 0.732899323862076
1 0.732413734110335
2 0.746515482087546
3 0.74253486488498
4 0.737928532124705
5 0.740701208495008
6 0.727148052172403
7 0.751085387879173
8 0.729686546768133
9 0.753409238418748
10 0.776495688081545
11 0.716359117772218
12 0.750595649558146
13 0.74068220432154
14 0.71183160248047
15 0.733654698149134
16 0.721187444339379
17 0.752951144447802
18 0.743400489790515
19 0.744424762871001
20 0.725344636795953
21 0.735586840222387
22 0.718065903041676
23 0.701535814877622
24 0.736039433981727
};
\addplot [only marks, mark=x,  draw=color8, fill=color8]
table{%
x                      y
0 0.732899323862076
1 0.732413734110335
2 0.746524694607568
3 0.742534868109283
4 0.737928189720026
5 0.740701208495008
6 0.727148052172403
7 0.751058897023314
8 0.622992765556764
9 0.753409416399111
10 0.716913903898735
11 0.716379509601455
12 0.750584101460918
13 0.740681314659474
14 0.711956551762607
15 0.733654698149134
16 0.721443064504271
17 0.752951144449943
18 0.743400489779174
19 0.744409084110813
20 0.725345517094425
21 0.735586839068977
22 0.718005093267194
23 0.903002055206534
24 0.737350979756762
};
\end{axis}

\end{tikzpicture}

%% file: mytikz_creation_t_12500.tex
\begin{tikzpicture}

\definecolor{color0}{rgb}{0.12156862745098,0.466666666666667,0.705882352941177}
\definecolor{color1}{rgb}{1,0.498039215686275,0.0549019607843137}
\definecolor{color2}{rgb}{0.172549019607843,0.627450980392157,0.172549019607843}
\definecolor{color3}{rgb}{0.83921568627451,0.152941176470588,0.156862745098039}
\definecolor{color4}{rgb}{0.580392156862745,0.403921568627451,0.741176470588235}
\definecolor{color5}{rgb}{0.549019607843137,0.337254901960784,0.294117647058824}
\definecolor{color6}{rgb}{0.890196078431372,0.466666666666667,0.76078431372549}
\definecolor{color7}{rgb}{0.737254901960784,0.741176470588235,0.133333333333333}
\definecolor{color8}{rgb}{0.0901960784313725,0.745098039215686,0.811764705882353}

\begin{axis}[
tick align=outside,
tick pos=both,
x grid style={white!69.01960784313725!black},
xmin=-1.28841079764563, xmax=25.2884107976456,
xtick style={color=black},
y grid style={white!69.01960784313725!black},
ymin=-0.1, ymax=1.1,
ytick style={color=black}
]
\addplot [only marks, mark=x,  draw=color0, fill=color0]
table{%
x                      y
0 0.87182383282512
1 0.83498267213499
2 0.85236555435441
3 0.856128406388783
4 0.839622189484368
5 0.843900491885511
6 0.876863806603875
7 0.855253820046327
8 0.874119891293588
9 0.830250063930492
10 0.889621343521379
11 0.913039454407375
12 0.874610336294447
13 0.837550177106623
14 0.849660387430847
15 0.843087267943272
16 0.854743449602045
17 0.847239392194624
18 0.875298599616667
19 0.832212242054068
20 0.851122957513504
21 0.854656579086878
22 0.832845763291454
23 0.866978075241345
24 0.841231407253797
};
\addplot [only marks, mark=x,  draw=color1, fill=color1]
table{%
x                      y
0 0.871826341141582
1 0.83498267213499
2 0.852063052521691
3 0.845008299318174
4 0.839622189484368
5 0.833091810437813
6 0.851472130016899
7 0.85524722905518
8 0.87413624975899
9 0.830256898311363
10 0.889543939918345
11 0.874581345812485
12 0.87647468791007
13 0.633601252543461
14 0.84966510386067
15 0.843177401458895
16 0.854743548890726
17 0.876366416290175
18 0.86310168774668
19 0.832036951979057
20 0.851122957513504
21 0.854656579086878
22 0.842114101666964
23 0.866978080018736
24 0.841231430513476
};
\addplot [only marks, mark=x,  draw=color2, fill=color2]
table{%
x                      y
0 0.871828725541487
1 0.83498267213499
2 0.891366205863556
3 0.856128406582757
4 0.839622189484368
5 0.837820612076008
6 0.875901658616454
7 0.855254679740825
8 0.87413624975899
9 0.830240381890924
10 0.889402033312781
11 0.874581345812485
12 0.873678160486636
13 0.837538812285338
14 0.84967832758616
15 0.843428735300536
16 0.854743548890726
17 0.876366468299824
18 0.863101687750768
19 0.831511081754025
20 0.851122957513504
21 0.854656588555673
22 0.791138233612116
23 0.866978082407432
24 0.841231427929067
};
\addplot [only marks, mark=x,  draw=color3, fill=color3]
table{%
x                      y
0 0.871819187941868
1 0.83498267213499
2 0.852085449626867
3 0.867248514235288
4 0.839622189484368
5 0.837820612076008
6 0.850028908035767
7 0.855254679740825
8 0.874128070526289
9 0.830256898311363
10 0.889483737115984
11 0.889977004569531
12 0.874610336294447
13 0.837572906749191
14 0.84967832758616
15 0.843312601732329
16 0.854743548890726
17 0.688760538832183
18 0.863101687748724
19 0.832212242054068
20 0.851122957513504
21 0.85465655778209
22 0.832845766087272
23 0.866978082407432
24 0.841231434131649
};
\addplot [only marks, mark=x,  draw=color4, fill=color4]
table{%
x                      y
0 0.871837045290839
1 0.83498267213499
2 0.852063052521691
3 0.845008298542277
4 0.839622189461947
5 0.827687497752583
6 0.876863806603875
7 0.8552632766858
8 0.999122807671152
9 0.830243229549621
10 0.889182723104184
11 0.863081159191287
12 0.872745984678824
13 0.837538812285338
14 0.849663670070335
15 0.843087267943272
16 0.854743548890726
17 0.876367404473511
18 0.977518660160019
19 0.831511081754025
20 0.851122957513504
21 0.854656096178354
22 0.823577424915945
23 0.866978075241345
24 0.84123142276025
};
\addplot [only marks, mark=x,  draw=color5, fill=color5]
table{%
x                      y
0 0.871828725541486
1 0.83498267213499
2 0.852115146928152
3 0.845008292917023
4 0.839622189484368
5 0.852007016990592
6 0.875420584622744
7 0.85524722905518
8 0.87412334950274
9 0.830237534232228
10 0.889552540318682
11 0.874581345812485
12 0.873678160486636
13 0.838612939177376
14 0.849641683796598
15 0.843087267943272
16 0.854743474424215
17 0.262029830115226
18 0.86310168774668
19 0.831978521954054
20 0.851122957513504
21 0.854656664306029
22 0.834640744050184
23 0.866978078502545
24 0.841231407253797
};
\addplot [only marks, mark=x,  draw=color6, fill=color6]
table{%
x                      y
0 0.871815146654922
1 0.83498267213499
2 0.852085449626867
3 0.845008299318174
4 0.839622189484368
5 0.829714115009545
6 0.84300823782124
7 0.855268434852785
8 0.87413624975899
9 0.830250063930492
10 0.889414933913287
11 0.84391418148929
12 0.872745984678824
13 0.837807344008348
14 0.849675300199252
15 0.843312601732329
16 0.854743548890726
17 0.876366156241928
18 0.887495511486653
19 0.83256282220409
20 0.934005750895416
21 0.854656777931564
22 0.836703254657254
23 0.866978075241345
24 0.841231417591432
};
\addplot [only marks, mark=x,  draw=white!49.80392156862745!black, fill=white!49.80392156862745!black]
table{%
x                      y
0 0.871825641211798
1 0.83498267213499
2 0.852063052521691
3 0.867248514235288
4 0.839622189484368
5 0.843900491885511
6 0.826080453429922
7 0.85524722905518
8 0.87412334950274
9 0.830243229549621
10 0.889621343521379
11 0.859247763650888
12 0.872745984678824
13 0.837807344008348
14 0.849660005691379
15 0.843267534974518
16 0.854743501454868
17 0.876366616527325
18 0.882616746738184
19 0.831861661904047
20 0.851122957513504
21 0.854656575028823
22 0.832845763291454
23 0.866978078824388
24 0.841231427929068
};
\addplot [only marks, mark=x,  draw=color7, fill=color7]
table{%
x                      y
0 0.871826669321695
1 0.83498267213499
2 0.852063052521691
3 0.845008299318174
4 0.839622189484368
5 0.84930494011546
6 0.876863806603875
7 0.85525754538915
8 0.874128070526289
9 0.830242470173968
10 0.88946653631531
11 0.874581345812485
12 0.873367435217365
13 0.860367192365612
14 0.84967832758616
15 0.843087267943272
16 0.854743548890726
17 0.876365844184033
18 0.871232962326671
19 0.83256282220409
20 0.851122957513504
21 0.854656564883686
22 0.81894325572819
23 0.866978075241345
24 0.84123142276025
};
\addplot [only marks, mark=x,  draw=color8, fill=color8]
table{%
x                      y
0 0.871819187941868
1 0.83498267213499
2 0.852063052521691
3 0.845008299318174
4 0.780142774492205
5 0.842009008155166
6 0.856550465334294
7 0.855249406947907
8 0.87413624975899
9 0.830240381890924
10 0.889414933913287
11 0.864358957704753
12 0.873678160486636
13 0.837561541927907
14 0.84967832758616
15 0.843461668700475
16 0.854743406583151
17 0.876365844184033
18 0.887495511486653
19 0.529664855114734
20 0.851122957513504
21 0.854656550680494
22 0.823577420256248
23 0.866978075241345
24 0.841231425861541
};
\end{axis}

\end{tikzpicture}

%% file: mytikz_creation_t_15000.tex
\begin{tikzpicture}

\definecolor{color0}{rgb}{0.12156862745098,0.466666666666667,0.705882352941177}
\definecolor{color1}{rgb}{1,0.498039215686275,0.0549019607843137}
\definecolor{color2}{rgb}{0.172549019607843,0.627450980392157,0.172549019607843}
\definecolor{color3}{rgb}{0.83921568627451,0.152941176470588,0.156862745098039}
\definecolor{color4}{rgb}{0.580392156862745,0.403921568627451,0.741176470588235}
\definecolor{color5}{rgb}{0.549019607843137,0.337254901960784,0.294117647058824}
\definecolor{color6}{rgb}{0.890196078431372,0.466666666666667,0.76078431372549}
\definecolor{color7}{rgb}{0.737254901960784,0.741176470588235,0.133333333333333}
\definecolor{color8}{rgb}{0.0901960784313725,0.745098039215686,0.811764705882353}

\begin{axis}[
tick align=outside,
tick pos=both,
x grid style={white!69.01960784313725!black},
xmin=-1.28841079764563, xmax=25.2884107976456,
xtick style={color=black},
y grid style={white!69.01960784313725!black},
ymin=-0.1, ymax=1.1,
ytick style={color=black}
]
\addplot [only marks, mark=x,  draw=color0, fill=color0]
table{%
x                      y
0 0.913911098762686
1 0.939938267906227
2 0.886120602424294
3 0.912033216240961
4 0.948407820613417
5 0.944376270812879
6 0.939957892955034
7 0.938236632153083
8 0.903021053807567
9 0.955215166890433
10 0.926733653401012
11 0.956344616216122
12 0.941142815448844
13 0.904339322596364
14 0.947434384262506
15 0.927897252206587
16 0.908458384874542
17 0.889288337776086
18 0.898557444809762
19 0.937109456389051
20 1.15199958805097
21 0.912690222050934
22 0.899893540805084
23 0.922111674536945
24 0.924690365574298
};
\addplot [only marks, mark=x,  draw=color1, fill=color1]
table{%
x                      y
0 0.913910808173684
1 0.939938267906227
2 0.887933214247657
3 0.912161734035258
4 0.94726346288539
5 0.938616586067112
6 0.939959600266588
7 0.93849452702916
8 0.903155628898581
9 0.955215166890433
10 0.926775335140503
11 0.727340862736732
12 0.941164694474453
13 0.910493800967836
14 0.946915893941311
15 0.932928367802027
16 0.908766683043363
17 0.889288008534609
18 0.898557444794175
19 0.938295421593727
20 0.879629995738521
21 0.912556536315152
22 0.900069505810337
23 0.922115564903593
24 0.92424907856975
};
\addplot [only marks, mark=x,  draw=color2, fill=color2]
table{%
x                      y
0 0.913912164255694
1 0.929957523594867
2 0.88430799060093
3 1.02744285443949
4 0.947749844443956
5 0.94833749585108
6 0.939959625010233
7 0.938067305101796
8 0.903078028848093
9 0.955215166890433
10 0.926747547314175
11 0.956344065578334
12 0.941111252264361
13 0.900464281259815
14 0.947434384262506
15 0.936268345336392
16 0.908561150930816
17 0.889288213229398
18 0.898557444750353
19 0.922538371059911
20 0.879629995738521
21 0.912799507608223
22 0.900323356965456
23 0.922116537495255
24 0.924800687325436
};
\addplot [only marks, mark=x,  draw=color3, fill=color3]
table{%
x                      y
0 0.913904705804634
1 0.934316498959879
2 0.887933214247657
3 0.912370575450991
4 0.947286207109235
5 0.944376270812879
6 0.939958152763314
7 0.939287068998261
8 0.902562767135076
9 0.955215166890433
10 0.927108789056435
11 0.956343239621653
12 0.941159314386188
13 0.885875887481946
14 0.947297755056245
15 0.936194322807764
16 0.909656285979883
17 0.889287994119483
18 1.17288579565071
19 0.919988973227759
20 0.879629995738521
21 0.912556536315152
22 0.899365645789327
23 0.922114592311931
24 0.924911009076573
};
\addplot [only marks, mark=x,  draw=color4, fill=color4]
table{%
x                      y
0 0.91390790228366
1 0.935406289352221
2 0.884049046054735
3 0.912033216240961
4 0.947626983680985
5 0.960928504047533
6 0.939959451804713
7 0.938250726531667
8 0.903233228949069
9 0.955215166890433
10 0.926942062098469
11 0.956344340897228
12 0.941111252264361
13 0.869463951628422
14 0.947066536399496
15 0.936226046748605
16 0.909256985611436
17 0.889288167869801
18 0.927501752729833
19 0.922538371059911
20 0.879629995738521
21 0.912494511914965
22 0.899746422522004
23 0.922116861692476
24 0.924886493131875
};
\addplot [only marks, mark=x,  draw=color5, fill=color5]
table{%
x                      y
0 0.913906836790651
1 0.940626835881609
2 0.88430799060093
3 0.912418769623853
4 0.947587672960699
5 0.950135955558645
6 0.939963608737191
7 0.938633283842135
8 0.375156306940363
9 0.955215166890433
10 0.926488194268451
11 0.956344685045845
12 0.941125599166399
13 0.910493800967836
14 0.947224185483643
15 0.936031473244783
16 0.908857685242989
17 0.889288144901701
18 0.898557444848749
19 0.922538371059911
20 0.879629995738521
21 0.912628813330699
22 0.900323356965456
23 0.922116537495255
24 0.924690365574298
};
\addplot [only marks, mark=x,  draw=color6, fill=color6]
table{%
x                      y
0 0.913912055284819
1 0.939938267906227
2 0.884204412782452
3 0.912418769623853
4 0.947184841444818
5 0.968697436445759
6 0.939961097257152
7 0.939287068998261
8 0.902511240963774
9 0.955215166890433
10 1.04474865380608
11 0.956341587708291
12 0.941161466421494
13 0.897403324805953
14 0.947224185483643
15 0.93631064392418
16 0.90945663579566
17 0.88928804409192
18 0.927501752856404
19 0.937911535309639
20 0.879629995738521
21 0.912628813330699
22 0.900323356965456
23 0.922112971325827
24 0.925131652578847
};
\addplot [only marks, mark=x,  draw=white!49.80392156862745!black, fill=white!49.80392156862745!black]
table{%
x                      y
0 0.913912164255694
1 0.930989956598139
2 0.883893679327018
3 0.912418769623853
4 0.946832692761146
5 0.948337495851079
6 0.939957892955034
7 0.938473893067285
8 0.903473891610835
9 0.955215166890433
10 0.926469669050899
11 0.956342964302759
12 0.941180117394143
13 0.898184844224891
14 0.947224185483643
15 0.936226046748605
16 0.908458384874542
17 0.889288078880424
18 0.898557444845961
19 0.929823913724481
20 0.879629995738521
21 0.912790024592805
22 0.8995618035001
23 0.922117510086918
24 0.924690365574298
};
\addplot [only marks, mark=x,  draw=color7, fill=color7]
table{%
x                      y
0 0.913902671681618
1 0.936152680227565
2 0.887933214247657
3 0.912418769623853
4 0.947536990128491
5 0.946321652549859
6 0.939962049887512
7 0.938473893067285
8 0.902511240963774
9 0.955215166890433
10 0.92669197166152
11 0.956345304513356
12 0.941111252264361
13 0.910493800967836
14 0.946880860811501
15 0.936226046748605
16 0.908612533958953
17 0.889288022949735
18 0.898557444798157
19 0.937109456389051
20 0.879629995738521
21 0.912710315738701
22 0.899942580232778
23 0.939249019976914
24 0.924968212947532
};
\addplot [only marks, mark=x,  draw=color8, fill=color8]
table{%
x                      y
0 0.913904705804634
1 0.942862936135759
2 0.885602713331904
3 0.912418769623853
4 0.946823595071608
5 0.944376270812879
6 0.9233222927603
7 0.937945024148549
8 0.903126285697861
9 0.955215166890433
10 0.926566926443046
11 0.956342413664972
12 0.941111252264361
13 0.905023155329662
14 0.947013986704781
15 0.919674204134037
16 0.909175268516757
17 0.88928813452281
18 0.898557444768282
19 0.923928235861213
20 0.879629995738521
21 0.912628813330699
22 0.900323356965456
23 0.922111674536945
24 1.08256485609128
};
\end{axis}

\end{tikzpicture}

%% file: mytikz_creation_t_25000.tex
\begin{tikzpicture}

\definecolor{color0}{rgb}{0.12156862745098,0.466666666666667,0.705882352941177}
\definecolor{color1}{rgb}{1,0.498039215686275,0.0549019607843137}
\definecolor{color2}{rgb}{0.172549019607843,0.627450980392157,0.172549019607843}
\definecolor{color3}{rgb}{0.83921568627451,0.152941176470588,0.156862745098039}
\definecolor{color4}{rgb}{0.580392156862745,0.403921568627451,0.741176470588235}
\definecolor{color5}{rgb}{0.549019607843137,0.337254901960784,0.294117647058824}
\definecolor{color6}{rgb}{0.890196078431372,0.466666666666667,0.76078431372549}
\definecolor{color7}{rgb}{0.737254901960784,0.741176470588235,0.133333333333333}
\definecolor{color8}{rgb}{0.0901960784313725,0.745098039215686,0.811764705882353}

\begin{axis}[
tick align=outside,
tick pos=both,
x grid style={white!69.01960784313725!black},
xmin=-1.28841079764563, xmax=25.2884107976456,
xtick style={color=black},
y grid style={white!69.01960784313725!black},
ymin=-0.1, ymax=1.1,
ytick style={color=black}
]
\addplot [only marks, mark=x,  draw=color0, fill=color0]
table{%
x                      y
0 1.00260404762773
1 0.996354343894981
2 0.993971791315378
3 0.98897211800504
4 1.00285996166489
5 0.988695839802607
6 1.24280055759885
7 0.997741167139969
8 0.982567934694853
9 1.0153077963947
10 0.969901020693593
11 1.01964158143743
12 0.957713783810821
13 0.965623286574657
14 0.818982150993833
15 1.00856600159819
16 0.979079254407037
17 1.02105153443835
18 1.00289798082872
19 1.00820088494936
20 0.998021934857392
21 0.990049782981175
22 0.990322464456794
23 1.02138339490623
24 1.01268498688276
};
\addplot [only marks, mark=x,  draw=color1, fill=color1]
table{%
x                      y
0 1.00418466444698
1 0.996637694226348
2 0.993971791183943
3 0.988798313020539
4 1.00283771933507
5 0.988695839802607
6 0.860354760249201
7 0.997741167139969
8 0.982172496386135
9 1.0153077963947
10 0.969901020693593
11 1.01983363466179
12 0.956369204080737
13 0.965623286574657
14 0.965590217707473
15 1.00847817379626
16 0.979079254407037
17 1.02105240799058
18 1.00289797579604
19 1.00980457099438
20 0.9751198252291
21 0.978037099814052
22 0.990385550135604
23 1.00857658236793
24 1.01268498723365
};
\addplot [only marks, mark=x,  draw=color2, fill=color2]
table{%
x                      y
0 1.00475867907547
1 0.996264893820757
2 0.993971791183943
3 0.988798313020539
4 1.0028204197452
5 0.988603295404418
6 1.02203579889208
7 0.997741167139969
8 0.982567934694853
9 1.01600476137462
10 0.969901020693593
11 1.01983363466179
12 0.955024624350653
13 0.988603234426578
14 0.916720862017544
15 1.00851570294575
16 0.979254386746004
17 1.02105122505526
18 1.00289797579604
19 1.00679071660887
20 0.990387898314628
21 0.978037099814052
22 0.990511721493226
23 1.02138339490623
24 1.0126849882863
};
\addplot [only marks, mark=x,  draw=color3, fill=color3]
table{%
x                      y
0 1.0060305908192
1 1.02535816379947
2 0.993971791242359
3 0.990737588636034
4 1.00006447317911
5 0.988664991669877
6 1.0111899828007
7 0.997741167139969
8 0.982567934694853
9 1.99789193794851
10 0.969620769561413
11 1.2895311820563
12 0.958639237223369
13 0.965623286574657
14 0.9655902175294
15 1.00861630025062
16 0.979254386746004
17 1.02105113406024
18 1.00289795933378
19 1.0089059691196
20 0.986570880043246
21 0.98524459931319
22 0.990217473396205
23 1.01556732019658
24 1.0126849882863
};
\addplot [only marks, mark=x,  draw=color4, fill=color4]
table{%
x                      y
0 1.00368590313213
1 0.995892093415167
2 0.993971791271567
3 1.08308134953655
4 1.00282635103315
5 0.988667799101553
6 1.01591315907298
7 0.997741167139969
8 0.981777058077416
9 1.0134492231149
10 0.969714186605473
11 1.01929763375773
12 0.956779891613357
13 0.988603525884063
14 0.965590217173255
15 1.00850032558658
16 0.979275374199812
17 1.02105240799058
18 1.00289797579604
19 1.01152931657049
20 0.951608992630739
21 0.999660150317145
22 0.990322464456794
23 1.01556732019658
24 1.01268498688276
};
\addplot [only marks, mark=x,  draw=color5, fill=color5]
table{%
x                      y
0 1.00476191044336
1 0.996264893820757
2 0.993971791340355
3 0.990737588636034
4 1.00284019070505
5 0.988646054442547
6 1.029033100176
7 0.997741167139969
8 0.981974777231776
9 1.01461083141477
10 0.968406347988631
11 1.01983363466179
12 0.992701342736444
13 0.942643371106902
14 0.965590217351328
15 1.00851570294575
16 0.979275374199812
17 1.02105122505526
18 1.00289797579604
19 1.01031613746009
20 0.986570880043246
21 0.999660150317145
22 0.990322464456794
23 1.16658284135672
24 1.01268498898807
};
\addplot [only marks, mark=x,  draw=color6, fill=color6]
table{%
x                      y
0 1.00348353596386
1 0.996451294023552
2 0.993971791183943
3 0.989013213836538
4 1.00281547700524
5 0.988695839802607
6 1.0198491428387
7 0.997741167139969
8 0.982567934694853
9 1.01623708303459
10 0.969901020693593
11 1.01983363466179
12 0.956163860314427
13 0.988603202042413
14 0.965590217707473
15 1.00856600159819
16 0.978904122068071
17 1.02105159813486
18 1.00289795694545
19 1.01031613746009
20 0.955763852913878
21 0.999660150317145
22 0.990511029583318
23 1.01585621590255
24 1.01268498688276
};
\addplot [only marks, mark=x,  draw=white!49.80392156862745!black, fill=white!49.80392156862745!black]
table{%
x                      y
0 1.00453522868854
1 0.996451294023552
2 0.99397179135919
3 0.988798313020539
4 1.00280064878535
5 0.988695839802607
6 1.0227355272231
7 0.997741167139969
8 0.982567934694853
9 1.02514055976549
10 0.969901020693593
11 1.01887336853998
12 0.956428838160816
13 0.98094324641191
14 0.965590217351328
15 1.00861630025062
16 0.979264880472908
17 1.02105240799058
18 0.984305564086486
19 1.01243138997081
20 0.924888852482887
21 0.970829416313021
22 0.990609342918961
23 1.00653523213102
24 1.01268498688276
};
\addplot [only marks, mark=x,  draw=color7, fill=color7]
table{%
x                      y
0 1.00453522868854
1 0.995892093415167
2 0.993971791280889
3 0.990737588636034
4 1.00281844264922
5 0.988646297869175
6 1.03153150680567
7 0.997741167139969
8 0.982172496386135
9 1.0153077963947
10 0.968281791929884
11 1.01983363466179
12 0.955024624350653
13 0.988603291098867
14 0.965590217440364
15 1.00851570294575
16 0.979254386746004
17 1.02105159813486
18 1.00289797579604
19 1.00608563243863
20 1.41304210202295
21 0.980439538313134
22 0.990706964344696
23 1.01941903866959
24 1.01268498968984
};
\addplot [only marks, mark=x,  draw=color8, fill=color8]
table{%
x                      y
0 1.00396706225321
1 0.996264893820757
2 0.993971791183943
3 0.990491519579135
4 1.00280064878535
5 0.988603295404418
6 1.0227355272231
7 0.997741167139969
8 0.982567934694853
9 1.01716636967449
10 0.969901020693593
11 1.01983363466179
12 0.955024624350653
13 0.942643371106902
14 0.965590217707473
15 1.00861630025062
16 1.01128949914603
17 1.02105240799058
18 1.00289794326038
19 1.01243138997081
20 0.9751198252291
21 0.98524459931319
22 0.99041709297501
23 1.02138339490623
24 1.01268498688276
};
\end{axis}

\end{tikzpicture}

%% file: mytikz_language_fig_1.tex
\begin{tikzpicture}

\definecolor{color0}{rgb}{0.12156862745098,0.466666666666667,0.705882352941177}

\begin{axis}[
tick align=outside,
tick pos=both,
x grid style={white!69.01960784313725!black},
xlabel={t},
xmin=-20, xmax=420,
xtick style={color=black},
y grid style={white!69.01960784313725!black},
ylabel={\(\displaystyle d(k^t,\MM_\mathcal{F})\)},
ymin=0.01, ymax=6.6407830863536,
ytick style={color=black},
ymode=log
]
\addplot [semithick, color0]
table {%
0 6.32455532033676
1 6.32455532033676
2 6.20375454714398
3 6.08242820190103
4 5.97231087767404
5 5.87023933024534
6 5.7621106522265
7 5.67823675130765
8 5.60459140405618
9 5.48903884194479
10 5.41346520987399
11 5.3537991500915
12 5.26556715859185
13 5.18106545652111
14 5.09571197750641
15 5.04071296492134
16 4.95413262155756
17 4.85710610123396
18 4.76589089480607
19 4.67841273445062
20 4.60489559124654
21 4.52588502908875
22 4.46144089527774
23 4.36656402012854
24 4.2921897152269
25 4.2174653271108
26 4.1503319893489
27 4.08196497627836
28 4.00381475606152
29 3.94961162669553
30 3.88283418619717
31 3.82248637339202
32 3.76866718527926
33 3.64576575589417
34 3.58207681277976
35 3.5064896059067
36 3.44785094340639
37 3.39486442323938
38 3.34257897295728
39 3.25817151387658
40 3.19206681306501
41 3.12945478094489
42 3.09565548278742
43 3.068466865655
44 3.00473804786331
45 2.95639152352361
46 2.91749715846042
47 2.87560317758665
48 2.79428068897611
49 2.75307885050631
50 2.68313698370123
51 2.63678731888001
52 2.61813085826003
53 2.585430540055
54 2.54277700094573
55 2.50486920857848
56 2.45907933349245
57 2.42398167877811
58 2.39992355086887
59 2.33851005317562
60 2.28725857569961
61 2.26647853904626
62 2.22282110510826
63 2.16696573921061
64 2.13698422615967
65 2.07581851614624
66 2.0528648377488
67 2.02789303664005
68 1.98673785245665
69 1.93798719668111
70 1.91741145553363
71 1.8864801401327
72 1.85589184403161
73 1.82870658773542
74 1.80592741906013
75 1.75924344382732
76 1.72228705018902
77 1.69029048485142
78 1.64760530312114
79 1.60575485650709
80 1.57983126503462
81 1.55890108931817
82 1.52743844561875
83 1.50459991224792
84 1.48045777662419
85 1.44917629275885
86 1.42112995270131
87 1.4124125280224
88 1.37149406713871
89 1.33422000103558
90 1.27889899275121
91 1.27304112837494
92 1.25207736520317
93 1.23676477509838
94 1.2035060592478
95 1.15873874341449
96 1.15045447216702
97 1.11330670574584
98 1.06770192642268
99 1.03805468822679
100 1.01487631577483
101 0.998519570870912
102 0.956158891095914
103 0.913798211320916
104 0.905579667169649
105 0.864934716147189
106 0.833618191171315
107 0.833618191171315
108 0.805333919923853
109 0.781191784300122
110 0.772973240148855
111 0.765652732073167
112 0.755234243172871
113 0.748731427774142
114 0.687414902798268
115 0.68505422302327
116 0.680912087399539
117 0.640912087399539
118 0.623591579323851
119 0.603591579323851
120 0.583591579323851
121 0.552275054347977
122 0.523931418404981
123 0.513931418404981
124 0.491570738629983
125 0.487428603006252
126 0.487428603006252
127 0.477428603006252
128 0.455067923231254
129 0.455067923231254
130 0.455067923231254
131 0.424886763397173
132 0.41246035652598
133 0.400997712826561
134 0.390997712826561
135 0.390997712826561
136 0.38685557720283
137 0.34685557720283
138 0.34685557720283
139 0.34685557720283
140 0.34685557720283
141 0.32685557720283
142 0.30685557720283
143 0.29685557720283
144 0.27685557720283
145 0.27685557720283
146 0.27685557720283
147 0.27685557720283
148 0.26685557720283
149 0.254721359549996
150 0.254721359549996
151 0.254721359549996
152 0.244721359549996
153 0.226502815398729
154 0.226502815398729
155 0.216502815398729
156 0.216502815398729
157 0.206502815398729
158 0.196502815398729
159 0.196502815398729
160 0.176502815398729
161 0.176502815398729
162 0.176502815398729
163 0.176502815398729
164 0.176502815398729
165 0.176502815398729
166 0.176502815398729
167 0.176502815398729
168 0.176502815398729
169 0.172360679774998
170 0.172360679774998
171 0.172360679774998
172 0.154142135623731
173 0.134142135623731
174 0.134142135623731
175 0.124142135623731
176 0.11
177 0.11
178 0.11
179 0.1
180 0.1
181 0.1
182 0.09
183 0.08
184 0.08
185 0.07
186 0.07
187 0.07
188 0.07
189 0.07
190 0.07
191 0.07
192 0.07
193 0.07
194 0.07
195 0.07
196 0.07
197 0.07
198 0.07
199 0.07
200 0.07
201 0.06
202 0.06
203 0.06
204 0.06
205 0.06
206 0.06
207 0.06
208 0.06
209 0.06
210 0.06
211 0.05
212 0.05
213 0.05
214 0.05
215 0.05
216 0.05
217 0.05
218 0.05
219 0.05
220 0.05
221 0.05
222 0.05
223 0.05
224 0.05
225 0.05
226 0.05
227 0.05
228 0.05
229 0.05
230 0.05
231 0.05
232 0.05
233 0.05
234 0.04
235 0.04
236 0.04
237 0.04
238 0.03
239 0.03
240 0.03
241 0.03
242 0.03
243 0.03
244 0.03
245 0.03
246 0.03
247 0.03
248 0.03
249 0.03
250 0.02
251 0.02
252 0.02
253 0.02
254 0.02
255 0.02
256 0.02
257 0.02
258 0.02
259 0
260 0
261 0
262 0
263 0
264 0
265 0
266 0
267 0
268 0
269 0
270 0
271 0
272 0
273 0
274 0
275 0
276 0
277 0
278 0
279 0
280 0
281 0
282 0
283 0
284 0
285 0
286 0
287 0
288 0
289 0
290 0
291 0
292 0
293 0
294 0
295 0
296 0
297 0
298 0
299 0
300 0
301 0
302 0
303 0
304 0
305 0
306 0
307 0
308 0
309 0
310 0
311 0
312 0
313 0
314 0
315 0
316 0
317 0
318 0
319 0
320 0
321 0
322 0
323 0
324 0
325 0
326 0
327 0
328 0
329 0
330 0
331 0
332 0
333 0
334 0
335 0
336 0
337 0
338 0
339 0
340 0
341 0
342 0
343 0
344 0
345 0
346 0
347 0
348 0
349 0
350 0
351 0
352 0
353 0
354 0
355 0
356 0
357 0
358 0
359 0
360 0
361 0
362 0
363 0
364 0
365 0
366 0
367 0
368 0
369 0
370 0
371 0
372 0
373 0
374 0
375 0
376 0
377 0
378 0
379 0
380 0
381 0
382 0
383 0
384 0
385 0
386 0
387 0
388 0
389 0
390 0
391 0
392 0
393 0
394 0
395 0
396 0
397 0
398 0
399 0
400 0
};
\end{axis}

\end{tikzpicture}